\documentclass[12pt]{amsart}
\usepackage[margin=.75in]{geometry}
\usepackage{amsmath}
\usepackage{amsxtra}
\usepackage{amscd}
\usepackage{amsthm}
\usepackage{amssymb}
\usepackage[normalem]{ulem}
\usepackage{comment}
\usepackage{color}
\usepackage{tikz}
\usetikzlibrary{matrix,arrows,decorations.pathmorphing,automata, matrix, positioning, calc, shapes.multipart}
\usepackage{graphicx}
\usepackage{hyperref}
\usepackage[shortlabels]{enumitem}
\usepackage{mynewcommands}
\usepackage{float}
\usepackage[vcentermath]{youngtab}
\usepackage{epigraph}

\usepackage{mynewcommands}

\numberwithin{equation}{section}

\theoremstyle{definition}
\newtheorem{theorem}[equation]{Theorem}
\newtheorem{lemma}[equation]{Lemma}
\newtheorem{proposition}[equation]{Proposition}
\newtheorem{corollary}[equation]{Corollary}

\newtheorem{definition}[equation]{Definition}
\newtheorem{notation}[equation]{Notation}
\newtheorem{example}[equation]{Example}

\newtheorem{remark}[equation]{Remark}

\newcommand\TikZ[1]{\begin{matrix}\begin{tikzpicture}#1\end{tikzpicture}\end{matrix}}

\newcommand{\hstar}{\mathfrak{h}^*}
\newcommand{\oh}{\mathcal{O}}
\newcommand{\el}{\mathrm{L}}

\newcommand{\emp}{\emptyset}
\newcommand{\bemp}{\bemptyset}

\newcommand{\slell}{\widehat{\mathfrak{sl}_\ell}}
\newcommand{\slinf}{\mathfrak{sl}_\infty}
\newcommand{\mbs}{\mathbf{s}}
\newcommand{\mbt}{\mathbf{t}}
\newcommand{\mbz}{\mathbf{z}}

\newcommand{\triv}{\mathrm{Triv}}

\newcommand{\mfh}{\mathfrak{h}}

\newcommand{\rk}{\operatorname{rk}}

\usepackage[foot]{amsaddr}

\title{The $\mathfrak{sl}_\infty$-crystal combinatorics of higher level Fock spaces}
\author{Thomas Gerber}
\address[T.G.]{Lehrstuhl D f\"ur Mathematik, RWTH Aachen, 52062 Aachen, Germany.}
\email{gerber@math.rwth-aachen.de}
\author{Emily Norton}
\address[E.N.]{Max Planck Institute for Mathematics, Vivatsgasse 7, 53111 Bonn, Germany.}
\email{enorton@mpim-bonn.mpg.de}

\begin{document}

\begin{abstract}
For integers $e,\ell\geq 2$, the level $\ell$ Fock space has an $\mathfrak{sl}_\infty$-crystal structure arising from 
the action of a Heisenberg algebra, 
intertwining the $\widehat{\mathfrak{sl}_e}$-crystal. 
The vertices of these crystals are charged $\ell$-partitions. 
We give the combinatorial rule for computing the arrows anywhere in the $\mathfrak{sl}_\infty$-crystal.
This allows us to pinpoint the location of any charged $\ell$-partition.
As an application, we compute the support of the spherical representation of a cyclotomic rational Cherednik algebra, and in particular, 
the set of parameters such that it is finite-dimensional. We also give an easy abacus characterization of all finite-dimensional representations of type $B$ Cherednik algebras.
\end{abstract}

\maketitle

\tableofcontents

\sloppy

\markleft{THOMAS GERBER AND EMILY NORTON}
\markright{COMBINATORICS OF THE $\slinf$-CRYSTAL}

\section*{Introduction}
\epigraph{...in einem blo\ss en Hin und Her, Vor und Zur\"{u}ck...}{Stefan Zweig}

The level $\ell$ Fock space,
whose standard basis consists of all $\ell$-partitions,
appears as the Grothendieck group of categories of representations of cyclotomic rational Cherednik algebras, Hecke and $q$-Schur algebras at roots of unity, and finite classical groups in non-defining characteristic. 
These identifications require a choice of parameters $\bs\in\Z^\ell$ and $e\in\Z_{\geq 2}$ for the Fock space, then denoted $\cF_\bs$.
There are two important crystal structures on $\mathcal{F}_\mbs$: the $\sle$-crystal and the $\slinf$-crystal,
both of which are represented by a graph whose vertices are the standard basis elements of $\cF_\bs$.
The $\sle$-crystal is categorified by parabolic branching rules in the respective module categories;
this is related to induction and restriction with respect to parabolics of the same type \cite{Ariki2007}, \cite{Shan2011}, \cite{DVV}.
Moreover, the rule for computing the $\sle$-crystal is known: the arrows are given by adding ``good" nodes of a given content modulo $e$ \cite{JMMO1991}, \cite{FLOTW1999}. 
The $\slinf$-crystal is categorified by the action of a Heisenberg algebra on the cyclotomic Cherednik category $\oh$ \cite{ShanVasserot2012},\cite{Losev2015}; 
this action is related to parabolic induction with respect to type $A$ parabolics \cite{ShanVasserot2012}.
Previous work by Losev \cite{Losev2015} and the first author \cite{Gerber2016} gave roundabout or indirect 
constructions of the $\slinf$-crystal, with explicit formulas only in special cases. 
This left open the problem of formulating a direct combinatorial rule for the arrows in the $\slinf$-crystal.

The motivation for studying the combinatorics of the $\slinf$-crystal is that the $\sle$- and $\slinf$-crystals taken in conjunction provide the answers to some important representation theoretic questions, 
in particular about the cyclotomic Cherednik category $\oh$. 
The simple objects $\el(\bla)$ of $\oh$ are parametrized by $\ell$-partitions $\bla$ \cite{GGOR}. 
The finite-dimensional simple modules (cuspidals) are exactly those $\el(\bla)$ such that $\bla$ is a source vertex for both the $\sle$- and $\slinf$-crystals \cite{ShanVasserot2012}, \cite{Losev2015}. 
More generally, the two crystals determine the Harish-Chandra series of $\oh$: the cuspidal support of $\el(\bla)$ can be read off from the location of $\bla$ in the $\sle$- and $\slinf$-crystals \cite{ShanVasserot2012}, \cite{Losev2015}, \cite{LosevSA}.

In this paper, we study the $\slinf$-crystal of higher level Fock spaces (so for $\ell\geq 2$) using combinatorics of $\ell$-abaci, which encode charged $\ell$-partitions using beads on $\ell$ runners.
Our first main result gives a simple rule for computing the incoming and outgoing arrows of the $\slinf$-crystal at a vertex $\cA$ for an arbitrary $\ell$-abacus $\cA$,
from which we derive a formula for the position of $\cA$ in its connected component.
In order to state our theorem, we adapt the notion of $e$-period from \cite{JaconLecouvey2012}, defining certain patterns of $e$ beads in the abacus called fore and aft periods (Definition \ref{defperiods}).

\medskip

\textbf{Theorem A.} 
\begin{enumerate}
\item[(a)] Let $\cA$ and $\cA'$ be $\ell$-abaci. There is an arrow $\cA\to\cA'$ in the $\slinf$-crystal if and only if for some $k\in\N$, $\cA'$ is obtained from $\cA$ by shifting the $k$'th fore period $P_k$ of $\cA$ one unit to the right, and the shift $\tilde{P_k}$ of $P_k$ is equal to $Q_k'$, the $k$'th aft period of $\cA'$.
\item[(b)] The position of a vertex $\cA$ in the $\slinf$-crystal is determined by a partition $\theta$ that can be read off directly from $\cA$.
\end{enumerate}

\medskip

Part (a) of the above statement is Theorem \ref{thmedge}, and provides the analogue for the $\slinf$-crystal to the celebrated $\sle$-crystal rule \cite{FLOTW1999}. 
It allows us to construct the entire connected component of the $\slinf$-crystal containing an arbitrary $\ell$-abacus $\cA$, completing the results of \cite[Section 5]{Gerber2016a}.
Part (b) is a reformulation of Theorem \ref{thmdepth}. This gives an easy way to compute the depth $q(\bla)$ of $\bla$ in the $\slinf$-crystal, which is equal to $|\theta|$
(where $\bla$ is the $\ell$-partition read off from $\cA$). The number $q(\bla)$ gives one part of the support of the simple module $\el(\bla)$ \cite{Losev2015}. 

In special cases, our results allow us to deduce closed formulas. Section \ref{widthe} presents a closed combinatorial formula for each $\ell$-partition in a certain connected component of the $\slinf$-crystal component when the charge $\mbs\in e\Z^\ell$ (this is the lattice $\mathfrak{c_\Z}$ of \cite[Section 1.2]{Losev2015} containing the origin). 
The connected component in question is the one containing the empty $\ell$-partition, and consists of all $\ell$-partitions which are simultaneously singular for the action of 
$\sle$ and of its level-rank dual (see \cite{Gerber2016}). 

\subsection*{Applications to representation theory of Cherednik algebras}

In Section \ref{Applications} we answer some classical questions about the representation theory of cyclotomic rational Cherednik algebras. 
Our first result on this topic consists of a formula for the set of charges $\mbs\in\Z^\ell$ such that the spherical representation $\el(\triv)$ of $H_c(G(\ell,1,n))$ has a given support. 
Fix $\mbs\in\Z^\ell$, normalized so that $s_1=n-e-1$ 
(without loss of generality).
Let $q(\triv)$ and $p(\triv))$ denote the depth of $\cA( \triv, \bs )$ in the $\slinf$- and $\sle$-crystals, respectively. Set $m=\min\{n \mod e,\; s_j \mod e\;|\;s_j\geq 0,\; 2\leq j\leq \ell \}.$ 
The following result is Theorem \ref{Depth of triv}.

\medskip

\textbf{Theorem B.} 
Write $n=qe+r$ with $q,r\in\N\cup\{0\}$ and $r<e$.  We have
\begin{align*}
q(\triv)&=\begin{cases} q\mbox{ if } s_j<0\mbox{ for all }j\geq 2 \\
0\mbox{ if }s_j\geq 0\mbox{ for some }j \geq 2
\end{cases}\\
p(\triv)&=\begin{cases} r\mbox{ if }s_j<0\mbox{ for all }j\geq 2 \\
m \mbox{ if }s_j\geq 0 \mbox{ for some }j\geq 2.
\end{cases}
\end{align*}

\medskip

Theorem B with $\ell=2$ overlaps with a result of \cite[Section 4.2]{Etingof2012}; Corollary \ref{cortrivfd} implies \cite[Corollary 5.4]{GGJL} and answers \cite[Question 5.5]{GGJL} in the affirmative, which is also answered by the forthcoming \cite{GriffethJuteau2017}.
Outside type $A$, finite-dimensional simples $\el(\bla)$ occur for $\bla$ such that $\dim\bla>1$, and it is an interesting and difficult problem to classify them. Theorem \ref{firstcompfd} replaces $\triv=((1^n),\emp,\dots,\emp)$ with $\bla:=(\lambda,\emp,\dots,\emp)$, $\lambda$ an arbitrary partition, and with similar formulas to Corollary \ref{cortrivfd} identifies the parameters such that $\dim\el(\bla)<\infty$.
Corollary \ref{2fd} classifies all finite-dimensional $H_c(B_n)$-modules by a pattern avoidance condition on abaci:

\medskip

\textbf{Theorem C.} A charged bipartition $|\bla,\mbs\rangle$ labels a finite-dimensional representation of $H_c(B_n)$ if and only if the abacus $\cA$ of $|\bla,\mbs\rangle$ avoids
the $e+1$ patterns of Theorem \ref{depth0bipartition}, and additionally, any bead in $\cA$ with a space directly to its left is either the last bead of a period, or sits above an empty space or the last bead of a period.

\medskip

\subsection*{Relation of our work to previous work on the $\slinf$-crystal}
We now discuss the history of the problem and previous work by other authors. A first method for computing the $\slinf$-crystal was found by Losev \cite{Losev2015}. He gives a formula for the action of the annihilation operators in the case of an asymptotic parameter (\cite[Proposition 1.1]{Losev2015}), then introduces wall-crossing functors which commute with the $\slinf$- and $\sle$-crystals in order to transfer the $\slinf$-crystal to other chambers of the parameter space. These functors induce a bijection on the set of $\ell$-partitions; concrete computations with the $\slinf$-crystal can then be performed given a precise understanding of the combinatorics of iterated wall-crossings. 
Jacon and Lecouvey recently found a combinatorial description of crossing a single wall \cite{JaconLecouvey2016},
however the composition of wall-crossings becomes complicated, 
and there is no known closed formula for an arbitrary composition of wall-crossings. 
The number of walls grows as the size of the $\ell$-partitions increases, but the number of asymptotic chambers stays fixed, 
so that the number of walls that must be crossed to compute the $\slinf$-crystal for all parameters and larger $\ell$-partitions increases. 

We use combinatorics of $e$-periods in abaci to study the $\slinf$-crystal, continuing the project begun in the first author's previous work \cite[Section 5]{Gerber2016a},
in which the action of the operators $\tilde{a}_\mu$ of \cite{ShanVasserot2012},\cite{Losev2015} was expressed in terms of adding ``good vertical $e$-strips''.
This yielded a way to compute the $\slinf$-crystal of the $\ell$ Fock space recursively, starting from the empty $\ell$-partition, 
and starting a new connected component at every $\ell$-partition with no incoming arrow.
However, the $\slinf$-crystal structure is not directly given by the operators $\tilde{a}_\mu$, which rather describe the action of the Heisenberg algebra at the crystal level. 
In particular, \cite{Gerber2016a} did not give a formula for the action of the creation and annihilation operators of $\slinf$. 
The questions of describing the edges in the $\slinf$-crystal, how to compute a path leading to the highest weight vertex, how to determine the position or depth of a multipartition in the $\slinf$-crystal, remained open. 
Likewise, the wall-crossing approach left it as an open question whether any direct combinatorial rule existed for computing the arrows in the $\slinf$-crystal or the depth of an $\ell$-partition in an arbitrary chamber \cite[Section 1.2]{Losev2015}. Our construction resolves these questions.

\section{Combinatorics of abaci and crystal structures}\label{abaci+crystals}
\subsection{Abaci}\label{abaci}

Fix $\ell\in\Z_{\geq2}$.
An \textit{
$\ell$-abacus} is a subset $\cA$ of $\Z \times \{ 1,\dots,\ell\}$
such that there exists $m_-,m_+\in\Z$ verifying:
\begin{itemize}
 \item For all $\be\leq m_-$ and for all $j\in\{ 1,\dots,\ell\}$, $(\be,j)\in\cA$.
 \item For all $\be\geq m_+$ and for all $j\in\{ 1,\dots,\ell\}$, $(\be,j)\notin\cA$.
\end{itemize}
The $\ell$-abacus $\cA$ is represented by $\ell$ rows of beads (numbered from $1$ at the bottom to $\ell$ at the top);
the position $(\mathrm{column},\mathrm{row})$ of a bead is given by $(\be,j)$,
and we will therefore write $b=(\be,j)$ for the bead $b$ of $\cA$ in position $(\be,j)$.
Accordingly, a pair $(\be,j)\in\Z \times \{ 1,\dots,\ell\}$ such that $(\be,j)\notin\cA$ will be called a \textit{space} of $\cA$.
We see that an abacus is infinitely full in the left direction and infinitely empty in the right direction.
For $k\in\N$ and $j\in\{ 1,\dots,\ell\}$, the \textit{$k$'th bead} in row $j$ of $\cA$, denoted $b_k^j$, is the $k$'th element of the
set $\cA\cap\{ (\be,j) \, ; \, \be\in\Z \}$ written in decreasing order (with respect to the natural order on $\Z\times\{j\}$ induced by $\geq$ on $\Z$).

We define the \textit{charge} of $\cA$ as the element $\bs=(s_1,\dots,s_\ell)$ of $\Z^\ell$ such that 
in the abacus obtained from $\cA$ by pushing all beads as far to the left as possible,
the rightmost bead in row $j$, say $b=(\be,j)$, verifies $\be=s_j$, for all $j\in\{1 \dots, \ell\}$.

Let $\bs=(s_1,\dots,s_\ell)\in\Z^\ell$.
The set of $\ell$-abaci with charge $\bs$ is in bijection with the set of $\ell$-partitions via the map
$\cA \mapsto \bla=( (\la^1_1,\la^1_2,\dots), \dots, (\la^\ell_1,\la^\ell_2,\dots) )$ defined by 
$$(\be,j) \mapsto \la_k^j=\be-s_j+k-1$$
for all $(\be,j)\in\cA$, where $k\in\Z_{\geq0}$ is such that $(\be,j)=b_k^j$.
We will sometimes also write $\be=\be(b)$ and $j=j(b)$ for a bead $b=(\be,j)$. 
These are the $\beta$-numbers associated to $\cA$ (or to $\bla$ and $\bs$),
see \cite{James1978}, \cite{JamesKerber1984}. 
In particular, we have $b_k^j=(\beta(b_k^j),j)$.

We write $|\bla,\bs\rangle$ for the data consisting of an element $\bs\in\Z^\ell$ and an $\ell$-partition $\bla$, 
and call it a charged $\ell$-partition.
Further, we write $\cA=\cA(\bla,\bs)$ for the corresponding $\ell$-abacus,
and will often identify $\cA$ with $|\bla,\bs\rangle$.

\begin{example}\label{abacusexample}
Let $\ell=3$, $\bla=((10,9,1) , (9^3,7,6,4,1), (6,3^2))$ and $\bs=(-4,0,-5)$. Then we have
$$ \begin{array}{c}
\cA(\bla,\bs)= \\ \vspace{0.4cm}
\end{array}
\TikZ{[scale=.5]
\draw
(13,2)node[]{3}
(13,1)node[]{2}
(13,0)node[]{1}
(11,2)node[fill,circle,inner sep=.5pt]{}
(10,2)node[fill,circle,inner sep=.5pt]{}
(9,2)node[fill,circle,inner sep=.5pt]{}
(8,2)node[fill,circle,inner sep=.5pt]{}
(7,2)node[fill,circle,inner sep=.5pt]{}
(6,2)node[fill,circle,inner sep=.5pt]{}
(5,2)node[fill,circle,inner sep=.5pt]{}
(4,2)node[fill,circle,inner sep=.5pt]{}
(3,2)node[fill,circle,inner sep=.5pt]{}
(2,2)node[fill,circle,inner sep=.5pt]{}
(1,2)node[fill,circle,inner sep=3pt]{}
(0,2)node[fill,circle,inner sep=.5pt]{}
(-1,2)node[fill,circle,inner sep=.5pt]{}
(-2,2)node[fill,circle,inner sep=.5pt]{}
(-3,2)node[fill,circle,inner sep=3pt]{}
(-4,2)node[fill,circle,inner sep=3pt]{}
(-5,2)node[fill,circle,inner sep=.5pt]{}
(-6,2)node[fill,circle,inner sep=.5pt]{}
(-7,2)node[fill,circle,inner sep=.5pt]{}
(-8,2)node[fill,circle,inner sep=3pt]{}
(-9,2)node[fill,circle,inner sep=3pt]{}
(11,1)node[fill,circle,inner sep=.5pt]{}
(10,1)node[fill,circle,inner sep=.5pt]{}
(9,1)node[fill,circle,inner sep=3pt]{}
(8,1)node[fill,circle,inner sep=3pt]{}
(7,1)node[fill,circle,inner sep=3pt]{}
(6,1)node[fill,circle,inner sep=.5pt]{}
(5,1)node[fill,circle,inner sep=.5pt]{}
(4,1)node[fill,circle,inner sep=3pt]{}
(3,1)node[fill,circle,inner sep=.5pt]{}
(2,1)node[fill,circle,inner sep=3pt]{}
(1,1)node[fill,circle,inner sep=.5pt]{}
(0,1)node[fill,circle,inner sep=.5pt]{}
(-1,1)node[fill,circle,inner sep=3pt]{}
(-2,1)node[fill,circle,inner sep=.5pt]{}
(-3,1)node[fill,circle,inner sep=.5pt]{}
(-4,1)node[fill,circle,inner sep=.5pt]{}
(-5,1)node[fill,circle,inner sep=3pt]{}
(-6,1)node[fill,circle,inner sep=.5pt]{}
(-7,1)node[fill,circle,inner sep=3pt]{}
(-8,1)node[fill,circle,inner sep=3pt]{}
(-9,1)node[fill,circle,inner sep=3pt]{}
(11,0)node[fill,circle,inner sep=.5pt]{}
(10,0)node[fill,circle,inner sep=.5pt]{}
(9,0)node[fill,circle,inner sep=.5pt]{}
(8,0)node[fill,circle,inner sep=.5pt]{}
(7,0)node[fill,circle,inner sep=.5pt]{}
(6,0)node[fill,circle,inner sep=3pt]{}
(5,0)node[fill,circle,inner sep=.5pt]{}
(4,0)node[fill,circle,inner sep=3pt]{}
(3,0)node[fill,circle,inner sep=.5pt]{}
(2,0)node[fill,circle,inner sep=.5pt]{}
(1,0)node[fill,circle,inner sep=.5pt]{}
(0,0)node[fill,circle,inner sep=.5pt]{}
(-1,0)node[fill,circle,inner sep=.5pt]{}
(-2,0)node[fill,circle,inner sep=.5pt]{}
(-3,0)node[fill,circle,inner sep=.5pt]{}
(-4,0)node[fill,circle,inner sep=.5pt]{}
(-5,0)node[fill,circle,inner sep=3pt]{}
(-6,0)node[fill,circle,inner sep=.5pt]{}
(-7,0)node[fill,circle,inner sep=3pt]{}
(-8,0)node[fill,circle,inner sep=3pt]{}
(-9,0)node[fill,circle,inner sep=3pt]{}
(11,-1.3)node[]{11}
(10,-1.3)node[]{10}
(9,-1.3)node[]{9}
(8,-1.3)node[]{8}
(7,-1.3)node[]{7}
(6,-1.3)node[]{6}
(5,-1.3)node[]{5}
(4,-1.3)node[]{4}
(3,-1.3)node[]{3}
(2,-1.3)node[]{2}
(1,-1.3)node[]{1}
(0,-1.3)node[]{0}
(-1,-1.3)node[]{-1}
(-2,-1.3)node[]{-2}
(-3,-1.3)node[]{-3}
(-4,-1.3)node[]{-4}
(-5,-1.3)node[]{-5}
(-6,-1.3)node[]{-6}
(-7,-1.3)node[]{-7}
(-8,-1.3)node[]{-8}
(-9,-1.3)node[]{-9}
;
}
$$
\end{example}

The elements $|\bla,\bs\rangle$, for $\bs$ fixed and $\bla$ varying, form the $\C$-basis of the so-called
\textit{level $\ell$ Fock space}, denoted $\cF_\bs$.

\subsection{$\sle$-crystal of the Fock space}\label{chapeaucrystal}

Let $e\in\Z_{\geq 2}$.
Then the Fock space has an $\sle$-crystal structure,
whose data is encoded in the $\sle$-crystal graph
(and as is classically done, we use the term ``crystal'' in place of ``crystal graph'').
This is usually achieved by replacing the ground field $\C$ by the field of rational functions $\C(v)$,
and by endowing the resulting space $\cF_\bs^v$
(called the $v$-deformed Fock space)
with the structure of an integrable $\Ue$-module \cite{JMMO1991}, \cite{GeckJacon2011},
which depends on the parameters $\bs$ and $e$.
By Kashiwara's theory of crystals for quantum groups \cite{Kashiwara1993}, \cite{HongKang2002},
the representation $\cF_\bs^v$ has a $\Ue$-crystal structure,
a combinatorial datum controlling the behaviour of the $\Ue$-module ``at $v=0$''.
Since the role of $v$ is irrelevant for our purpose, we will simply talk of the $\sle$-crystal for $\cF_\bs$. (This is also justified by the fact that the cyclotomic Cherenik category $\cO$ provides an $\sle$-categorification of $\cF_\bs$,
see Section \ref{shancat}.)

The $\sle$-crystal is a directed colored graph whose vertices are the standard basis elements $|\bla,\bs\rangle$ of $\cF_\bs$,
and whose arrows represent the action of the Kashiwara crystal operators $\tf_i$ (or equivalently $\te_i$ by reversing the arrows) 
for $i=0,\dots, e-1$.
Explicit formulas for computing it are known.
We briefly recall the rule for the arrows in the $\sle$-crystal in the language of abaci.

Firstly, this requires the notion of good left- and right-shiftable $i$-beads. 
Let $\cA$ be an $\ell$-abacus with charge $\bs$.
Let $b=(\be,j)\in\cA$, and let $i=\be\mod e$.
If $(\be+1,j)\notin\cA$ (respectively $(\be-1,j)\notin\cA$), then $b$ is called a \textit{right-shiftable $i$-bead} 
(respectively a \textit{left-shiftable $(i-1)$-bead}).
For a right-shiftable bead $b=(\be,j)\in\cA$ (regardless of the value of $\be\mod e$), one can define
\textit{shifting $b$ in $\cA$ one unit to the right} as the procedure replacing $\cA$ by $(\cA\setminus \{ b \}) \sqcup \{ b' \}$, where $b'=(\be+1,j)$.
One defines similarly \textit{shifting $b$ one unit to the left} for a left-shiftable bead $b$.

Secondly, this requires an order on the beads of a given abacus with charge $\bs$.
For beads $b=(\be,j)$ and $b'=(\be',j')$ in $\cA$, set 
$$
b<b' 
\hspace{1cm}
\Longleftrightarrow
\hspace{1cm}
\be<\be' 
\text{ \quad or \quad } 
(\, \be=\be' \mand j>j' \,).
$$
Form the sequence $(b_1,b_2,\dots, b_r)$ of all left- and right-shiftable $i$ beads of $\cA$, where $b_1<b_2<\dots<b_r$, and
encode each $b_k$ by $-$ if it is left-shiftable and by $+$ if it is right-shiftable, and form the corresponding word $w$.
Delete all subwords of the form $-+$ in $w$ recursively, ending up with 
a word $w'$ of the form $++\dots+-\dots--$. The word $w'$ depends only on $w$ and is independent of the order in which the deletions are made.
The \textit{good left-shiftable $i$-bead} (respectively \textit{good right-shiftable $i$-bead}) of $\cA$ is, if it exists, 
the bead $b_k$ of $w$ corresponding to the leftmost $-$ (respectively the rightmost $+$) in $w'$.

\begin{remark}\label{remUglovvsKleshchev}
Note that the order generalises the usual order on beads of level one (i.e. one-runner) abaci in a non-canonical way. 
Indeed, this order gives rises to the so-called \textit{Uglov} realization of the $\sle$-crystal,
in opposition to the \textit{Kleshchev} realization (inducing an isomorphic but different crystal) which arises from another order on beads,
see for instance \cite[Example 6.2.16]{GeckJacon2011}.
For our purposes, it is essential to work with Uglov's realization, see Section \ref{shancat}.
\end{remark}

\begin{theorem}\cite[Section 3]{JMMO1991}, \cite[Theorem 2.8]{FLOTW1999}\label{thmchapeaucrystal}
There is an arrow $\cA\overset{i}{\rightarrow}\cA'$ in the $\sle$-crystal 
if and only if the following equivalent situations hold:
\begin{enumerate}
 \item $\cA'$ is obtained from $\cA$ by shifting its
good right-shiftable $i$-node one unit to the right.
\item $\cA$ is obtained from $\cA'$ by shifting its
good left-shiftable $i$-node one unit to the left.
\end{enumerate}
\end{theorem}

\begin{remark}\label{remtranslationcharge1}
Translating the charge $\bs$ by an integer $t$ ($\mbs\mapsto \mbs+(t,t,\dots,t)$) amounts to translating the 
labels of the arrows of the $\sle$-crystal graph by $t\mod e$.
\end{remark}


Each connected component of the $\sle$-crystal has a unique source vertex. Such elements are called \textit{highest weight vertices} for $\sle$.
In \cite{JaconLecouvey2012}, Jacon and Lecouvey have given a
combinatorial characterisation of highest weight vertices for $\sle$.
Let us recall it here by translating it into the language of abaci.

\begin{definition}\cite[Definition 2.2]{JaconLecouvey2012} \label{defperiodJL}
Let $\cA$ be an $\ell$-abacus.
The \textit{first $e$-period} of $\cA$ is, if it exists, the sequence of $e$ beads 
$P_1=(b_i=(\be_i,j_i))_{i=1,\dots,e}$ of $\cA$ satisfying $\be_{i+1} = \be_i-1$ for all $i=1,\dots,e-1$
such that 
\begin{enumerate} 
\item $j_{i+1} \leq j_i$ for all $i=1,\dots,e-1$,  
 \item $\be_1 = \max \left\{ \be \mid (\be,j)\in\cA, 1\leq j \leq \ell \right\}$,
\item $j_i = \min \left\{ j \mid (\be,j) \in\cA, \; \beta=\beta_i \right\}$.
\end{enumerate}
\end{definition}
The procedure $\cA \mapsto \cA\setminus P_1$ is called \textit{peeling $P_1$ off $\cA$}.
Recursively, one defines the \textit{$k$'th $e$-period} of $\cA$ as, if it exists,
the first $e$-period of the abacus obtained from $\cA$ by peeling off $P_1 \dots, P_{k-1}$.
We call $\cA$ \textit{totally $e$-periodic} if there exists $k\geq0$ and $\br\in\Z^\ell$ such that
peeling off the first $k$ periods of $\cA$ results in $\cA(\bemptyset,\br)$.
We have the following result \cite[Theorem 5.9]{JaconLecouvey2012}.
\begin{proposition} \label{hwvchapeau}
An abacus is a highest weight vertex for $\sle$ if and only if it is totally $e$-periodic.
\end{proposition}

\subsection{$\slinf$-crystal of the Fock space}\label{Scrystal}

There is a second important action on level $\ell$ Fock spaces, namely that of a Heisenberg algebra.
This was first investigated by Uglov \cite{Uglov1999}, where he considered
the action of a quantum Heisenberg algebra on the $v$-deformed Fock space.
It has been shown in \cite{Losev2015} (using a Heisenberg categorical action appearing in \cite{ShanVasserot2012}) that this action yields a new crystal structure on $\cF_\bs$, called the $\slinf$-crystal
(or sometimes the Heisenberg crystal).
Another definition of the $\slinf$-crystal has then been given in \cite{Gerber2016a} in purely combinatorial terms.

The $\slinf$-crystal is determined by the action of some operators $\tb_\si$ \footnote{In \cite{ShanVasserot2012}, they are denoted $\ta_\si$.}, where $\si$ is a partition,
that commute with the Kashiwara $\sle$-crystal operators $\tf_i$, $\te_i$.
It is encoded in a directed graph whose vertices are all abaci with charge $\bs$,
and each of whose connected components is isomorphic to the Kashiwara $\slinf$-crystal on
partitions (the direct limit of the Kashiwara $\mathcal{U}_v(\mathfrak{sl}_e)$-crystals on partitions),
i.e. to the Young graph \cite[Definition 5.1.2]{Sagan2001} (the branching graph of the symmetric group in characteristic $0$).
Each connected component of this graph has a unique source vertex, called a \textit{highest weight vertex} for $\slinf$. 
For each $\cA\in\cF_\bs$, there exists a unique partition $\theta$ such that
$\cA=\tb_\theta\cA^\circ$, where $\cA^\circ$ is the unique highest weight vertex $\cA^\circ$ in the connected component of $\cA$.
\footnote{In \cite{Gerber2016} and \cite{Gerber2016a}, the partition $\theta$ is denoted $\ka$.}.

In \cite[Section 6.3]{Gerber2016}, the operators $\tb_\si$ have been expressed as
the composition of more elementary operators $\tb^+_{k}$, where $k\in\Z$.
Drawing an arrow labeled by $k$ between two abaci $\cA$ and $\cA'$
whenever $\cA' = \tb^+_{k}\cA$ gives the description of the $\slinf$-crystal
isomorphic to the Young graph.
Reversing the arrows gives the action of operators denoted by $\tb^-_{k}$
\footnote{In \cite{Gerber2016} and \cite{Gerber2016a}, $\tb^\pm_k$ are denoted $\tb_{\pm1,k}$}.
The graph isomorphism between an arbitrary connected component of the $\slinf$-crystal and the Young graph is
then given by the bijection $\cA=\tilde{b}_\theta \cA^\circ \mapsto \theta$, and an arrow labeled $k$ goes to an arrow that adds a box of content $k$ to a partition.

\begin{remark}
The operators $\tb^\pm_k$ are analogues of the Kashiwara crystal operators for the Heisenberg algebra.
More precisely, they are the image of the Kashiwara operators of $\Uinf$ under a certain bijection, see \cite{Gerber2016} and \cite{Gerber2016a}.
\end{remark}

An explicit formula for computing the $\slinf$-crystal of $\cF_\bs$ using $\tb_\si$ has been first
given in \cite[Theorem 5.11]{Gerber2016a} using the combinatorics of vertical strips. Let us recall it.

\begin{theorem}\label{thmverticalstrips} Let $\si=(\si_1,\si_2,\dots)$ be a partition.
The operator $\tb_\si$ acts on any charged multipartition by recursively adding 
the $k$'th good vertical $e$-strip $\si_k$ times for each $k\geq1$.
\end{theorem}

Thanks to this formula we can reconstruct the whole $\slinf$-crystal of $\cF_\bs$. 
However, the process of constructing the $\slinf$-crystal using Theorem \ref{thmverticalstrips} is recursive: it requires one to start from the highest weight vertex $(\emp,\emp,\dots,\emp)$
and then to apply the various operators $\tb_\si$, adding a new crystal component whenever there is no incoming arrow at an $\ell$-partition $\bla$.
In Section \ref{S-crystal edges}, we will give an explicit description of the action of $\tb^\pm_k$ on an arbitrary $\ell$-partition $\bla$ in terms of abaci, consequently permitting the construction of any connected component of the $\slinf$-crystal starting from any $\bla$ in that connected component.
This will yield an easy formula (Theorem \ref{thmdepth}) for finding the position of any vertex in its connected component without going back to the source.

\begin{remark}\label{remtranslationcharge2}
The $\slinf$-crystal is invariant under componentwise translation of the charge $\bs$ by any integer $t$.
For this reason, together with Remark \ref{remtranslationcharge1}, we may omit the labeling of columns by $\beta$-numbers when drawing an abacus.
\end{remark}

\begin{example}\label{exampleHcrystal}
The following is the beginning of two connected components of the $\slinf$-crystal of the level 2 Fock space
with parameters $\bs=(0,1)$ and $e=3$.

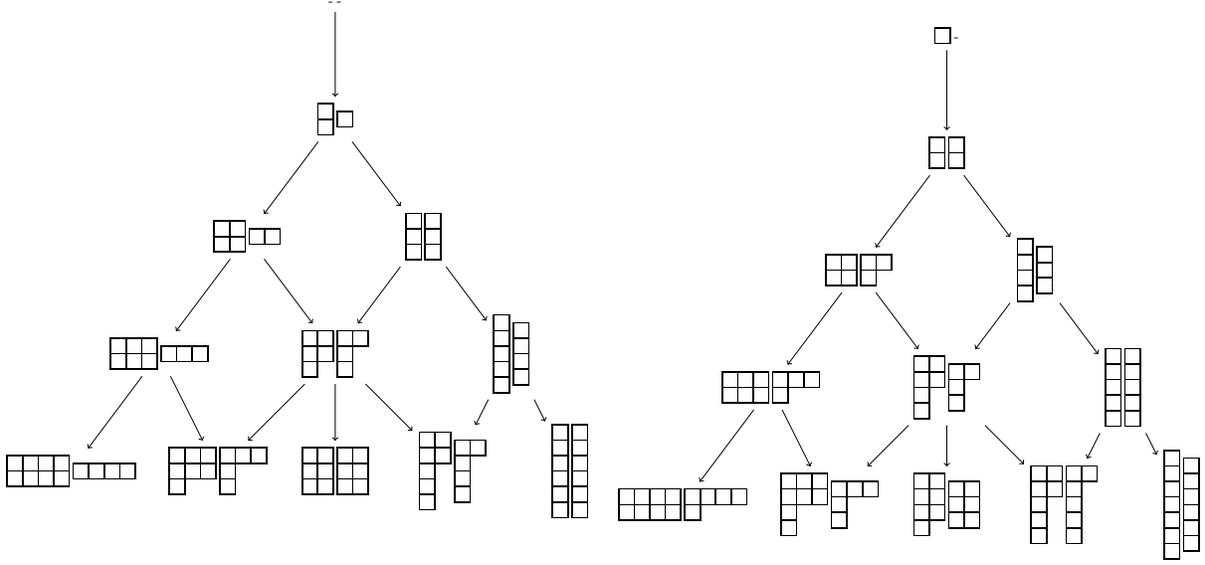
\begin{figure}[H] 

\makebox[\textwidth][c]{

\scalebox{.65}{

\begin{tabular}{cc}

\begin{tikzpicture}[scale=0.6, every node/.style={scale=1}]

\node (a) at (8,12) {%
${ \tiny  \m \, \m }$
};

\node (b) at (8,8) {%
{${ \tiny \young(\ ,\ ) \, \young(\ )   }$}
};

\node (c1) at (5,4) {%
{${ \tiny \young(\ \ ,\ \ ) \, \young(\ \ ) }$}
}; 
\node (c2) at (11,4) {%
{${ \tiny \young(\ ,\ ,\ ) \, \young(\ ,\ ,\ ) }$}
};

\node(d1) at (2,0) {%
{${ \tiny \young(\ \ \ ,\ \ \ ) \, \young(\ \ \ )  }$}
};
\node(d2) at (8,0) {%
{${ \tiny \young(\ \ ,\ \ ,\ ) \, \young(\ \ ,\ ,\ ) }$}
};
\node(d3) at (14,0) {%
{${ \tiny \young(\ ,\ ,\ ,\ ,\ ) \, \young(\ ,\ ,\ ,\ ) }$}
};

\node(e1) at (-1,-4) {%
{${ \tiny \young(\ \ \ \ ,\ \ \ \ ) \, \young(\ \ \ \ )   }$}
};
\node(e2) at (4,-4) {%
{${ \tiny \young(\ \ \ ,\ \ \ ,\ ) \, \young(\ \ \ ,\ ,\ )  }$}
};
\node(e3) at (8,-4) {%
{${ \tiny \young(\ \ ,\ \ ,\ \ ) \, \young(\ \ ,\ \ ,\ \ ) }$}
};
\node(e4) at (12,-4) {%
{${ \tiny \young(\ \ ,\ \ ,\ ,\ ,\ ) \, \young(\ \ ,\ ,\ ,\ )  }$}
};
\node(e5) at (16,-4) {%
{${ \tiny \young(\ ,\ ,\ ,\ ,\ ,\ ) \, \young(\ ,\ ,\ ,\ ,\ ,\ ) }$}
};

\draw[->] (a) -- node [font=\tiny, left] {} (b) ;
\draw[->] (b) -- node [font=\tiny,left] {} (c1) ;
\draw[->] (b) -- node [font=\tiny, above] {} (c2) ;
\draw[->] (c1) -- node [font=\tiny, above] {} (d1) ;
\draw[->] (c1) -- node [font=\tiny, above] {} (d2) ;
\draw[->] (c2) -- node [font=\tiny] {} (d3) ;
\draw[->] (c2) -- node [font=\tiny, above] {} (d2) ;

\draw[->] (d1) -- node [font=\tiny, above] {} (e1) ;
\draw[->] (d1) -- node [font=\tiny, above] {} (e2) ;
\draw[->] (d2) -- node [font=\tiny, above] {} (e2) ;
\draw[->] (d2) -- node [font=\tiny, above] {} (e3) ;
\draw[->] (d2) -- node [font=\tiny, above] {} (e4) ;
\draw[->] (d3) -- node [font=\tiny, above] {} (e4) ;
\draw[->] (d3) -- node [font=\tiny, above] {} (e5) ;

\node (f) at (0,-7) { \ };

\end{tikzpicture}

&

\begin{tikzpicture}[scale=0.6, every node/.style={scale=1}]

\node (a) at (8,12) {%
${ \tiny  \young(\ ) \, \m }$
};

\node (b) at (8,8) {%
{${ \tiny \young(\ ,\ ) \, \young(\ ,\ )   }$}
};

\node (c1) at (5,4) {%
{${ \tiny \young(\ \ ,\ \ ) \, \young(\ \ ,\ ) }$}
}; 
\node (c2) at (11,4) {%
{${ \tiny \young(\ ,\ ,\ ,\ ) \, \young(\ ,\ ,\ ) }$}
};

\node(d1) at (2,0) {%
{${ \tiny \young(\ \ \ ,\ \ \ ) \, \young(\ \ \ ,\ )  }$}
};
\node(d2) at (8,0) {%
{${ \tiny \young(\ \ ,\ \ ,\ ,\ ) \, \young(\ \ ,\ ,\ ) }$}
};
\node(d3) at (14,0) {%
{${ \tiny \young(\ ,\ ,\ ,\ ,\ ) \, \young(\ ,\ ,\ ,\ ,\ ) }$}
};

\node(e1) at (-1,-4) {%
{${ \tiny \young(\ \ \ \ ,\ \ \ \ ) \, \young(\ \ \ \ ,\ )   }$}
};
\node(e2) at (4,-4) {%
{${ \tiny \young(\ \ \ ,\ \ \ ,\ ,\ ) \, \young(\ \ \ ,\ ,\ )  }$}
};
\node(e3) at (8,-4) {%
{${ \tiny \young(\ \ ,\ \ ,\ \ ,\ ) \, \young(\ \ ,\ \ ,\ \ ) }$}
};
\node(e4) at (12,-4) {%
{${ \tiny \young(\ \ ,\ \ ,\ ,\ ,\ ) \, \young(\ \ ,\ ,\ ,\ ,\ )  }$}
};
\node(e5) at (16,-4) {%
{${ \tiny \young(\ ,\ ,\ ,\ ,\ ,\ ,\ ) \, \young(\ ,\ ,\ ,\ ,\ ,\ ) }$}
};

\draw[->] (a) -- node [font=\tiny, left] {} (b) ;
\draw[->] (b) -- node [font=\tiny,left] {} (c1) ;
\draw[->] (b) -- node [font=\tiny, above] {} (c2) ;
\draw[->] (c1) -- node [font=\tiny, above] {} (d1) ;
\draw[->] (c1) -- node [font=\tiny, above] {} (d2) ;
\draw[->] (c2) -- node [font=\tiny] {} (d3) ;
\draw[->] (c2) -- node [font=\tiny, above] {} (d2) ;

\draw[->] (d1) -- node [font=\tiny, above] {} (e1) ;
\draw[->] (d1) -- node [font=\tiny, above] {} (e2) ;
\draw[->] (d2) -- node [font=\tiny, above] {} (e2) ;
\draw[->] (d2) -- node [font=\tiny, above] {} (e3) ;
\draw[->] (d2) -- node [font=\tiny, above] {} (e4) ;
\draw[->] (d3) -- node [font=\tiny, above] {} (e4) ;
\draw[->] (d3) -- node [font=\tiny, above] {} (e5) ;
\end{tikzpicture}
\end{tabular}
}
}
\label{exScrystal}
\end{figure}

\end{example}

\section{The cyclotomic rational Cherednik algebra and its Category $\oh$}\label{rca}
 In this section we explain the representation theoretic meaning of the $\slinf$-crystal from the viewpoint of cyclotomic rational Cherednik algebras.

\subsection{The rational Cherednik algebra $\mathsf{H}_c(W)$} A finite subgroup $W\subset \mathsf{GL}_n(\C)$ is called a complex reflection group if it is generated by $S:=\{s\in W\mid\rk (\mathrm{Id}-s)=1\}$ (such elements $s$ are called reflections). For example, any finite Coxeter group is a complex reflection group. 
The $\C$-vector space $\mathfrak{h}$ of minimal dimension such that there is a faithful representation $\rho:W\rightarrow \mathsf{GL}(\mathfrak{h})$ is called the reflection representation of $W$. 
For each reflection $s\in S$, let $\alpha_s\in \mathfrak{h}^*$ and $\alpha_s^\vee\in\mathfrak{h}$ be eigenvectors with eigenvalue different from $1$ such that $\langle \alpha_s,\alpha_s^\vee\rangle=2$ where $\langle-,-\rangle:\hstar\times\mathfrak{h}\longrightarrow\mathbb{C}$ is the natural pairing. 
Let $c:S\rightarrow \C$ be a conjugacy-invariant function on the set of reflections of $W$ and write $c_s$ for $c(s)$. We may think of $c$ as a multiparameter attaching a parameter to each conjugacy class of reflections; with this understood, we will refer to $c$ simply as a \textit{parameter}. The rational Cherednik algebra $H_c(W)$ is the quotient of $\mathrm{T}(\mathfrak{h}\oplus\hstar)\rtimes \mathbb{C}[W]$ by the relations $$[x,x']=0,\quad[y,y']=0,\quad[y,x]=\langle y,x \rangle -\sum_{s\in S} c_s\langle \alpha_s,y\rangle \langle x,\alpha_s^\vee\rangle s$$
for all $x,x'\in\hstar$ and all $y,y'\in\mathfrak{h}$ \cite{EtingofGinzburg2002}. The rational Cherednik algebra is called rational double affine Hecke algebra (rational DAHA) by some authors, e.g. in \cite{ShanVasserot2012}, \cite{VV}.

\subsection{The cyclotomic rational Cherednik algebra $\mathsf{H}_c(G(\ell,1,n))$}\label{RCAparams1}

The group $G(\ell,1,n)\subset\mathrm{GL}(n,\C)$ consists of all $n\times n$ permutation matrices whose nonzero entries are $\ell$'th roots of $1$. It is the semidirect product $\Gamma^n\rtimes S_n$ where $\Gamma$ is the cyclic group of order $\ell$. The Weyl groups $S_n=A_{n-1}=G(1,1,n)$ and $B_n=G(2,1,n)$ appear as special cases. From now on, take $\ell\geq 2$. The reflection representation $\mathfrak{h}$ is then isomorphic to $\C^n$. Fix the standard basis $\{y_i\}$, $i=1,\dots,n$, for $\mathfrak{h}$ (i.e. $y_i$ is the vector with $1$ in the $i$'th place and $0$'s elsewhere). Then $\mathfrak{h}^*\simeq\C^n$ with dual basis $\{x_i\}$, $i=1,\dots,n$.  The set of reflections $S$ in $G(\ell,1,n)$ is given by:
$$S=\{\gamma_i\mid 1\neq\gamma\in\Gamma,\;i=1,\dots, n\}\cup
\{s_{ij}^\gamma\mid \gamma\in\Gamma,\; 1\leq i<j\leq n\}.$$
where $\gamma_i:=\mathrm{diag}(1,\dots,\gamma,\dots,1)$ with $\gamma$ in the $i$'th place, and $s_{ij}^\gamma:=\gamma_i^{-1}s_{ij}\gamma_i=s_{ij}\gamma_i\gamma_j^{-1}$.
There are $\ell$ conjugacy classes of reflections in $G(\ell,1,n)$: one class containing all $s_{ij}^\gamma$, and one class for each $\gamma\in\Gamma$, $\gamma\neq 1$, consisting of all $\gamma_i$, $i=1,\dots,n$. 

For each reflection $s\in S$ we fix the eigenvectors $\alpha_s$, $\alpha_s^\vee$ as follows. For $s=s_{ij}^\gamma$ we take $\alpha_s=\alpha_{ij}^\gamma:=(0,\dots,-\gamma,\dots,1,\dots,0)\in\hstar$ with $-\gamma$ in the $i$'th place and $1$ in the $j$'th place. Then 
$(\alpha_{ij}^\gamma)^\vee:=(0,\dots,-\gamma^{-1},\dots,1,\dots,0)\in\mathfrak{h}$. 
The reflection $s=\gamma_i$ has $x_i\in\hstar$ and $y_i\in\mathfrak{h}$ as eigenvectors with eigenvalue different from $1$, but in order that $\langle\alpha_s,\alpha_s^\vee\rangle=2$, we take $\alpha_s=x_i$ and $\alpha_s^\vee=2y_i$.

Fix a parameter $c:S\rightarrow\C$. Let $\kappa$ be the parameter for the conjugacy class of $s_{ij}^\gamma$, and let $c_\gamma$ be the parameter for the conjugacy class of $\gamma_i$, $\gamma\neq 1\in\Gamma$. The cyclotomic rational Cherednik algebra is generated by the group algebra $\C[G(\ell,1,n)]$ and the two polynomial algebras $\C[x_1,\dots,x_n]$ and $\C[y_1,\dots,y_n]$ with relations $s_{ij}y_k=y_{s_{ij}(k)}s_{ij}$, $s_{ij}x_k=x_{s_{ij}(k)}s_{ij}$, $\gamma_j y_i=\delta_{ij}\gamma y_i\gamma_j$, $\gamma_j x_i=\delta_{ij}\gamma^{-1} x_i\gamma_j$, and the following two relations.
If $i\neq j$ then:
$$[y_i,x_j]=-\sum_{\gamma\in\Gamma}\kappa \langle \alpha_{ij}^\gamma,y_i\rangle \langle x_j,(\alpha_{ij}^\gamma)^\vee\rangle s_{ij}^\gamma=\kappa\sum_{\gamma\in\Gamma}\gamma s_{ij}^\gamma$$
If $i=j$ then:
$$
[y_i,x_i]=1-\kappa\sum_{j\neq i}\sum_{\gamma\in\Gamma}s_{ij}^\gamma -\sum_{1\neq\gamma\in\Gamma}2c_\gamma\gamma_i
$$

\subsection{Parameters for the cyclotomic Cherednik algebra versus parameters for the Fock space}\label{RCAparams2} Shan and Vasserot introduce a reparametrization for $H_c(G(\ell,1,n))$ in order to translate between the Cherednik algebra parameters and the Fock space parameters \cite[Section 3.3]{ShanVasserot2012}:
$$h=-\kappa,\quad -c_\gamma'=\sum_{p=0}^{\ell-1}\gamma^{-p}h_p$$
Here, $c_\gamma'=2c_\gamma$ for $\gamma\neq 1$ and $c_1'=-1$. If $e\in\Z_{\geq 2}$, $\mbs=(s_1,\dots,s_\ell)\in\Z^\ell$ are the parameters for the level $\ell$ Fock space, then following \cite[Theorem 6.10]{ShanVasserot2012} the corresponding Cherednik algebra parameters are given by:
$$h=-1/e,\quad h_p=(s_{p+1}-s_p)/e,\quad p>0. $$

Then $\kappa=1/e>0$, and using this convention we are opposite to Losev in \cite{Losev2015} who takes $\kappa=-1/e$; therefore by \cite[Section 4.1.4]{Losev2015}, whenever Losev has $|(\lambda^1,\lambda^2,\dots,\lambda^\ell),(s_1,s_2,\dots,s_\ell)\rangle$, we have\\ $|((\lambda^1)^t,(\lambda^2)^t,\dots,(\lambda^\ell)^t),(-s_1,-s_2,\dots,-s_\ell)\rangle$.

\subsection{Category $\oh_c(W)$} 
The category $\oh_c(W)$ is the category of finitely generated $H_c(W)$-modules on which $\mfh$ acts locally nilpotently \cite{GGOR}. It is a highest weight category with finitely many simple objects, which are labeled by the irreducible representations of $W$ over $\C$ \cite{GGOR}. In the case of $W=G(\ell,1,n)$, this means that the simple objects are labeled by $\ell$-partitions $\bla$ of $n$. We write $\el(\bla)$ for the simple module in $\oh_c(W)$ labeled by $\bla$. For a generic parameter $c$, $\oh_c(W)\simeq \C[W]$ and $\el(\lambda)\cong \C[x_1,\dots,x_n]\otimes\lambda$ as $\C$-vector spaces for all $\lambda\in\Irr W$. We are interested in those parameters for which $\oh_c(W)$ is not a semisimple category and for which the simple modules can be smaller.

\subsection{Integral parameters}\label{integralpar}
In this paper we always assume that $\kappa=1/e$ for some integer $e\geq 2$ and that the Fock space $\mathcal{F}_\mbs$ has integral charge $\mbs\in\Z^\ell$; we call this the case of integral parameters. 
To consider all parameters $c$ such that $\oh_c(G(\ell,1,n))$ is not a semisimple category, it is necessary to allow $\kappa=r/e$ with $r\in\Z$, $\mathrm{gcd}(r,e)=1$, and also to allow rational charges $\mbs\in\Q^\ell$ for $\mathcal{F}_\mbs$. 
However, the general case of rational parameters can be reduced to the case of integral parameters.
Dipper and Mathas proved such a reduction in the case of Hecke algebras of $G(\ell,1,n)$ because of the existence of a certain Morita equivalence, see \cite[Theorem 1.1]{DipperMathas} or \cite[Section 5.4]{GeckJacon2011} for details.
The Cherednik category $\oh_c(G(\ell,1,n))$ is a highest weight cover of the category of finite-dimensional modules
for the Hecke algebra of $G(\ell,1,n)$ at certain parameters determined by $c$ \cite{Rouquier} (highest weight covers were defined in \cite{Rouquier}, and further developed in \cite{RSVV},\cite{Losev}). 
Highest weight covers are unique up to equivalence \cite{Rouquier},\cite{RSVV},\cite{Losev}; using this property, Rouquier showed that $\oh_c(G(\ell,1,n))$ is equivalent to a tensor product of category $\oh$'s of $G(\ell,1,n)$ with integral parameters \cite[Theorem 6.13, Remark 6.16]{Rouquier}, with a restriction on $\kappa$ which was later removed by Losev \cite[Proposition 3.2]{Losev2015a}. Without loss of generality we may therefore consider the integral parameter case.

\subsection{Parabolic induction and restriction, and Harish-Chandra series}
For any parabolic subgroup $W'\subset W$ (i.e. the stabilizer of a point in $\mathfrak{h}$), there are exact, biadjoint induction and restriction functors $\Ind_{W'}^W:\oh_c(W')\rightarrow\oh_c(W)$ and $\Res_{W'}^W:\oh_c(W)\rightarrow\oh_c(W')$ \cite{BezrukavnikovEtingof}. Finite-dimensional modules are characterized by the property that their restriction to $\oh_c(W')$ for any parabolic $W'\subset W$ is $0$ \cite{BezrukavnikovEtingof}. For this reason, finite-dimensional simple modules are sometimes referred to as \textit{cuspidal} modules (as in \cite[Section 5.7]{GGJL}, for example).  Any simple module $\el\in\oh_c(W)$ appears as a direct summand of the head of $\Ind_{W'}^W\;\el'$
for some finite-dimensional simple module $\el'\in\oh_c(W')$ and some parabolic subgroup $W'\subset W$ \cite{BezrukavnikovEtingof}. Let us call $(W',\el')$ a \textit{cuspidal pair} if $\el'$ is a finite-dimensional simple $\oh_c(W')$-module. The cuspidal pair $(W',\el')$ such that $\Hom_{\oh_c(W)}(\Ind_{W'}^W\el',\el)\neq 0$ is unique up to $W$-conjugacy \cite{LosevSA}. We will then use the standard terminology from Lie theory and refer to $(W',\el')$ as the \textit{cuspidal support} of $\el$, and say that $\el$ belongs to the \textit{Harish-Chandra series} of $(W',\el')$. This is completely analogous to Harish-Chandra series in Lie theory. In the case of finite unitary groups in positive, non-defining characteristic, the parametrization of Harish-Chandra series by cuspidal pairs is known to coincide with the parametrization of the Harish-Chandra series of a sum of type $B$ Cherednik algebras, see \cite{DVV2}.

\subsection{Parabolic subgroups and supports of simple modules in $\oh_c(G(\ell,1,n))$}\label{support} The support of $\el(\bla)$ is the set-theoretic support of $\el(\bla)$ as a $\C[\mathfrak{h}]$-module. It has a concrete description for $G(\ell,1,n)$ which we explain now. Let $W=G(\ell,1,n)$ for $\ell\geq 2$ and $W'\subset W$ be a parabolic subgroup. 
Parabolics $W'$ of $W$ are of the form $G(\ell,1,n_1)\times S_{n_2}\times S_{n_3}\times\dots\times S_{n_u}$ with $n_1+n_2+n_3+\dots+n_u\leq n$, up to $W$-conjugacy. We are interested in those parabolics $W'$ such that $\oh_c(W')$ contains a finite-dimensional representation. Suppose $\kappa=1/e$ is the parameter for the transpositions of $S_n\subset W$. Then $\oh_c(W')$ can contain a finite-dimensional representation only when $W'$ is of the form: 
$$W'=G(\ell,1,n-eq-p)\times S_e^{\times q}$$
up to conjugacy, with $0\leq q\leq \left\lfloor\frac{n}{e}\right\rfloor$ and $0\leq p\leq n-eq$; such a parabolic $W'$ is the stabilizer of the point 
$$a=(0,...,0,x_1,x_2,\dots,x_p,y_1,y_1,\dots,y_1,y_2,y_2,\dots,y_2,\dots,y_q,y_q,\dots,y_q)\in\mfh$$
where the entry $0$ occurs $n-eq-p$ times, and each entry $y_i$ occurs $e$ times \cite{Losev2015}. 
If $(\el',W')$ is the cuspidal support of $\el(\bla)$ for some finite-dimensional simple $\el'\in\oh_c(W')$, $W\mfh^{W'}$ coincides with the support of $\el(\bla)$.   We then write $p(\bla)$ for $p$ above and $q(\bla)$ for $q$ above, as in \cite{Losev2015}.

\subsection{$i$-induction and the $\sle$-crystal}\label{shancat} 
One part of the support of $\el(\bla)$, determined by the integer $p(\bla)$, is  given by the depth of $\bla$ in the $\sle$-crystal \cite{Shan2011},\cite{Losev2015}. Using Chuang and Rouquier's technique of $\mathfrak{sl}_2$-categorification \cite{ChuangRouquier2008}, Shan showed that the induction and restriction functors $\Ind_{W'}^W$ and $\Res_{W'}^W$,
for $W=G(\ell,1,n)$ and $W'=G(\ell,1,n-1)$, split into a direct sum of functors $F_i$ and $E_i$ respectively, $i=0,\dots, e-1$, which satisfy the defining relations of the Chevalley generators $f_i$, $e_i$ of $\sle$ \cite{Shan2011}: 
$$\Ind_{W'}^W =\bigoplus_{0\leq i\leq e-1} F_i,\quad\quad \Res_{W'}^W=\bigoplus_{0\leq i\leq e-1}E_i.$$ 
The functors $F_i$ and $E_i$ are called $i$-induction and $i$-restriction, respectively. Taking the head of $F_i(\el(\bla))$ and the socle of $E_i(\el(\bla))$ gives rise to an abstract crystal structure on 
the Grothendieck group of $\bigoplus_{n\geq 0}\oh_c(G(\ell,1,n))$, 
isomorphic to the $\sle$-crystal of the Fock space $\cF_\bs$ (explained in Section \ref{chapeaucrystal}), where $\bs$ is determined from $c$ by the formulas of Section \ref{RCAparams2}.
Note that this isomorphism has been made explicit by Losev \cite[Theorem 5.1]{Losev2013}: 
provided one works with Uglov's realization of the Fock space crystal, it is simply given by the correspondence $[\el(\bla)]\leftrightarrow \bla$, i.e.
if $\tilde{F_i}$ denotes $\mathrm{head}(F_i(-))$ and $\tilde{E_i}$ denotes $\mathrm{socle}(E_i(-))$, we have
$$ \tilde{F_i}(\el(\bla)) = L(\tf_i(\bla)) \mand \tilde{E_i}(\el(\bla)) = L(\te_i(\bla)),$$
where $\tf_i$ and $\te_i$ are the Kashiwara crystal operators on $\cF_\bs$ of Section \ref{chapeaucrystal}.

\subsection{The functor $A_\si$ and the $\slinf$-crystal}\label{amu}
We summarize the construction of the $\slinf$-crystal on $\mathcal{F}_\mbs$ by Losev in \cite[Section 5.1]{Losev2015}, who adapts the results of Shan and Vasserot in \cite[Section 5.6]{ShanVasserot2012} to the language of crystals. Shan and Vasserot define an exact functor from  $\oh_c(G(\ell,1,n-em))$ to $\oh_c(G(\ell,1,n))$, by: 
$$A_\si(M)=\Ind_{G(\ell,1,n-em)\times S_{em}}^{G(\ell,1,n)} M\boxtimes \el(e\si)$$
for $\si$ a partition of $m$ and $em\leq n$ (see \cite[Definition 5.12]{ShanVasserot2012}, note that they denote $A_\sigma$ by $a^*_\sigma$). Shan and Vasserot showed that the functor $A_\si$ gives rise to an exact functor on $\oplus_n \oh_c(G(\ell,1,n))$ and that $A_\sigma$ commutes with the $i$-induction and $i$-restriction functors $F_i$ and $E_i$ \cite[Proposition 5.15]{ShanVasserot2012}.

There is a bijection $$\{\hbox{partitions}\}\times\{\ell\hbox{-partitions }\bla\;|\;p(\bla)=0,\;q(\bla)=0\}\longrightarrow\{\ell\hbox{-partitions }\bmu\mid p(\bmu)=0\}$$ given by sending $(\sigma,\bla)$ to $\bmu=:\tilde{b}_\sigma(\bla)$ where $\el(\bmu)$ is the unique irreducible constituent (up to multiplicity) in the head of $A_\sigma(\bla)$ \cite[Section 5.6]{ShanVasserot2012},\cite[Section 5.1]{Losev2015}. Since the set of partitions has an $\slinf$-crystal structure, this bijection induces an $\slinf$-crystal structure on the set of $\ell$-partitions which have depth $0$ in the $\sle$-crystal. 
Losev then extends the crystal to all of $\mathcal{F}_\mbs$ by using the commutative diagram of \cite[Proposition 5.1]{Losev2015}.This gives a bijection as above: $$\{\hbox{partitions} \}\times\{\ell\hbox{-partitions }\bla\;|\;p(\bla)=p,\;q(\bla)=0\}\longrightarrow \{\ell\hbox{-partitions }\bmu\;|\;p(\bmu)=p\}$$
sending $(\sigma,\bla)$ to $\bmu=\tilde{b}_\sigma(\bla)$.
These bijections taken over all $p\geq 0$ then induce an $\slinf$-crystal structure on $\mathcal{F}_\mbs$, that commutes with the $\sle$-crystal.
The depth $|\bmu|$ of $\bmu$ in the $\slinf$-crystal is equal to $q(\bmu)$, determining the other part of the support of the simple module $\el(\bmu)\in\oh_c(G(\ell,1,n))$ \cite{Losev2015}.

\subsection{Finite-dimensional representations of $\mathsf{H}_c(G(\ell,1,n))$}\label{findim}
When $W=S_n$, it is well-known that $\oh_c(S_n)$ contains a finite-dimensional module if and only if $c=\pm r/n$ with $r\in\N$ coprime to $n$, and then there is exactly one finite-dimensional simple module, $\el(\triv)$ if $c>0$ and $\el(\mathrm{Sign})$ if $c<0$ \cite{BEG}. The category $\oh$ for a direct product of groups is the tensor product of the categories, and thus it follows that any $\el(\lambda)\in\oh_c(S_n)$ when $c=r/e>0$ with $r$ coprime to $e$ and $2\leq e\leq n$ appears as a direct summand of the head of $\Ind_{W'}^{S_n}\el(\triv^{\times m})$ where $W'=S_e^{\times m}$ for some $0\leq m\leq \left\lfloor\frac{n}{e}\right\rfloor$ \cite{BezrukavnikovEtingof}. 
Wilcox solved the problem of finding the support of any simple $\el(\lambda)\in\oh_c(S_n)$ \cite{Wilcox}.

In general, a module $\el\in\oh_c(W)$ is finite-dimensional if and only if its support is $0$. By the results of Shan, Shan and Vasserot, and Losev explained in the preceding paragraphs, this means that the simple module $\el(\bla)\in\oh_c(G(\ell,1,n))$ is finite-dimensional if and only if $\bla$ has depth $0$ in both the $\sle$- and $\slinf$-crystals. The crystal structures depend on the Fock space charge $(e,\mbs)$, equivalently, on the parameter $c$, so the classification of finite-dimensional representations also depends on $c$. We can find the cuspidal support $(W',\el')$ of $\el(\bla)\in\oh_c(W)$ by finding $W'$ using the formula in Section \ref{support}, and finding $\el'=\el(\tilde{\bla}^\circ\times\triv)\in\oh_c(G(\ell,1,n-eq-p)\times (S_e^{\times q}))$ by first finding the source vertex $\tilde{\bla}$ of the $\sle$-crystal component containing $\bla$, and then finding the source vertex $\tilde{\bla}^\circ$ of the $\slinf$-crystal component containing $\tilde{\bla}$.

\section{The rule for the arrows in the $\slinf$-crystal}\label{S-crystal edges}
\subsection{Quasiperiods, fore periods, and aft periods}

\begin{definition}\label{defquasiperiod}
Let $\cA$ be an $\ell$-abacus.
An \textit{$e$-quasiperiod} in $\cA$ is an ordered set of $e$ beads $\{b_1,\dots,b_e\}$ of $\cA$ such that that, if $b_i=(\be_i,j_i)$ for $i=1\dots e$, 
the following two conditions holds:
\begin{itemize}
 \item $\be_{i+1}=\be_i-1$ for all $i=1,\dots, e-1$.
 \item $j_i \geq j_{i+1}$ for all $i=1,\dots, e-1$.
 \end{itemize}
\end{definition}

\begin{definition} 
The abacus $\cA$ is totally $e$-quasiperiodic if the set of beads of $\cA$ 
can be partitioned into $e$-quasiperiods.
\end{definition}
In particular, $e$-periods (see Definition \ref{defperiodJL}) are $e$-quasiperiods. 
When $e$ is clear from the context, we will simply write \textit{periods} for $e$-periods and \textit{quasiperiods} for $e$-quasiperiods.
Likewise, we might simply talk about totally periodic/quasiperiodic abaci.

\medskip

Recall the total order on beads defined in Section \ref{chapeaucrystal}. The lexicographic extension of that order to quasiperiods is a total order on the set of quasiperiods of $\cA$. 
If $\cA$ is a totally quasiperiodic abacus with a chosen set of quasiperiods $\tilde{P}_1,\tilde{P}_2,\dots$ partitioning its beads, the quasiperiods $\tilde{P}_1,\tilde{P}_2,\dots$ may therefore be enumerated in such a way that $\tilde{P}_k>\tilde{P}_{k'}$ for all $k<k'$. 

\begin{notation}Given a totally ordered set of quasiperiods $\tilde{P}_k$ of $\cA$, $k\in\N$, write $\{b_k^{(m)}\}$ for the beads of $\tilde{P}_k$, $m\in\{1,\dots,e\}$, ordered so that $b_k^{(m)}>b_k^{(m+1)}$ for each $m\leq e-1.$
\end{notation}

\begin{lemma}\label{lemqper=>per}
An abacus $\cA$ is totally quasiperiodic if and only if it is totally periodic.
\end{lemma}

\begin{proof}It is clear that if $\cA$ is totally periodic then $\cA$ is totally quasiperiodic. 

Conversely, assume that $\cA$ is totally quasiperiodic. 
We will describe an algorithm that constructs the first period $P_1$ of $\cA$ step by step in such a way that the abacus at each step remains totally quasiperiodic. 
The following picture illustrates the process:

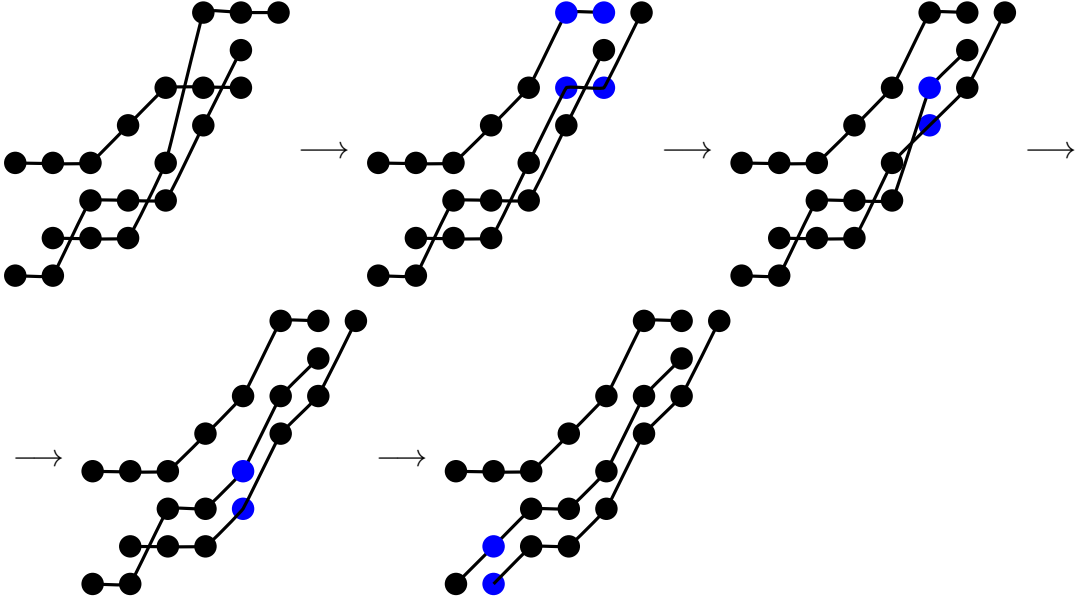
\begin{align*}
&\TikZ{[scale=.5]
\draw
(1,0)node[fill,circle,inner sep=3pt]{}
(2,0)node[fill,circle,inner sep=3pt]{}
(3,2)node[fill,circle,inner sep=3pt]{}
(4,2)node[fill,circle,inner sep=3pt]{}
(5,2)node[fill,circle,inner sep=3pt]{}
(6,4)node[fill,circle,inner sep=3pt]{}
(7,6)node[fill,circle,inner sep=3pt]{}
;
\draw[very thick] plot [smooth,tension=.1] coordinates{(1,0)(2,0)(3,2)(5,2)(6,4)(7,6)};
\draw
(1,3)node[fill,circle,inner sep=3pt]{}
(2,3)node[fill,circle,inner sep=3pt]{}
(3,3)node[fill,circle,inner sep=3pt]{}
(4,4)node[fill,circle,inner sep=3pt]{}
(5,5)node[fill,circle,inner sep=3pt]{}
(6,5)node[fill,circle,inner sep=3pt]{}
(7,5)node[fill,circle,inner sep=3pt]{}
;
\draw[very thick] plot [smooth,tension=.1] coordinates{(1,3)(3,3)(5,5)(7,5)};
\draw
(2,1)node[fill,circle,inner sep=3pt]{}
(3,1)node[fill,circle,inner sep=3pt]{}
(4,1)node[fill,circle,inner sep=3pt]{}
(5,3)node[fill,circle,inner sep=3pt]{}
(6,7)node[fill,circle,inner sep=3pt]{}
(7,7)node[fill,circle,inner sep=3pt]{}
(8,7)node[fill,circle,inner sep=3pt]{};
\draw[very thick] plot [smooth, tension=.1] coordinates{(2,1)(4,1)(5,3)(6,7)(7,7)(8,7)};
}
\longrightarrow\;
\TikZ{[scale=.5]
\draw
(1,0)node[fill,circle,inner sep=3pt]{}
(2,0)node[fill,circle,inner sep=3pt]{}
(3,2)node[fill,circle,inner sep=3pt]{}
(4,2)node[fill,circle,inner sep=3pt]{}
(5,2)node[fill,circle,inner sep=3pt]{}
(6,4)node[fill,circle,inner sep=3pt]{}
(7,6)node[fill,circle,inner sep=3pt]{}
;
\draw[very thick] plot [smooth,tension=.1] coordinates{(1,0)(2,0)(3,2)(5,2)(6,4)(7,6)};
\draw
(1,3)node[fill,circle,inner sep=3pt]{}
(2,3)node[fill,circle,inner sep=3pt]{}
(3,3)node[fill,circle,inner sep=3pt]{}
(4,4)node[fill,circle,inner sep=3pt]{}
(5,5)node[fill,circle,inner sep=3pt]{}
(6,5)node[blue,fill,circle,inner sep=3pt]{}
(7,5)node[blue,fill,circle,inner sep=3pt]{}
;
\draw[very thick] plot [smooth,tension=.1] coordinates{(1,3)(3,3)(5,5)(6,7)(7,7)};
\draw
(2,1)node[fill,circle,inner sep=3pt]{}
(3,1)node[fill,circle,inner sep=3pt]{}
(4,1)node[fill,circle,inner sep=3pt]{}
(5,3)node[fill,circle,inner sep=3pt]{}
(6,7)node[blue,fill,circle,inner sep=3pt]{}
(7,7)node[blue,fill,circle,inner sep=3pt]{}
(8,7)node[fill,circle,inner sep=3pt]{};
\draw[very thick] plot [smooth, tension=.1] coordinates{(2,1)(4,1)(5,3)(6,5)(7,5)(8,7)};
}
\longrightarrow\;
\TikZ{[scale=.5]
\draw
(1,0)node[fill,circle,inner sep=3pt]{}
(2,0)node[fill,circle,inner sep=3pt]{}
(3,2)node[fill,circle,inner sep=3pt]{}
(4,2)node[fill,circle,inner sep=3pt]{}
(5,2)node[fill,circle,inner sep=3pt]{}
(6,4)node[blue,fill,circle,inner sep=3pt]{}
(7,6)node[fill,circle,inner sep=3pt]{}
;
\draw[very thick] plot [smooth,tension=.1] coordinates{(1,0)(2,0)(3,2)(5,2)(6,5)(7,6)};
\draw
(1,3)node[fill,circle,inner sep=3pt]{}
(2,3)node[fill,circle,inner sep=3pt]{}
(3,3)node[fill,circle,inner sep=3pt]{}
(4,4)node[fill,circle,inner sep=3pt]{}
(5,5)node[fill,circle,inner sep=3pt]{}
(6,5)node[blue,fill,circle,inner sep=3pt]{}
(7,5)node[fill,circle,inner sep=3pt]{}
;
\draw[very thick] plot [smooth,tension=.1] coordinates{(1,3)(3,3)(5,5)(6,7)(7,7)};
\draw
(2,1)node[fill,circle,inner sep=3pt]{}
(3,1)node[fill,circle,inner sep=3pt]{}
(4,1)node[fill,circle,inner sep=3pt]{}
(5,3)node[fill,circle,inner sep=3pt]{}
(6,7)node[fill,circle,inner sep=3pt]{}
(7,7)node[fill,circle,inner sep=3pt]{}
(8,7)node[fill,circle,inner sep=3pt]{};
\draw[very thick] plot [smooth, tension=.1] coordinates{(2,1)(4,1)(5,3)(6,4)(7,5)(8,7)};}
\longrightarrow
\\
&\longrightarrow\;
\TikZ{[scale=.5]
\draw
(1,0)node[fill,circle,inner sep=3pt]{}
(2,0)node[fill,circle,inner sep=3pt]{}
(3,2)node[fill,circle,inner sep=3pt]{}
(4,2)node[fill,circle,inner sep=3pt]{}
(5,2)node[blue,fill,circle,inner sep=3pt]{}
(6,4)node[fill,circle,inner sep=3pt]{}
(7,6)node[fill,circle,inner sep=3pt]{}
;
\draw[very thick] plot [smooth,tension=.1] coordinates{(1,0)(2,0)(3,2)(4,2)(5,3)(6,5)(7,6)};
\draw
(1,3)node[fill,circle,inner sep=3pt]{}
(2,3)node[fill,circle,inner sep=3pt]{}
(3,3)node[fill,circle,inner sep=3pt]{}
(4,4)node[fill,circle,inner sep=3pt]{}
(5,5)node[fill,circle,inner sep=3pt]{}
(6,5)node[fill,circle,inner sep=3pt]{}
(7,5)node[fill,circle,inner sep=3pt]{}
;
\draw[very thick] plot [smooth,tension=.1] coordinates{(1,3)(3,3)(5,5)(6,7)(7,7)};
\draw
(2,1)node[fill,circle,inner sep=3pt]{}
(3,1)node[fill,circle,inner sep=3pt]{}
(4,1)node[fill,circle,inner sep=3pt]{}
(5,3)node[blue,fill,circle,inner sep=3pt]{}
(6,7)node[fill,circle,inner sep=3pt]{}
(7,7)node[fill,circle,inner sep=3pt]{}
(8,7)node[fill,circle,inner sep=3pt]{};
\draw[very thick] plot [smooth, tension=.1] coordinates{(2,1)(4,1)(5,2)(6,4)(7,5)(8,7)};}
\longrightarrow\;
\TikZ{[scale=.5]
\draw
(1,0)node[fill,circle,inner sep=3pt]{}
(2,0)node[blue,fill,circle,inner sep=3pt]{}
(3,2)node[fill,circle,inner sep=3pt]{}
(4,2)node[fill,circle,inner sep=3pt]{}
(5,2)node[fill,circle,inner sep=3pt]{}
(6,4)node[fill,circle,inner sep=3pt]{}
(7,6)node[fill,circle,inner sep=3pt]{}
;
\draw[very thick] plot [smooth,tension=.1] coordinates{(1,0)(2,1)(3,2)(4,2)(5,3)(6,5)(7,6)};
\draw
(1,3)node[fill,circle,inner sep=3pt]{}
(2,3)node[fill,circle,inner sep=3pt]{}
(3,3)node[fill,circle,inner sep=3pt]{}
(4,4)node[fill,circle,inner sep=3pt]{}
(5,5)node[fill,circle,inner sep=3pt]{}
(6,5)node[fill,circle,inner sep=3pt]{}
(7,5)node[fill,circle,inner sep=3pt]{}
;
\draw[very thick] plot [smooth,tension=.1] coordinates{(1,3)(3,3)(5,5)(6,7)(7,7)};
\draw
(2,1)node[blue,fill,circle,inner sep=3pt]{}
(3,1)node[fill,circle,inner sep=3pt]{}
(4,1)node[fill,circle,inner sep=3pt]{}
(5,3)node[fill,circle,inner sep=3pt]{}
(6,7)node[fill,circle,inner sep=3pt]{}
(7,7)node[fill,circle,inner sep=3pt]{}
(8,7)node[fill,circle,inner sep=3pt]{};
\draw[very thick] plot [smooth, tension=.1] coordinates{(2,0)(3,1)(4,1)(5,2)(6,4)(7,5)(8,7)};}
\end{align*}

Now we explain the algorithm. Let $\tilde{P}_1,\;\tilde{P}_2,\dots$ be a chosen set of quasiperiods partitioning the beads of $\cA$ with $\tilde{P}_k>\tilde{P}_{k'}$ for all $k<k'$. Then $b_1^{(1)}$ is the maximal bead in $\cA$. If $\tilde{P}_1$ is a period, we are done. Otherwise, if $\tilde{P}_1$ is not a period, then there is a bead $(\beta,j)\in\cA$ where $(\beta, j')\in \tilde{P}_1$ and $j<j'$. Pick the largest bead $(\beta,j)$ such that this happens. Since $\cA$ is totally quasiperiodic, $(\beta,j)$ belongs to a quasiperiod $\tilde{P}_k$ for some $k\neq 1$. Write $(\beta,j')=b_1^{(m)}$ and $(\beta,j)=b_k^{(n)}$ as above (i.e. $(\beta,j')$ is the $m$'th bead of $\tilde{P}_1$ and $(\beta,j)$ is the $n$'th bead of $\tilde{P}_k$). Observe that $m\geq n$ since $\beta(b_k^{(1)})\leq\beta(b_1^{(1)})$ due to $b_1^{(1)}$ being the maximal bead of $\cA$.
Let $s\in\Z$, $0<s\leq e-m$ be minimal such that $b_1^{(m+s)}$ is in a lower row than $b_k^{(n+s)}$, if such an $s$ exists; otherwise, set $s=e-m+1$. Now set $$\tilde{P}'_1:=\{b_1^{(1)},b_1^{(2)},\dots,b_1^{(m-1)},b_k^{(n)},b_k^{(n+1)},\dots,b_k^{(n+s-1)},b_1^{(m+s)},b_1^{(m+s+1)},\dots b_1^e\}$$
$$\tilde{P}'_k:=\{b_k^{(1)},b_k^{(2)},\dots,b_k^{(n-1)},b_1^{(m)},b_1^{(m+1)},\dots,b_1^{(m+s-1)},b_k^{(n+s)},b_k^{(n+s+1)},\dots,b_k^{(e)}\}$$
Thus we exchange an interval of $s$ beads which are ``out of order" between the two quasiperiods $\tilde{P}_1$ and $\tilde{P}_k$. 

We must check that $\tilde{P}_1'$ and $\tilde{P}_k'$ are quasiperiods. It is clear from the construction that the $\beta$-numbers of the successive elements of $\tilde{P}_1'$ are successive elements of $\N$, and likewise for $\tilde{P}_k'$. 
Thus it suffices to check that $j(b_1^{(m-1)})\geq j(b_k^{(n)})$, $j(b_k^{(n-1)})\geq j(b_1^{(m)})$, $j(b_k^{(n+s-1)})\geq j(b_1^{(m+s)})$, and $j(b_1^{(m+s-1)})\geq j(b_k^{(n+s)})$. We have $j(b_1^{(m-1)})\geq j(b_1^{(m)})>j(b_k^{(n)})$ where the first inequality is because $\tilde{P}_1$ is a quasiperiod and the second inequality is by the definitions of $m,n$; likewise $j(b_k^{(n-1)})> j(b_1^{(m-1)})\geq j(b_1^{(m)})$, where the first inequality is because $\beta(b_k^{(n-1)})=\beta(b_1^{(m-1)})=\beta+1$ and by our assumption on the maximality of $(\beta,j)$, for all $b\in\tilde{P}_1$ with $\beta(b)>\beta$ there is no bead $b'\in\cA$ with $\beta(b')=\beta(b)$ and $j(b')<j(b)$. The remaining two inequalities are checked similarly.

It follows that $\tilde{P}_1',\tilde{P}_2,\dots,\tilde{P}_{k-1},\tilde{P}_k',\tilde{P}_{k+1},\dots$ is a partitioning 
of the beads of $\cA$ into quasiperiods, of which $\tilde{P}_1'$ is the maximal quasiperiod since it contains the maximal bead $b_1^{(1)}$ of $\cA$. Moreover, if $b$ is among the first $m$ beads of $\tilde{P}_1'$, $b$ satisfies the condition that if $b'\in\cA$, $b'\neq b$ and $\beta(b')=\beta(b)$, then $j(b')>j(b)$. Now iterate the procedure on $\tilde{P}_1'$ with respect to 
this new partitioning, and so on. With each iteration, the number $M$ such that the first $M$ beads of the maximal quasiperiod belong to the minimal row among beads in that column, increases by at least $1$. Thus the process terminates after at most $e-1$ iterations with the construction of a partitioning of the beads of $\cA$ into quasiperiods such that the maximal quasiperiod is a period, $P_1$. Moreover, by construction $\cA\setminus P_1$ is totally quasiperiodic.

Iterating on $\cA\setminus P_1$ to construct $P_2$,
and so on, 
it follows that $\cA$ is totally periodic.
\end{proof}

We now define some new combinatorial notions which will be used in Theorem \ref{thmedge}:
\begin{definition}\label{defperiods} Let $\cA$ be an $\ell$-abacus.
\begin{enumerate}
 \item\label{fore} The \textit{first fore period} $P_1$ is the largest quasiperiod of $\cA$. Inductively, the \textit{$k$'th fore period} $P_k$ is the largest quasiperiod of $\cA$ satisfying $P_k<P_{k-1}$ and $P_k\cap P_i=\emp$ for all $i<k$.
\item A bead of $\cA$ which does not belong to any fore period of $\cA$ is called a \textit{free bead}.
\item\label{vessel} 
For each $k\in\N$, let $P_k$ be the $k$'th fore period of $\cA$ and $b_k^{(i)}$ be the $i$'th bead of $P_k$.
The $k$'th \textit{vessel} $V_k$ of $\cA$ is the union of $P_k$ and the set of free beads $b=(\be,j)$ of $\cA$, $\beta(b_{k+1}^{(e)})\leq\be\leq\beta(b_k^{(1)})$, satisfying: 
\begin{enumerate}
 \item if $\beta=\beta(b_k^{(i)})$ for some $i\in\{1,\dots,e\}$ then $j>j(b_k^{(i)})$
 \item if $\beta=\beta(b_{k+1}^{(i)})$ for some $i\in\{1,\dots,e\}$ then $j<j(b_{k+1}^{(i)})$
 \item if $\beta<\beta(b_k^{(e)})$ then there exists $j'$ such that $(\beta+1,j')\in V_k$.
\end{enumerate}
Informally, we think of these conditions as saying that $b$ is ``between" $P_k$ and $P_{k+1}$ and $V_k$ is ``connected."
\item The $k$'th \textit{aft period} $Q_k$ is the minimal quasiperiod in the $k$'th vessel of $\cA$. 
\item A bead of $\cA$ that does not belong to a vessel is called \textit{adrift}.
\end{enumerate}
\end{definition}
Definition \ref{defperiods} (\ref{fore}) adapts \cite[Definition 5.5]{Gerber2016a}, which should read $(\cup_{i=1}^{k-1}X_i)\cap X=\emp$ rather than $X_{k-1}\cap X=\emp$.

\begin{remark}\label{remforeaft}
Note that fore and aft periods are not, in general, periods in the sense of \cite{JaconLecouvey2012}.
When an abacus is totally $e$-periodic, there are no free beads, and thus the fore periods and the aft periods coincide. 
In this case they also coincide with Definition \ref{defperiodJL}, and
we will simply refer to them as the \textit{periods} of $\cA$.
\end{remark}

\begin{example}\label{rank4example}
Let $e=3$ and $\ell=4$. The fore periods of $\cA$ are drawn in red, the aft periods of $\cA$ are drawn in blue when they differ from the fore periods. The first vessel consists of the rightmost three beads linked in red which form $P_1$, the rightmost three beads linked in blue which form $Q_1$, and two more beads -- the free bead one unit below and left of $P_1$, and the free bead above the middle of $Q_1$. There are two beads that are adrift, below and right of $P_1$.
$$\dots\quad
\TikZ{[scale=.5]
\draw
(11,3)node[fill,circle,inner sep=.5pt]{}
(10,3)node[fill,circle,inner sep=.5pt]{}
(9,3)node[fill,circle,inner sep=3pt]{}
(8,3)node[fill,circle,inner sep=.5pt]{}
(7,3)node[fill,circle,inner sep=3pt]{}
(6,3)node[fill,circle,inner sep=.5pt]{}
(5,3)node[fill,circle,inner sep=3pt]{}
(4,3)node[fill,circle,inner sep=3pt]{}
(3,3)node[fill,circle,inner sep=.5pt]{}
(2,3)node[fill,circle,inner sep=.5pt]{}
(1,3)node[fill,circle,inner sep=3pt]{}
(0,3)node[fill,circle,inner sep=.5pt]{}
(-1,3)node[fill,circle,inner sep=.5pt]{}
(-2,3)node[fill,circle,inner sep=3pt]{}
(-3,3)node[fill,circle,inner sep=.5pt]{}
(-4,3)node[fill,circle,inner sep=.5pt]{}
(-5,3)node[fill,circle,inner sep=.5pt]{}
(-6,3)node[fill,circle,inner sep=.5pt]{}
(-7,3)node[fill,circle,inner sep=3pt]{}
(-8,3)node[fill,circle,inner sep=3pt]{}
(11,2)node[fill,circle,inner sep=.5pt]{}
(10,2)node[fill,circle,inner sep=.5pt]{}
(9,2)node[fill,circle,inner sep=.5pt]{}
(8,2)node[fill,circle,inner sep=3pt]{}
(7,2)node[fill,circle,inner sep=3pt]{}
(6,2)node[fill,circle,inner sep=3pt]{}
(5,2)node[fill,circle,inner sep=.5pt]{}
(4,2)node[fill,circle,inner sep=.5pt]{}
(3,2)node[fill,circle,inner sep=3pt]{}
(2,2)node[fill,circle,inner sep=3pt]{}
(1,2)node[fill,circle,inner sep=3pt]{}
(0,2)node[fill,circle,inner sep=3pt]{}
(-1,2)node[fill,circle,inner sep=.5pt]{}
(-2,2)node[fill,circle,inner sep=3pt]{}
(-3,2)node[fill,circle,inner sep=3pt]{}
(-4,2)node[fill,circle,inner sep=3pt]{}
(-5,2)node[fill,circle,inner sep=.5pt]{}
(-6,2)node[fill,circle,inner sep=.5pt]{}
(-7,2)node[fill,circle,inner sep=3pt]{}
(-8,2)node[fill,circle,inner sep=3pt]{}
(11,1)node[fill,circle,inner sep=.5pt]{}
(10,1)node[fill,circle,inner sep=.5pt]{}
(9,1)node[fill,circle,inner sep=3pt]{}
(8,1)node[fill,circle,inner sep=.5pt]{}
(7,1)node[fill,circle,inner sep=3pt]{}
(6,1)node[fill,circle,inner sep=.5pt]{}
(5,1)node[fill,circle,inner sep=.5pt]{}
(4,1)node[fill,circle,inner sep=.5pt]{}
(3,1)node[fill,circle,inner sep=.5pt]{}
(2,1)node[fill,circle,inner sep=3pt]{}
(1,1)node[fill,circle,inner sep=.5pt]{}
(0,1)node[fill,circle,inner sep=3pt]{}
(-1,1)node[fill,circle,inner sep=3pt]{}
(-2,1)node[fill,circle,inner sep=.5pt]{}
(-3,1)node[fill,circle,inner sep=.5pt]{}
(-4,1)node[fill,circle,inner sep=.5pt]{}
(-5,1)node[fill,circle,inner sep=3pt]{}
(-6,1)node[fill,circle,inner sep=3pt]{}
(-7,1)node[fill,circle,inner sep=3pt]{}
(-8,1)node[fill,circle,inner sep=3pt]{}
(11,0)node[fill,circle,inner sep=.5pt]{}
(10,0)node[fill,circle,inner sep=3pt]{}
(9,0)node[fill,circle,inner sep=.5pt]{}
(8,0)node[fill,circle,inner sep=.5pt]{}
(7,0)node[fill,circle,inner sep=.5pt]{}
(6,0)node[fill,circle,inner sep=3pt]{}
(5,0)node[fill,circle,inner sep=.5pt]{}
(4,0)node[fill,circle,inner sep=3pt]{}
(3,0)node[fill,circle,inner sep=3pt]{}
(2,0)node[fill,circle,inner sep=.5pt]{}
(1,0)node[fill,circle,inner sep=.5pt]{}
(0,0)node[fill,circle,inner sep=.5pt]{}
(-1,0)node[fill,circle,inner sep=.5pt]{}
(-2,0)node[fill,circle,inner sep=3pt]{}
(-3,0)node[fill,circle,inner sep=.5pt]{}
(-4,0)node[fill,circle,inner sep=.5pt]{}
(-5,0)node[fill,circle,inner sep=3pt]{}
(-6,0)node[fill,circle,inner sep=3pt]{}
(-7,0)node[fill,circle,inner sep=3pt]{}
(-8,0)node[fill,circle,inner sep=3pt]{}
;

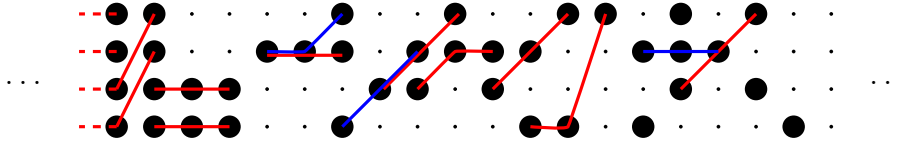
\draw[very thick,red] plot [smooth,tension=.1] coordinates{(9,3)(8,2)(7,1)};
\draw[very thick,blue] plot [smooth,tension=.1] coordinates{(8,2)(7,2)(6,2)};
\draw[very thick,red] plot [smooth,tension=.1] coordinates{(5,3)(4,0)(3,0)};
\draw[very thick,red] plot [smooth,tension=.1] coordinates{(4,3)(3,2)(2,1)};
\draw[very thick,red] plot [smooth,tension=.1] coordinates{(2,2)(1,2)(0,1)};
\draw[very thick,red] plot [smooth,tension=.1] coordinates{(1.1,3)(0.1,2)(-0.9,1)};
\draw[very thick,blue] plot [smooth,tension=.1] coordinates{(0,2)(-1,1)(-2,0)};
\draw[very thick,red] plot [smooth,tension=.1] coordinates{(-2,1.9)(-3,1.9)(-4,1.9)};
\draw[very thick,blue] plot [smooth,tension=.1] coordinates{(-2,3)(-3,2)(-4,2)};
\draw[very thick,red] plot [smooth,tension=.1] coordinates{(-5,1)(-6,1)(-7,1)};
\draw[very thick,red] plot [smooth,tension=.1] coordinates{(-5,0)(-6,0)(-7,0)};
\draw[very thick,red] plot [smooth,tension=.1] coordinates{(-7,2)(-8,0)};
\draw[very thick,red, dashed] plot [smooth,tension=.1] coordinates{(-8,0)(-9,0)};
\draw[very thick,red] plot [smooth,tension=.1] coordinates{(-7,3)(-8,1)};
\draw[very thick,red, dashed] plot [smooth,tension=.1] coordinates{(-8,1)(-9,1)};
\draw[very thick,red, dashed] plot [smooth,tension=.1] coordinates{(-8,2)(-9,2)};
\draw[very thick,red, dashed] plot [smooth,tension=.1] coordinates{(-8,3)(-9,3)};

}\quad\dots
$$
\end{example}

\begin{lemma}\label{lemvessels}
Let $V_k$ and $V_{k'}$ be vessels in an abacus $\cA$, $k\neq k'$. Then $V_k\cap V_{k'} = \emp$.
\end{lemma}

\begin{proof}
Suppose by way of contradiction that $b=(\be,j)\in V_k\cap V_{k'}$ is a bead in the intersection of two different vessels. Since $P_k\cap P_{k'}=\emp$, and $V_k\setminus P_k$ consists of free beads, likewise $V_{k'}\setminus P_{k'}$, $b$ must be a free bead.
Without loss of generality, $k<k'$, so $P_{k'}<P_k$. We then have 
$$\beta(b_{k'}^{(e)}) \leq \beta(b_{k+1}^{(e)}) \leq \beta \leq \beta(b_{k'}^{(1)}) \leq \beta(b_{k+1}^{(1)}) $$
so $\beta=\beta(b_{k'}^{(i')})$ for some $i'\in\{1,\dots,e\}$ and also $\beta=\beta(b_{k+1}^{(i'')})$ for some $i''\in\{1,\dots,e\}$. On the one hand, Definition \ref{defperiods}(\ref{vessel})(b) then implies that $j<j(b_{k+1}^{(i'')})$, and on the other hand, Definition \ref{defperiods}(\ref{vessel})(a) implies that $j>j(b_{k'}^{(i')})$. But $j(b_{k+1}^{(i'')})\leq j(b_{k'}^{(i')})$ (with equality if and only if $k+1=k'$). So $j<j(b_{k+1}^{(i'')})\leq j(b_{k'}^{(i')})<j$, a contradiction.
\end{proof}

 \begin{definition} Let $P$ be a quasiperiod in $\cA$. Let $j_1,j_2,\dots j_s$ 
 be the distinct $j$'s such that $(\beta,j)\in P$ for some $\beta\in\Z$. For each $i=1,\dots,s$, let $\beta_i=\max\{\beta\;|\;(\beta,j_i)\in P\}$. If $(\beta_i+1,j_i)\notin \cA$ for all $i=1,\dots,s$ then we may define the \textit{right shift} of $P$ as $$\tilde{P}=\{(\beta+1,j)\;|\;(\beta,j)\in P \}$$ and we say that $P$ is \textit{right-shiftable}. Let $P$ be a right-shiftable quasiperiod. We say that we \textit{shift $P$ one unit to the right} if we replace $\cA$ with $(\cA\setminus P)\cup \tilde{P}$. The left versions of these definitions are similar.
 \end{definition}
 
 \begin{example}
 Let $|\bla,\mbs\rangle=|((1),(1^2)),(0,1)\rangle$ and $e=3$. Let $P=P_1$. Then  $P$ is both right- and left-shiftable. Here is $P$ being shifted to the right:
 $$
 \TikZ{[scale=.5]
 \draw
 (0,0)node[fill,circle,inner sep=3pt]{}
 (1,0)node[fill,circle,inner sep=3pt]{}
 (2,0)node[fill,circle,inner sep=.5pt]{}
 (3,0)node[fill,circle,inner sep=3pt]{}
 (4,0)node[fill,circle,inner sep=.5pt]{}
 (5,0)node[fill,circle,inner sep=.5pt]{}
 (6,0)node[fill,circle,inner sep=.5pt]{}
 (0,1)node[fill,circle,inner sep=3pt]{}
 (1,1)node[fill,circle,inner sep=3pt]{}
 (2,1)node[fill,circle,inner sep=.5pt]{}
 (3,1)node[fill,circle,inner sep=.5pt]{}
 (4,1)node[fill,circle,inner sep=3pt]{}
 (5,1)node[fill,circle,inner sep=3pt]{}
 (6,1)node[fill,circle,inner sep=.5pt]{}
 ;}
 \quad
 \longrightarrow
 \quad
 \TikZ{[scale=.5]
 \draw
 (0,0)node[fill,circle,inner sep=3pt]{}
 (1,0)node[fill,circle,inner sep=3pt]{}
 (2,0)node[fill,circle,inner sep=.5pt]{}
 (3,0)node[fill,circle,inner sep=.5pt]{}
 (4,0)node[fill,circle,inner sep=3pt]{}
 (5,0)node[fill,circle,inner sep=.5pt]{}
 (6,0)node[fill,circle,inner sep=.5pt]{}
 (0,1)node[fill,circle,inner sep=3pt]{}
 (1,1)node[fill,circle,inner sep=3pt]{}
 (2,1)node[fill,circle,inner sep=.5pt]{}
 (3,1)node[fill,circle,inner sep=.5pt]{}
 (4,1)node[fill,circle,inner sep=.5pt]{}
 (5,1)node[fill,circle,inner sep=3pt]{}
 (6,1)node[fill,circle,inner sep=3pt]{}
 ;}
 $$
 Here is $P$ being shifted to the left:
 $$
 \TikZ{[scale=.5]
 \draw
 (0,0)node[fill,circle,inner sep=3pt]{}
 (1,0)node[fill,circle,inner sep=3pt]{}
 (2,0)node[fill,circle,inner sep=.5pt]{}
 (3,0)node[fill,circle,inner sep=3pt]{}
 (4,0)node[fill,circle,inner sep=.5pt]{}
 (5,0)node[fill,circle,inner sep=.5pt]{}
 (6,0)node[fill,circle,inner sep=.5pt]{}
 (0,1)node[fill,circle,inner sep=3pt]{}
 (1,1)node[fill,circle,inner sep=3pt]{}
 (2,1)node[fill,circle,inner sep=.5pt]{}
 (3,1)node[fill,circle,inner sep=.5pt]{}
 (4,1)node[fill,circle,inner sep=3pt]{}
 (5,1)node[fill,circle,inner sep=3pt]{}
 (6,1)node[fill,circle,inner sep=.5pt]{}
 ;}
 \quad
 \longrightarrow
 \quad
 \TikZ{[scale=.5]
 \draw
 (0,0)node[fill,circle,inner sep=3pt]{}
 (1,0)node[fill,circle,inner sep=3pt]{}
 (2,0)node[fill,circle,inner sep=3pt]{}
 (3,0)node[fill,circle,inner sep=.5pt]{}
 (4,0)node[fill,circle,inner sep=.5pt]{}
 (5,0)node[fill,circle,inner sep=.5pt]{}
 (6,0)node[fill,circle,inner sep=.5pt]{}
 (0,1)node[fill,circle,inner sep=3pt]{}
 (1,1)node[fill,circle,inner sep=3pt]{}
 (2,1)node[fill,circle,inner sep=.5pt]{}
 (3,1)node[fill,circle,inner sep=3pt]{}
 (4,1)node[fill,circle,inner sep=3pt]{}
 (5,1)node[fill,circle,inner sep=.5pt]{}
 (6,1)node[fill,circle,inner sep=.5pt]{}
 ;}
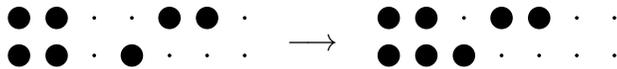
 $$
 \end{example}
 
 \begin{lemma}\label{shiftsandvessels}
For an abacus $\cA$ and $k\in\N$, let $P_k$ be the $k$'th fore period of $\cA$, and assume $P_k$ is right-shiftable. 
Let $\tilde{P}_k$ be the right shift of $P_k$ and let $V_k$ be the $k$'th vessel of $\cA$.
\begin{enumerate}
 \item $(V_k\setminus P_k)\cup \tilde{P_k}$ contains a unique quasiperiod, which is equal to $\tilde{P}_k$.
\item Let $\cA'= (\cA\setminus P_k) \cup \tilde{P}_k$ be the abacus obtained by shifting $P_k$ one unit to the right.
For $a\in\N$, denote $P'_a$ (respectively $Q'_a$) the $a$'th fore (respectively aft) periods of $\cA'$.
Then
$\tilde{P_k}=Q_k'$ if and only if $\tilde{P}_{k}\cap P_a'=\emp$ for all $1\leq a\leq k-1$.
\end{enumerate}
\end{lemma}

\begin{proof}
\begin{enumerate}
 \item Let $Q\subset((V_k\setminus P_k)\cup\tilde{P}_k)$ be a quasiperiod. Definition \ref{defperiods} implies that $Q$ is not contained in $V_k\setminus P_k$. Let $b=(\beta,j)\in Q\cap \tilde{P}_k$. Say $b$ is the $i$'th bead of $Q$ and the $m$'th bead of $\tilde{P}_k$. Consider what the $i+1$'st bead of $Q$ can then be (if $i<e$): by the definition of quasiperiod, we know it is $(\beta-1,j')$ for some $j'\leq j$, and that $(\beta,j)\in\tilde{P}_k$. By Definition \ref{defperiods}(\ref{vessel}), note that if $(\alpha,h)\in P_k$, then $(\alpha,h'')\notin V_k$ for $h''<h$. This implies: (a) if $(\alpha,h)$ and $(\alpha-1,h')$ are successive beads of $P_k$, and so $(\alpha+1,h)$ and $(\alpha,h')$ are successive beads of $\tilde{P}_k$, then $(\alpha,h'')\in (V_k\setminus P_k)\cup\tilde{P}_k$ with $h''\leq h$ if and only if $h''=h'$; and (b) if $(\alpha,h)$ is the $e$'th bead of $P_k$ then $(\alpha, h')\notin (V_k\setminus P_k)\cup\tilde{P}_k$ for all $h'\leq h$.  Therefore $(\beta-1,j')$, the $i+1$'st bead of $Q$, must be the $m+1$'st bead of $\tilde{P}_k$. Iterating this argument, it follows that once $Q$ meets $\tilde{P}_k$ in a bead $b$, then all beads of $Q$ which are smaller than $b$ coincide with beads of $\tilde{P}_k$. Since both $Q$ and $\tilde{P}_k$ consist of $e$ beads with consecutive $\beta$-numbers, the first bead of $Q$ must then have at least as big a $\beta$-number as the first bead of $\tilde{P}_k$. But by Definition \ref{defperiods}(\ref{vessel})(a), the maximum $\beta$-number of a bead in $V_k$ is the $\beta$-number of the first bead of $P_k$. The only bead of $Q\subset((V_k\setminus P_k)\cup\tilde{P}_k)$ with $\beta$-number equal to that of the first bead of $\tilde{P}_k$ is the first bead of $\tilde{P}_k$ itself. So $Q$ and $\tilde{P}_k$ have the same first bead, and therefore they completely coincide.
 \item Suppose $\tilde{P}_{k}\cap P_a'=\emp$ for all $1\leq a\leq k-1$.
Since $\cA'$ is obtained from $\cA$ by shifting its $k$'th fore period $P_k$ to the right and this does not affect the fore periods smaller than $P_k$, we have
$P_a=P'_a$ and $\tilde{P}_k\cap P_a'=\emp$ for all $a\geq k+1$.
Any quasiperiod in an abacus must intersect some fore period,
so it must hold that $\tilde{P}_k\cap P'_k\neq\emp$. Moreover, the preceding remarks also imply that $\tilde{P}_k\subset V'_k$. The first bead $b'$ of $Q_k'$ cannot be to the right of the first bead $b$ of $\tilde{P}_k$ by minimality of $Q_k'$ in $V_k'$. If some bead of $Q_k'$ lies in the same column and below a bead of $\tilde{P}_k$ then, since $P_a=P_a'$ and $P_a'\cap Q_k'=\emp$ for all $1\leq a\leq k-1$, we'd be able to construct a larger quasiperiod than $P_k$ in $\cA$ which doesn't intersect $P_a$ for all $1\leq a\leq k-1$, contradicting that $P_k$ is the $k$'th fore period of $\cA$. If $b'$ is in the same column as $b$ and above it, then $(b',b_k^{(1)},b_k^{(2)},\dots,b_k^{(e-1)})$ would be the $k$'th fore period of $\cA$ instead of $P_k$. It follows that $Q_k'$ belongs to $(V_k\setminus P_k)\cup \tilde{P_k}$, but by Part (1), the only quasiperiod of $(V_k\setminus P_k)\cup \tilde{P_k}$ is $\tilde{P}_k$. Therefore $\tilde{P_k}= Q_k'$.

Conversely, suppose that $\tilde{P}_k=Q_k'$. 
Then $\tilde{P}_k$ is a subset of $V_k'$. 
The vessels are disjoint by Lemma \ref{lemvessels}, and $P'_a$ is a subset of $V_a'$ for all $a$, 
so $\tilde{P}_k\cap P'_a=\emp$ for all $a\neq k$.
\end{enumerate}

\end{proof}

\begin{remark} Note that taking $k=1$ in Lemma \ref{shiftsandvessels} (2), it follows that $\tilde{P}_1=Q_1'$ always.
\end{remark}

\subsection{Arrows in the $\slinf$-crystal}\label{subsectionedge} 
Our main theorem allows us to 
construct
the entire connected $\slinf$-crystal component of any $\ell$-abacus $\cA$ starting from nothing but knowledge of $\cA$ itself (and the specification of $e$, of course), using a rule 
for computing all incoming and outgoing arrows
that is explicit and easy to use. 
We indicate direction of motion in the $\slinf$-crystal moving away from a source by \textit{traveling downstream}, and moving towards a source by \textit{traveling upstream.}

Recall from Section \ref{Scrystal} that there is a graph isomorphism 
from a connected component of the $\slinf$-crystal on $\mathcal{F}_\mbs$ with source $\cA^\circ$ to the Young graph, sending $\cA=\tilde{b}_\theta\cA^\circ$ to $\theta$ and sending an arrow $\tilde{b}_\theta\cA^\circ\xrightarrow{m}\tilde{b}_\rho\cA^\circ$, $m\in\Z$, to the arrow that adds a box of content $m$ to $\theta$ to obtain $\rho$.

\begin{definition}\label{Ups_k} For $k\in\N$ and a connected component of the Fock space $\mathcal{F}_\mbs$, let $\Upsilon_k^+$ be the operator which corresponds under the graph isomorphism with the Young graph to the operator on partitions which acts on a partition $\theta$ by adding a box to the $k$'th row of $\theta$ if the number of nonzero parts of $\theta$ is at least $k-1$, and acts by $0$ if the number of parts of $\theta$ is less than $k-1$. Likewise, let $\Upsilon_k^-$ be the operator which corresponds under the graph isomorphism with the Young graph to the operator on partitions which acts on a partition $\theta$ by removing a box from row $k$ of $\theta$ if $\theta$ has at least $k$ nonzero parts, and acts by $0$ otherwise.
\end{definition}
Thus $\Upsilon_k^+$ acts on $\cA=\tilde{b}_\theta\cA^\circ$ by $\tb_{\theta_k+1-k}^+$ and $\Upsilon_k^-$ acts on $\cA$ by $\tb_{\theta_k-k}$, and the data of $\Upsilon_k^\pm$, $k\in\N$ is equivalent to the data of $\tb_m^\pm$, $m\in\Z$.

The following theorem is a direct analogue of Theorem \ref{thmchapeaucrystal}.
It gives the action of the maps $\Upsilon_k^\pm$ for $k\geq1$.
Using these, starting from any $\cA$, one can construct the entire connected component of the $\slinf$-crystal containing $\cA$.
In particular, one can recover its source vertex.

\begin{theorem}\label{thmedge}
Let $\cA$ and $\cA'$ be $\ell$-abaci.
There is an arrow $\cA\to\cA'$ in the $\slinf$-crystal if and only if the following equivalent situations hold for some $k\in\N$:
\begin{enumerate}
\item (\textit{traveling downstream}) 
$\cA'$ is obtained from $\cA$ by shifting the $k$'th fore period $P_k$ of $\cA$ one unit to the right, and the shift $\tilde{P_k}$ of $P_k$ is equal to $Q_k'$, the $k$'th aft period of $\cA'$. 
\item (\textit{traveling upstream})
$\cA$ is obtained from $\cA'$ by shifting the $k$'th aft period $Q_k'$ of $\cA'$ one unit to the left, and the shift $\tilde{Q}_k'$ of $Q_k'$ is equal to $P_k$, the $k$'th fore period of $\cA$. 
\end{enumerate}
In this case, we have $\cA' = \Upsilon_k^+ \cA$, or equivalently $\cA = \Upsilon_k^-\cA'$.
\end{theorem}

\begin{proof}
It is easily seen that (1) holds if and only if (2) holds, so we will prove that there is an arrow $\cA\to\cA'$ in the $\slinf$-crystal if and only if (1) holds.

First we argue that arrows in the $\slinf$-crystal are given by shifting fore periods to the right. From Sections \ref{Scrystal} and \ref{amu}, we know that $\cA=\tilde{b}_\theta\cA^\circ$ for a unique partition $\theta=(\theta_1,\theta_2,\dots)$ and a unique highest weight vertex $\cA^\circ$ of the $\slinf$-crystal, that $\cA'$ is in the connected component of $\cA$ if and only if $\cA'=\tilde{b}_\sigma\cA^\circ$ for a unique partition $\sigma$, and that the connected component of $\cA^\circ$ is isomorphic to the Young graph. This implies that $\cA'=\Upsilon_k^+\cA$ if and only if $\cA'=\tilde{b}_{\theta\cup\ga}\cA^\circ$ where $\ga$ is an addable box in row $k$ of $\theta$. Now we recall how $\tilde{b}_\theta$ was defined in \cite{Gerber2016a}: this operator recursively adds the $k$'th good vertical $e$-strip $\theta_k$ times to a charged multipartition, for each $k\geq 1$ (Theorem \ref{thmverticalstrips}).
Adding the $k$'th good addable strip corresponds to shifting the $k$'th fore period one unit to the right in the abacus, since adding a box to the diagram corresponds to shifting a bead to the right in the abacus, and the fore periods correspond to the good addable strips (Definition \ref{defperiods}). Therefore, the vertex $\cA$ has an outgoing arrow given by shifting the $k$'th fore period one unit to the right for each $k$ such that $\theta$ has an addable box in row $k$. 

Now we are ready to prove that if there is an arrow $\cA\to\cA'$ in the $\slinf$-crystal then (1) holds. By the previous paragraph, $\cA\to\cA'$ implies that $\cA'$ is obtained from $\cA$ by shifting the $k$'th fore period $P_k$ of $\cA$ to the right for some $k\in\N$. Write $\cA=\tilde{b}_\theta\cA^\circ$ and $\cA'=\Upsilon_k^+\cA=\tilde{b}_{\theta\cup\ga}\cA^\circ$ where $\ga$ is an addable box in row $k$ of $\theta$. 
If $k=1$ then by Lemma \ref{shiftsandvessels}(2), $\tilde{P}_1=Q_1'$. 
So suppose $k>1$. 
By Lemma \ref{shiftsandvessels} (2),
it suffices to show that $\tilde{P}_{k}\cap P_a'=\emp$ for all $1\leq a\leq k-1$.
By contradiction, suppose that there exists a bead $b=(\beta,j)$ in $\tilde{P}_{k}\cap P_c'$ for some $1\leq c\leq k-1$. We may choose the bead $b$ so that $(\beta-1,j)\notin\tilde{P}_k\cap P_a'$ for any $1\leq a\leq k-1$. Since $\tilde{P}_k$ is the right shift of $P_k$, either $(\beta-1,j)\in\tilde{P}_k$ or $(\beta-1,j)\notin\cA'$. So $(\beta-1,j)\notin P_a'$ for all $1\leq a\leq k-1$.
Let $\sigma=\theta+(1^k)=\theta\cup \ga+(1^{k-1})=(\theta_1+1,\theta_2+1,\dots,\theta_k+1,\theta_{k+1},\theta_{k+2},\dots)$. Now consider the position $(\beta,j)$ in the abacus $\tilde{b}_\sigma\cA^\circ$.
On the one hand, computing first $\tilde{b}_{\theta\cup \ga}\cA^\circ=\cA'$ and then applying $\tilde{b}_{(1^{k-1})}$, we observe that $b\in P_c'$ 
is shifted one unit to the right in $\tilde{b}_{(1^{k-1})}\cA'$, but since $(\beta-1,j)\notin P_a'$ for all $1\leq a\leq k-1$, then $(\beta,j)\notin (\cA'\setminus(P_1'\cup\dots\cup P_{k-1}'))\cup \tilde{P}_1'\cup\dots\cup\tilde{P}_{k-1}'=\tilde{b}_{(1^{k-1})}\cA'=
\tilde{b}_\sigma\cA^\circ$.
On the other hand, computing $\tilde{b}_{\theta+(1^{k-1})}\cA^\circ$ first and then applying $\Upsilon_{k}^+$, 
we observe that in $\tilde{b}_{\theta+(1^{k-1})}\cA^\circ=\tilde{b}_{(1^{k-1})}(\tilde{b}_\theta \cA^\circ)$, the $k$'th period has not shifted the $\theta_{k}+1$'st time yet and coincides with the $k$'th period $P_k$ of $\tilde{b}_\theta \cA^\circ$; upon applying $\Upsilon_k^+$, we have $(\beta,j)\in\tilde{P}_{k}\subset \Upsilon^+_k \tilde{b}_{\theta+(1^{k-1})}\cA^\circ =\tilde{b}_\sigma\cA^\circ$.
This is a contradiction, and therefore $\tilde{P}_{k}\cap P_a'=\emp$ for all $1\leq a \leq k-1$.

Conversely, suppose that (1) holds, i.e. suppose that $P_k\subset\cA$ can be shifted to the right yielding $\cA'$, and that the right shift $\tilde{P}_k$ of $P_k$ is $Q_k'$, the $k$'th aft period of $\cA'$. Let $\cA''=\tilde{b}_{(1^k)}\cA$. By Lemma \ref{shiftsandvessels}(2), $\tilde{P}_k\cap P_a'=\emp$ for all $1\leq a<k$. Thus $\tilde{b}_{(1^{k-1})}\cA'=\cA''$. Let $\tilde{b}_{(1^{k-1})}^-$ be the inverse map to $\tilde{b}_{(1^{k-1})}$. Then $\cA'=\tilde{b}_{(1^{k-1})}^-(\tilde{b}_{(1^k)}\cA)$ and therefore $\cA'=\Upsilon_k^+\cA$.
\end{proof}

\begin{remark}
In particular, one recovers \cite[Theorem 5.11]{Gerber2016a} by expressing the operators $\tb_\si$ as compositions of operators $\tb^+_k$ (or equivalently, operators $\Upsilon^+_k$).
\end{remark}

\begin{example} In Example \ref{rank4example}, the only fore periods which can travel downstream are $P_1$, $P_2$, and $P_6$. The only aft periods which can travel upstream are $Q_1$ and $Q_5$.
\end{example}

\begin{example}
Let $\ell=3$, $e=5$, $\lambda=((4,1),(7,2),(2,1^5))$, and $\mbs=(3,-3,1)$. We have its abacus $\cA=\cA(\bla,\bs)$:
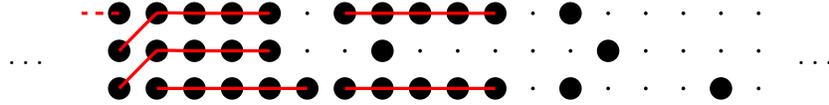
$$\dots\quad\TikZ{[scale=.5]
\draw
(8,0)node[fill,circle,inner sep=.5pt]{}
(8,1)node[fill,circle,inner sep=.5pt]{}
(8,2)node[fill,circle,inner sep=.5pt]{}
(7,0)node[fill,circle,inner sep=3pt]{}
(7,1)node[fill,circle,inner sep=.5pt]{}
(7,2)node[fill,circle,inner sep=.5pt]{}
(6,0)node[fill,circle,inner sep=.5pt]{}
(6,1)node[fill,circle,inner sep=.5pt]{}
(6,2)node[fill,circle,inner sep=.5pt]{}
(5,0)node[fill,circle,inner sep=.5pt]{}
(5,1)node[fill,circle,inner sep=.5pt]{}
(5,2)node[fill,circle,inner sep=.5pt]{}
(4,0)node[fill,circle,inner sep=.5pt]{}
(4,1)node[fill,circle,inner sep=3pt]{}
(4,2)node[fill,circle,inner sep=.5pt]{}
(3,0)node[fill,circle,inner sep=3pt]{}
(3,1)node[fill,circle,inner sep=.5pt]{}
(3,2)node[fill,circle,inner sep=3pt]{}
(2,0)node[fill,circle,inner sep=.5pt]{}
(2,1)node[fill,circle,inner sep=.5pt]{}
(2,2)node[fill,circle,inner sep=.5pt]{}
(1,0)node[fill,circle,inner sep=3pt]{}
(1,1)node[fill,circle,inner sep=.5pt]{}
(1,2)node[fill,circle,inner sep=3pt]{}
(0,0)node[fill,circle,inner sep=3pt]{}
(0,1)node[fill,circle,inner sep=.5pt]{}
(0,2)node[fill,circle,inner sep=3pt]{}
(-1,0)node[fill,circle,inner sep=3pt]{}
(-1,1)node[fill,circle,inner sep=.5pt]{}
(-1,2)node[fill,circle,inner sep=3pt]{}
(-2,0)node[fill,circle,inner sep=3pt]{}
(-2,1)node[fill,circle,inner sep=3pt]{}
(-2,2)node[fill,circle,inner sep=3pt]{}
(-3,0)node[fill,circle,inner sep=3pt]{}
(-3,1)node[fill,circle,inner sep=.5pt]{}
(-3,2)node[fill,circle,inner sep=3pt]{}
(-4,0)node[fill,circle,inner sep=3pt]{}
(-4,1)node[fill,circle,inner sep=.5pt]{}
(-4,2)node[fill,circle,inner sep=.5pt]{}
(-5,0)node[fill,circle,inner sep=3pt]{}
(-5,1)node[fill,circle,inner sep=3pt]{}
(-5,2)node[fill,circle,inner sep=3pt]{}
(-6,0)node[fill,circle,inner sep=3pt]{}
(-6,1)node[fill,circle,inner sep=3pt]{}
(-6,2)node[fill,circle,inner sep=3pt]{}
(-7,0)node[fill,circle,inner sep=3pt]{}
(-7,1)node[fill,circle,inner sep=3pt]{}
(-7,2)node[fill,circle,inner sep=3pt]{}
(-8,0)node[fill,circle,inner sep=3pt]{}
(-8,1)node[fill,circle,inner sep=3pt]{}
(-8,2)node[fill,circle,inner sep=3pt]{}
(-9,0)node[fill,circle,inner sep=3pt]{}
(-9,1)node[fill,circle,inner sep=3pt]{}
(-9,2)node[fill,circle,inner sep=3pt]{}
;

\draw[very thick,red] plot [smooth,tension=.1] coordinates{(1,2)(0,2)(-1,2)(-2,2)(-3,2)};
\draw[very thick,red] plot [smooth,tension=.1] coordinates{(1,0)(0,0)(-1,0)(-2,0)(-3,0)};
\draw[very thick,red] plot [smooth,tension=.1] coordinates{(-4,0)(-5,0)(-6,0)(-7,0)(-8,0)};
\draw[very thick,red] plot [smooth,tension=.1] coordinates{(-5,1)(-6,1)(-7,1)(-8,1)(-9,0)};
\draw[very thick,red] plot [smooth,tension=.1] coordinates{(-5,2)(-6,2)(-7,2)(-8,2)(-9,1)};
\draw[very thick,red,dashed] plot [smooth,tension=.1] coordinates{(-9,2)(-10,2)};

}\quad\dots
$$
In this example, the aft periods and the fore periods coincide: $P_k=Q_k$ for all $k\in\N$. The only period that has space to move to the left is $Q_2$ in the top row. However, the shift of $Q_2$ one unit to the left, $\tilde{Q}_2$, is not the $2$nd fore period in the resulting abacus, so shifting $Q_2$ to the left is not a move that allows $\cA$ to travel upstream in the crystal. Indeed, here is the abacus $(\cA\setminus Q_2)\cup \tilde{Q}_2$:
$$\dots\quad\TikZ{[scale=.5]
\draw
(8,0)node[fill,circle,inner sep=.5pt]{}
(8,1)node[fill,circle,inner sep=.5pt]{}
(8,2)node[fill,circle,inner sep=.5pt]{}
(7,0)node[fill,circle,inner sep=3pt]{}
(7,1)node[fill,circle,inner sep=.5pt]{}
(7,2)node[fill,circle,inner sep=.5pt]{}
(6,0)node[fill,circle,inner sep=.5pt]{}
(6,1)node[fill,circle,inner sep=.5pt]{}
(6,2)node[fill,circle,inner sep=.5pt]{}
(5,0)node[fill,circle,inner sep=.5pt]{}
(5,1)node[fill,circle,inner sep=.5pt]{}
(5,2)node[fill,circle,inner sep=.5pt]{}
(4,0)node[fill,circle,inner sep=.5pt]{}
(4,1)node[fill,circle,inner sep=3pt]{}
(4,2)node[fill,circle,inner sep=.5pt]{}
(3,0)node[fill,circle,inner sep=3pt]{}
(3,1)node[fill,circle,inner sep=.5pt]{}
(3,2)node[fill,circle,inner sep=3pt]{}
(2,0)node[fill,circle,inner sep=.5pt]{}
(2,1)node[fill,circle,inner sep=.5pt]{}
(2,2)node[fill,circle,inner sep=.5pt]{}
(1,0)node[fill,circle,inner sep=3pt]{}
(1,1)node[fill,circle,inner sep=.5pt]{}
(1,2)node[fill,circle,inner sep=.5pt]{}
(0,0)node[fill,circle,inner sep=3pt]{}
(0,1)node[fill,circle,inner sep=.5pt]{}
(0,2)node[fill,circle,inner sep=3pt]{}
(-1,0)node[fill,circle,inner sep=3pt]{}
(-1,1)node[fill,circle,inner sep=.5pt]{}
(-1,2)node[fill,circle,inner sep=3pt]{}
(-2,0)node[fill,circle,inner sep=3pt]{}
(-2,1)node[fill,circle,inner sep=3pt]{}
(-2,2)node[fill,circle,inner sep=3pt]{}
(-3,0)node[fill,circle,inner sep=3pt]{}
(-3,1)node[fill,circle,inner sep=.5pt]{}
(-3,2)node[fill,circle,inner sep=3pt]{}
(-4,0)node[fill,circle,inner sep=3pt]{}
(-4,1)node[fill,circle,inner sep=.5pt]{}
(-4,2)node[fill,circle,inner sep=3pt]{}
(-5,0)node[fill,circle,inner sep=3pt]{}
(-5,1)node[fill,circle,inner sep=3pt]{}
(-5,2)node[fill,circle,inner sep=3pt]{}
(-6,0)node[fill,circle,inner sep=3pt]{}
(-6,1)node[fill,circle,inner sep=3pt]{}
(-6,2)node[fill,circle,inner sep=3pt]{}
(-7,0)node[fill,circle,inner sep=3pt]{}
(-7,1)node[fill,circle,inner sep=3pt]{}
(-7,2)node[fill,circle,inner sep=3pt]{}
(-8,0)node[fill,circle,inner sep=3pt]{}
(-8,1)node[fill,circle,inner sep=3pt]{}
(-8,2)node[fill,circle,inner sep=3pt]{}
(-9,0)node[fill,circle,inner sep=3pt]{}
(-9,1)node[fill,circle,inner sep=3pt]{}
(-9,2)node[fill,circle,inner sep=3pt]{};

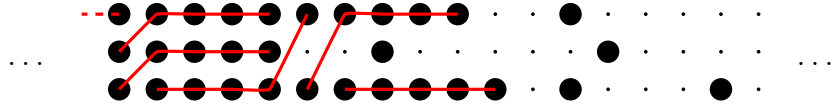
\draw[very thick,red] plot [smooth,tension=.1] coordinates{(1,0)(0,0)(-1,0)(-2,0)(-3,0)};
\draw[very thick,red] plot [smooth,tension=.1] coordinates{(0,2)(-1,2)(-2,2)(-3,2)(-4,0)};
\draw[very thick,red] plot [smooth,tension=.1] coordinates{(-4,2)(-5,0)(-6,0)(-7,0)(-8,0)};
\draw[very thick,red] plot [smooth,tension=.1] coordinates{(-5,1)(-6,1)(-7,1)(-8,1)(-9,0)};
\draw[very thick,red] plot [smooth,tension=.1] coordinates{(-5,2)(-6,2)(-7,2)(-8,2)(-9,1)};
\draw[very thick,red,dashed] plot [smooth,tension=.1] coordinates{(-9,2)(-10,2)};

}\quad\dots
$$
Its second period is not the left shift of $Q_2$ from $\cA$. So no aft period $Q_k$ can travel upstream, and $\cA$ is a highest weight vertex for the $\slinf$-crystal.

The only period of $\cA$ that can travel downstream is $P_1$. Note that this is always the case if $\cA$ is a highest weight vertex for the $\slinf$-crystal.
\end{example}

\begin{example}
Let $\ell=5$, $e=4$, $\bla=((9,2),(5^2,4,3,2,1^2),(2,1^2),(6,4,2,1^2),(4^2,2^2,1^2))$, and $\mbs=(-4,2,-1,2,3)$. 
Let $\cA=\cA(\bla,\mbs)$. 

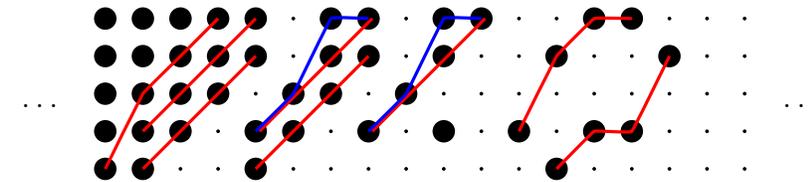
$$\dots\quad\TikZ{[scale=.5]
\draw
(10,0)node[fill,circle,inner sep=.5pt]{}
(10,1)node[fill,circle,inner sep=.5pt]{}
(10,2)node[fill,circle,inner sep=.5pt]{}
(10,3)node[fill,circle,inner sep=.5pt]{}
(10,4)node[fill,circle,inner sep=.5pt]{}
(9,0)node[fill,circle,inner sep=.5pt]{}
(9,1)node[fill,circle,inner sep=.5pt]{}
(9,2)node[fill,circle,inner sep=.5pt]{}
(9,3)node[fill,circle,inner sep=.5pt]{}
(9,4)node[fill,circle,inner sep=.5pt]{}
(8,0)node[fill,circle,inner sep=.5pt]{}
(8,1)node[fill,circle,inner sep=.5pt]{}
(8,2)node[fill,circle,inner sep=.5pt]{}
(8,3)node[fill,circle,inner sep=3pt]{}
(8,4)node[fill,circle,inner sep=.5pt]{}
(7,0)node[fill,circle,inner sep=.5pt]{}
(7,1)node[fill,circle,inner sep=3pt]{}
(7,2)node[fill,circle,inner sep=.5pt]{}
(7,3)node[fill,circle,inner sep=.5pt]{}
(7,4)node[fill,circle,inner sep=3pt]{}
(6,0)node[fill,circle,inner sep=.5pt]{}
(6,1)node[fill,circle,inner sep=3pt]{}
(6,2)node[fill,circle,inner sep=.5pt]{}
(6,3)node[fill,circle,inner sep=.5pt]{}
(6,4)node[fill,circle,inner sep=3pt]{}
(5,0)node[fill,circle,inner sep=3pt]{}
(5,1)node[fill,circle,inner sep=.5pt]{}
(5,2)node[fill,circle,inner sep=.5pt]{}
(5,3)node[fill,circle,inner sep=3pt]{}
(5,4)node[fill,circle,inner sep=.5pt]{}
(4,0)node[fill,circle,inner sep=.5pt]{}
(4,1)node[fill,circle,inner sep=3pt]{}
(4,2)node[fill,circle,inner sep=.5pt]{}
(4,3)node[fill,circle,inner sep=.5pt]{}
(4,4)node[fill,circle,inner sep=.5pt]{}
(3,0)node[fill,circle,inner sep=.5pt]{}
(3,1)node[fill,circle,inner sep=.5pt]{}
(3,2)node[fill,circle,inner sep=.5pt]{}
(3,3)node[fill,circle,inner sep=.5pt]{}
(3,4)node[fill,circle,inner sep=3pt]{}
(2,0)node[fill,circle,inner sep=.5pt]{}
(2,1)node[fill,circle,inner sep=3pt]{}
(2,2)node[fill,circle,inner sep=.5pt]{}
(2,3)node[fill,circle,inner sep=3pt]{}
(2,4)node[fill,circle,inner sep=3pt]{}
(1,0)node[fill,circle,inner sep=.5pt]{}
(1,1)node[fill,circle,inner sep=.5pt]{}
(1,2)node[fill,circle,inner sep=3pt]{}
(1,3)node[fill,circle,inner sep=.5pt]{}
(1,4)node[fill,circle,inner sep=.5pt]{}
(0,0)node[fill,circle,inner sep=.5pt]{}
(0,1)node[fill,circle,inner sep=3pt]{}
(0,2)node[fill,circle,inner sep=.5pt]{}
(0,3)node[fill,circle,inner sep=3pt]{}
(0,4)node[fill,circle,inner sep=3pt]{}
(-1,0)node[fill,circle,inner sep=.5pt]{}
(-1,1)node[fill,circle,inner sep=.5pt]{}
(-1,2)node[fill,circle,inner sep=3pt]{}
(-1,3)node[fill,circle,inner sep=3pt]{}
(-1,4)node[fill,circle,inner sep=3pt]{}
(-2,0)node[fill,circle,inner sep=.5pt]{}
(-2,1)node[fill,circle,inner sep=3pt]{}
(-2,2)node[fill,circle,inner sep=3pt]{}
(-2,3)node[fill,circle,inner sep=.5pt]{}
(-2,4)node[fill,circle,inner sep=.5pt]{}
(-3,0)node[fill,circle,inner sep=3pt]{}
(-3,1)node[fill,circle,inner sep=3pt]{}
(-3,2)node[fill,circle,inner sep=.5pt]{}
(-3,3)node[fill,circle,inner sep=3pt]{}
(-3,4)node[fill,circle,inner sep=3pt]{}
(-4,0)node[fill,circle,inner sep=.5pt]{}
(-4,1)node[fill,circle,inner sep=.5pt]{}
(-4,2)node[fill,circle,inner sep=3pt]{}
(-4,3)node[fill,circle,inner sep=3pt]{}
(-4,4)node[fill,circle,inner sep=3pt]{}
(-5,0)node[fill,circle,inner sep=.5pt]{}
(-5,1)node[fill,circle,inner sep=3pt]{}
(-5,2)node[fill,circle,inner sep=3pt]{}
(-5,3)node[fill,circle,inner sep=3pt]{}
(-5,4)node[fill,circle,inner sep=3pt]{}
(-6,0)node[fill,circle,inner sep=3pt]{}
(-6,1)node[fill,circle,inner sep=3pt]{}
(-6,2)node[fill,circle,inner sep=3pt]{}
(-6,3)node[fill,circle,inner sep=3pt]{}
(-6,4)node[fill,circle,inner sep=3pt]{}
(-7,0)node[fill,circle,inner sep=3pt]{}
(-7,1)node[fill,circle,inner sep=3pt]{}
(-7,2)node[fill,circle,inner sep=3pt]{}
(-7,3)node[fill,circle,inner sep=3pt]{}
(-7,4)node[fill,circle,inner sep=3pt]{}
;

\draw[very thick,red] plot [smooth,tension=.1] coordinates{(8,3)(7,1)(6,1)(5,0)};
\draw[very thick,red] plot [smooth,tension=.1] coordinates{(7,4)(6,4)(5,3)(4,1)};
\draw[very thick,red] plot [smooth,tension=.1] coordinates{(3.1,4)(2.1,3)(1.1,2)(0.1,1)};
\draw[very thick, blue] plot [smooth,tension=.1] coordinates{(3,4)(2,4)(1,2)(0,1)};
\draw[very thick,red] plot [smooth,tension=.1] coordinates{(0,3)(-1,2)(-2,1)(-3,0)};
\draw[very thick,red] plot [smooth,tension=.1] coordinates{(0.1,4)(-0.9,3)(-1.9,2)(-2.9,1)};
\draw[very thick,blue] plot [smooth,tension=.1] coordinates{(0,4)(-1,4)(-2,2)(-3,1)};
\draw[very thick,red] plot [smooth,tension=.1] coordinates{(-3,3)(-4,2)(-5,1)(-6,0)};
\draw[very thick,red] plot [smooth,tension=.1] coordinates{(-3,4)(-4,3)(-5,2)(-6,1)};
\draw[very thick,red] plot [smooth,tension=.1] coordinates{(-4,4)(-5,3)(-6,2)(-7,0)};

}\quad\dots
$$

All but two pairs of the aft periods and fore periods coincide: $P_k=Q_k$ except for when $k=3,5$. The fore periods are drawn in red, and the aft periods which differ from the fore periods are drawn in blue. We stopped drawing the periods after $P_8$. The fore periods $P_1$, $P_2$, $P_3$, $P_4$, and $P_6$ can travel downstream, and we have:
\begin{align*}
\Upsilon_1^+\cA&=((10,2),(6^2,4,3,2,1^2),(2,1^2),(7,4,2,1^2),(4^3,2^2,1^2))\\
\Upsilon_2^+\cA&=((9,2),(5^3,3,2,1^2),(2,1^2),(6,5,2,1^2),(5^2,2^2,1^2))\\
\Upsilon_3^+\cA&=((9,2),(5^2,4,3^2,1^2),(3,1^2),(6,4,3,1^2),(4^2,3,2,1^2))\\
\Upsilon_4^+\cA&=((9,3),(5^2,4,3,2^2,1),(2^2,1),(6,4,2^2,1),(4^2,2^2,1^2))\\
\Upsilon_6^+\cA&=((9,2,1),(5^2,4,3,2,1^3),(2,1^3),(6,4,2,1^3),(4^2,2^2,1^2))
\end{align*}
and $\Upsilon_k^+\cA=0$ otherwise. The aft periods $Q_1$, $Q_2$, $Q_3$, and $Q_5$ can travel upstream, and we have:
\begin{align*}
\Upsilon_1^-\cA&=((8,2),(4^3,3,2,1^2),(2,1^2),(5,4,2,1^2),(4^2,2^2,1^2))\\
\Upsilon_2^-\cA&=((9,2),(5^2,3^2,2,1^2),(2,1^2),(6,3,2,1^2),(3^2,2^2,1^2))\\
\Upsilon_3^-\cA&=((9,2),(5^2,4,3,1^3),(1^3),(6,4,2,1^2),(4^2,1^4))\\
\Upsilon_5^-\cA&=((9,2),(5^2,4,3,2,1),(2,1),(6,4,2,1^2),(4^2,2^2))
\end{align*}
and $\Upsilon_k^-\cA=0$ otherwise.
\end{example}

\begin{example} We leave it as an exercise for the reader to verify that taking $e=3$, there is no aft period that can travel upstream in the $3$-abacus below, and therefore this is the abacus of a highest weight vertex for the $\slinf$-crystal when $e=3$:
$$\dots\quad\TikZ{[scale=.5]
\draw
(11,0)node[fill,circle,inner sep=.5pt]{}
(11,1)node[fill,circle,inner sep=.5pt]{}
(11,2)node[fill,circle,inner sep=.5pt]{}
(10,0)node[fill,circle,inner sep=.5pt]{}
(10,1)node[fill,circle,inner sep=.5pt]{}
(10,2)node[fill,circle,inner sep=3pt]{}
(9,0)node[fill,circle,inner sep=3pt]{}
(9,1)node[fill,circle,inner sep=.5pt]{}
(9,2)node[fill,circle,inner sep=.5pt]{}
(8,0)node[fill,circle,inner sep=3pt]{}
(8,1)node[fill,circle,inner sep=.5pt]{}
(8,2)node[fill,circle,inner sep=3pt]{}
(7,0)node[fill,circle,inner sep=3pt]{}
(7,1)node[fill,circle,inner sep=3pt]{}
(7,2)node[fill,circle,inner sep=.5pt]{}
(6,0)node[fill,circle,inner sep=3pt]{}
(6,1)node[fill,circle,inner sep=3pt]{}
(6,2)node[fill,circle,inner sep=3pt]{}
(5,0)node[fill,circle,inner sep=3pt]{}
(5,1)node[fill,circle,inner sep=3pt]{}
(5,2)node[fill,circle,inner sep=3pt]{}
(4,0)node[fill,circle,inner sep=3pt]{}
(4,1)node[fill,circle,inner sep=3pt]{}
(4,2)node[fill,circle,inner sep=3pt]{}
(3,0)node[fill,circle,inner sep=3pt]{}
(3,1)node[fill,circle,inner sep=.5pt]{}
(3,2)node[fill,circle,inner sep=.5pt]{}
(2,0)node[fill,circle,inner sep=3pt]{}
(2,1)node[fill,circle,inner sep=3pt]{}
(2,2)node[fill,circle,inner sep=3pt]{}
(1,0)node[fill,circle,inner sep=3pt]{}
(1,1)node[fill,circle,inner sep=3pt]{}
(1,2)node[fill,circle,inner sep=3pt]{}
(0,0)node[fill,circle,inner sep=3pt]{}
(0,1)node[fill,circle,inner sep=3pt]{}
(0,2)node[fill,circle,inner sep=3pt]{};
}\quad\dots$$
\end{example}

Recall from Section \ref{amu} that the $\sle$-crystal and the $\slinf$-crystal commute. This implies that all vertices in a single $\slinf$-crystal component have the same depth in the $\sle$-crystal. When the vertices of the $\slinf$-crystal have depth $0$ in the $\sle$-crystal, equivalently, when the abaci are totally $e$-periodic, the rule for going upstream or downstream in the crystal simplifies:

\begin{corollary}\label{coredge}Let $\cA,\cA'$ be totally $e$-periodic $\ell$-abaci. 
There is an arrow $\cA\to\cA'$ in the $\slinf$-crystal if and only if the following equivalent situations hold for some $k\in\N$:
\begin{enumerate}
\item $\cA'$ is obtained from $\cA$ by shifting the $k$'th period $P_k$ one unit to the right, and the shift $\tilde{P}_k$ of $P_k$ is the $k$'th period of $\cA'$. 
\item $\cA$ is obtained from $\cA'$ by shifting the $k$'th period $P_k$ one unit to the left, and the shift $\tilde{P}_k$ of $P_k$ is the $k$'th period of $\cA$. 
\end{enumerate}
\end{corollary}

\begin{example} 
The following $3$-abacus is totally $4$-periodic.
$$\dots\quad\TikZ{[scale=.5]
\draw
(8,0)node[fill,circle,inner sep=.5pt]{}
(8,1)node[fill,circle,inner sep=.5pt]{}
(8,2)node[fill,circle,inner sep=.5pt]{}
(7,0)node[fill,circle,inner sep=.5pt]{}
(7,1)node[fill,circle,inner sep=.5pt]{}
(7,2)node[fill,circle,inner sep=3pt]{}
(6,0)node[fill,circle,inner sep=3pt]{}
(6,1)node[fill,circle,inner sep=3pt]{}
(6,2)node[fill,circle,inner sep=.5pt]{}
(5,0)node[fill,circle,inner sep=3pt]{}
(5,1)node[fill,circle,inner sep=3pt]{}
(5,2)node[fill,circle,inner sep=.5pt]{}
(4,0)node[fill,circle,inner sep=3pt]{}
(4,1)node[fill,circle,inner sep=3pt]{}
(4,2)node[fill,circle,inner sep=.5pt]{}
(3,0)node[fill,circle,inner sep=3pt]{}
(3,1)node[fill,circle,inner sep=.5pt]{}
(3,2)node[fill,circle,inner sep=3pt]{}
(2,0)node[fill,circle,inner sep=.5pt]{}
(2,1)node[fill,circle,inner sep=.5pt]{}
(2,2)node[fill,circle,inner sep=3pt]{}
(1,0)node[fill,circle,inner sep=.5pt]{}
(1,1)node[fill,circle,inner sep=.5pt]{}
(1,2)node[fill,circle,inner sep=3pt]{}
(0,0)node[fill,circle,inner sep=3pt]{}
(0,1)node[fill,circle,inner sep=.5pt]{}
(0,2)node[fill,circle,inner sep=.5pt]{}
(-1,0)node[fill,circle,inner sep=3pt]{}
(-1,1)node[fill,circle,inner sep=.5pt]{}
(-1,2)node[fill,circle,inner sep=.5pt]{}
(-2,0)node[fill,circle,inner sep=3pt]{}
(-2,1)node[fill,circle,inner sep=.5pt]{}
(-2,2)node[fill,circle,inner sep=3pt]{}
(-3,0)node[fill,circle,inner sep=3pt]{}
(-3,1)node[fill,circle,inner sep=.5pt]{}
(-3,2)node[fill,circle,inner sep=3pt]{}
(-4,0)node[fill,circle,inner sep=3pt]{}
(-4,1)node[fill,circle,inner sep=3pt]{}
(-4,2)node[fill,circle,inner sep=.5pt]{}
(-5,0)node[fill,circle,inner sep=3pt]{}
(-5,1)node[fill,circle,inner sep=3pt]{}
(-5,2)node[fill,circle,inner sep=3pt]{}
(-6,0)node[fill,circle,inner sep=3pt]{}
(-6,1)node[fill,circle,inner sep=3pt]{}
(-6,2)node[fill,circle,inner sep=3pt]{}
(-7,0)node[fill,circle,inner sep=3pt]{}
(-7,1)node[fill,circle,inner sep=3pt]{}
(-7,2)node[fill,circle,inner sep=3pt]{}
(-8,0)node[fill,circle,inner sep=3pt]{}
(-8,1)node[fill,circle,inner sep=3pt]{}
(-8,2)node[fill,circle,inner sep=3pt]{}
(-9,0)node[fill,circle,inner sep=3pt]{}
(-9,1)node[fill,circle,inner sep=3pt]{}
(-9,2)node[fill,circle,inner sep=3pt]{}
;

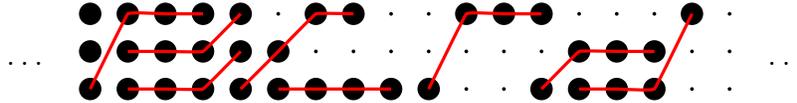
\draw[very thick,red] plot [smooth,tension=.1] coordinates{(7,2)(6,0)(5,0)(4,0)};
\draw[very thick,red] plot [smooth,tension=.1] coordinates{(6,1)(5,1)(4,1)(3,0)};
\draw[very thick,red] plot [smooth,tension=.1] coordinates{(3,2)(2,2)(1,2)(0,0)};
\draw[very thick,red] plot [smooth,tension=.1] coordinates{(-1,0)(-2,0)(-3,0)(-4,0)};
\draw[very thick,red] plot [smooth,tension=.1] coordinates{(-2,2)(-3,2)(-4,1)(-5,0)};
\draw[very thick,red] plot [smooth,tension=.1] coordinates{(-5,1)(-6,0)(-7,0)(-8,0)};
\draw[very thick,red] plot [smooth,tension=.1] coordinates{(-5,2)(-6,1)(-7,1)(-8,1)};
\draw[very thick,red] plot [smooth,tension=.1] coordinates{(-6,2)(-7,2)(-8,2)(-9,0)};
}\quad\dots
$$
Here, only $P_2$ can travel upstream. In fact, the highest weight vertex is obtained by shifting $P_2$ twice to the left, and then $P_1$ twice to the left.
In particular, $\theta=(2^2)$.
\end{example}

The following lemma will be used in Section \ref{Applications}.

\begin{lemma}\label{lemgesperrt}Suppose $Q_k=(b_k^{(1)},\dots,b_k^{(e)})$ is a left-shiftable aft period of an abacus $\cA$. Suppose $P_m=(b_m^{(1)},\dots,b_m^{(e)})$, $m>k$, is a fore period of $\cA$ such that either (1) or (2) holds:
\begin{enumerate}
\item $\beta(b_m^{(1)})=\beta(b_k^{(1)})$,
\item $\beta(b_m^{(1)})=\beta(b_k^{(e)})-1$ and $j(b_m^{(1)})<j(b_k^{(e)})$.
\end{enumerate}
Then $\Upsilon_k^-\cA=0$, i.e., shifting $Q_k$ to the left does not give an edge in the $\slinf$-crystal.
\end{lemma}

\begin{proof}
Let $\widetilde{\cA}$ be the abacus obtained by shifting $Q_k$ to the left, let $\tilde{Q}_k$ be the left shift of $Q_k$ and let $\tilde{b}_k^{(i)}$ be its beads. If $P_1,\dots,P_{k-1}$ are the first $k-1$ fore periods of $\cA$ then they are also the first $k-1$ fore periods of $\widetilde{\cA}$ as shifting $Q_k$ to the left does not affect the quasiperiods which are both disjoint from $Q_k$ and larger than $Q_k$. If (1) holds then necessarily $j(b_m^{(1)})>j(b_k^{(1)})$, so $P=(b_m^{(1)},\tilde{b}_k^{(1)},\tilde{b}_k^{(2)},\dots,\tilde{b}_k^{(e-1)})$ is a larger quasiperiod of $\widetilde{\cA}$ than $\tilde{Q}_k$ whose intersection with $P_1,\dots,P_{k-1}$ is empty, so $\tilde{Q}_k$ cannot be the $k$'th fore period of $\widetilde{\cA}$. Similarly, if (2) holds then the quasiperiod $\{\tilde{b}_k^{(1)},\dots,\tilde{b}_k^{(e-1)}, b_m^{(1)}\}>\tilde{Q}_k$. By Theorem \ref{thmedge}, this implies that $\Upsilon_k^-\cA=0$. 
\end{proof}

\section{Position of an abacus in its $\slinf$-crystal component}\label{depth}

For all $\cA\in\cF_\bs$, there exists a unique highest weight vertex $\cA^\circ\in\cF_\bs$ for $\slinf$
and a unique partition $\theta$ such that $\cA = \tb_\theta \cA^\circ$, and $|\theta|$ is the depth of $\cA$ in the $\slinf$-crystal (see Sections \ref{Scrystal} and \ref{amu}). Pairing this observation with Theorem \ref{thmedge} gives an easy way to read off $\theta$ from $\cA$.
For all $k\in\N$, let $\al_k=\max\{q\in \Z_{\geq 0} \mid (\Upsilon_k^-)^q(\cA)\neq 0\},$ where $(\Upsilon_k^-)^q$ is $\Upsilon_k^-$ applied $q$ times. In abacus terms, $\alpha_k$ is the number of successive times $Q_k\subset\cA$ can travel a unit upstream. Let $r=\max\{n\in\N\mid\alpha_n\neq 0\}$. Thus $r$ is maximal such that $Q_r\subset \cA$ can travel upstream. Set $\de_k = \sum_{m=k}^r \al_m$.
 
\begin{theorem}\label{thmdepth} The partition $\theta$ is given by the formula
$$\theta = (\de_1,\de_2,\dots,\de_r).$$
\end{theorem}

\begin{proof}
By Theorem \ref{thmedge}, there is an edge $\cB\rightarrow\cA$ in the $\slinf$-crystal if and only if $Q_k\subset\cA$ can travel upstream for some $k$ yielding $\cB$ if and only if $\Upsilon_k^-\cA=\cB $.
By Definition \ref{Ups_k}, $(\Upsilon_k^-)^a(\cA)\neq0$ if and only if $\theta$ has $a$ successively removable boxes in row $k$. Thus $\alpha_k$ is the maximum number of successively removable boxes in row $k$, which means that $\theta$ has exactly $\alpha_k$ columns of length $k$. Therefore, 
$\theta^t=(k_s^{\alpha_{k_s}},k_{s-1}^{\alpha_{k_{s-1}}},\dots,k_1^{\alpha_{k_1}})$, where $k_1<k_2<\dots<k_s\in\N$ are all of the $k\in\N$ such that $\alpha_k\neq 0$. It then follows by the formula for the transpose of a partition that $\theta=(\de_1,\de_2,\dots,\de_r)$.
\end{proof}

\begin{example}
Let $\ell=2$, $e=2$, $\mbs=(0,-1)$, and $\bla=((6,4,4,2,2,1),(5,5,5,3,3))$. Then $$\cA(\bla,\bs)=\dots\quad
\TikZ{[scale=.5]
\draw
(8,0)node[fill,circle,inner sep=.5pt]{}
(7,0)node[fill,circle,inner sep=.5pt]{}
(6,0)node[fill,circle,inner sep=3pt]{}
(5,0)node[fill,circle,inner sep=.5pt]{}
(4,0)node[fill,circle,inner sep=.5pt]{}
(3,0)node[fill,circle,inner sep=3pt]{}
(2,0)node[fill,circle,inner sep=3pt]{}
(1,0)node[fill,circle,inner sep=.5pt]{}
(0,0)node[fill,circle,inner sep=.5pt]{}
(-1,0)node[fill,circle,inner sep=3pt]{}
(-2,0)node[fill,circle,inner sep=3pt]{}
(-3,0)node[fill,circle,inner sep=.5pt]{}
(-4,0)node[fill,circle,inner sep=3pt]{}
(-5,0)node[fill,circle,inner sep=.5pt]{}
(-6,0)node[fill,circle,inner sep=3pt]{}
(-7,0)node[fill,circle,inner sep=3pt]{}
(-8,0)node[fill,circle,inner sep=3pt]{}
(-9,0)node[fill,circle,inner sep=3pt]{}
(8,1)node[fill,circle,inner sep=.5pt]{}
(7,1)node[fill,circle,inner sep=.5pt]{}
(6,1)node[fill,circle,inner sep=.5pt]{}
(5,1)node[fill,circle,inner sep=.5pt]{}
(4,1)node[fill,circle,inner sep=3pt]{}
(3,1)node[fill,circle,inner sep=3pt]{}
(2,1)node[fill,circle,inner sep=3pt]{}
(1,1)node[fill,circle,inner sep=.5pt]{}
(0,1)node[fill,circle,inner sep=.5pt]{}
(-1,1)node[fill,circle,inner sep=3pt]{}
(-2,1)node[fill,circle,inner sep=3pt]{}
(-3,1)node[fill,circle,inner sep=.5pt]{}
(-4,1)node[fill,circle,inner sep=.5pt]{}
(-5,1)node[fill,circle,inner sep=.5pt]{}
(-6,1)node[fill,circle,inner sep=3pt]{}
(-7,1)node[fill,circle,inner sep=3pt]{}
(-8,1)node[fill,circle,inner sep=3pt]{}
(-9,1)node[fill,circle,inner sep=3pt]{}
;

\draw[very thick,red] plot [smooth,tension=.1] coordinates{(4,1)(3,0)};
\draw[very thick,red] plot [smooth,tension=.1] coordinates{(3.1,1)(2.1,0)};
\draw[very thick,blue] plot [smooth,tension=.1] coordinates{(3,1)(2,1)};
\draw[very thick,red] plot [smooth,tension=.1] coordinates{(-1,1)(-2,1)};
\draw[very thick,red] plot [smooth,tension=.1] coordinates{(-1,0)(-2,0)};
\draw[very thick,red] plot [smooth,tension=.1] coordinates{(-6,1)(-7,1)};
\draw[very thick,red] plot [smooth,tension=.1] coordinates{(-6,0)(-7,0)};
\draw[very thick,red] plot [smooth,tension=.1] coordinates{(-8,1)(-9,1)};
\draw[very thick,red] plot [smooth,tension=.1] coordinates{(-8,0)(-9,0)};

}\quad\dots
$$
We have $Q_i=P_i$ for all $i\neq 2$; $Q_2$ is the quasiperiod drawn in blue. We have $r=4$, $\alpha_4=2$, $\alpha_2=2$, and $\alpha_k=0$ for all $k\neq 2,4$. The position of $\lambda$ in the $\slinf$-crystal is therefore $\theta=(4,4,2,2)^t=(4,4,2,2)$, and $q(\lambda)=|\theta|=12$.
\end{example}

\begin{example} The abacus in Example \ref{rank4example} has only two aft periods which can travel upstream, $Q_5$ and $Q_1$, and $\alpha_5=1=\alpha_1.$ This abacus thus has position $\theta=(5,1)^t=(2,1^4)$ in its crystal component, and its depth in the $\slinf$-crystal is $6$.
\end{example}

\begin{remark}\label{depthchapeau} 
Even though the $\sle$-crystal structure of the Fock space has been widely studied and is considered well-understood
(\cite{JMMO1991}, \cite{FLOTW1999}, \cite{GeckJacon2011}, \cite{Gerber2015}),
it is a difficult problem to find an explicit combinatorial formula for the depth of an abacus $\cA(\bla,\mbs)=:\cA$ in this crystal.
However, this piece of information is one half of the knowledge of the support of $\el(\bla)\in\oh_c(G(\ell,1,n))$ (see Section \ref{support}). Note that if
$\cA(\overline{\bla},\mbs)$ is the source vertex of the $\sle$-crystal component containing $\cA$, then $ p(\bla) = |\bla|-|\overline{\bla}|$ is the depth of $\cA$ in the $\sle$-crystal. An efficient algorithm for computing $\overline{\bla}$ has been given in \cite[Theorem 6.3]{Gerber2015}. 
This does not require computing the action of the Kashiwara crystal operators,
and relies instead on an affine analogue of the Robinson-Schensted correspondence, denoted $\Phi$. 
Constructing $\Phi$ however requires iterated use of the Schensted insertion procedure on symbols.
We pose the question: is there a way to compute $p(\bla)$ without constructing $\overline{\bla}$, as we have done for $q(\bla)$?
\end{remark}

\section{Charges in $e\Z^\ell$ and a closed formula for doubly highest weight vertices} \label{widthe}

In this section, we consider $\ell$-charges of the form $\bs=(s_1,\dots,s_\ell)$ 
where $s_j-s_{j'}\in e\Z$ for all $j,j'$.
Without loss of generality in virtue of Remarks \ref{remtranslationcharge1} and \ref{remtranslationcharge2},
we can assume that 
$\min\{ s_j \mid 1\leq j \leq \ell \}=0$ and thus $s_j=ez_j$ with $z_j\in\Z_{\geq0}$ for all $j=1\dots,\ell$ and $\min\{z_j\}=0$.
Write $\bz=(z_1,\dots,z_\ell)$, and $\cA^\circ=\cA(\bemp,\bs)$.
We will give a combinatorial formula for the vertices belonging to the connected component of the $\slinf$-crystal with highest weight vertex $\cA^\circ$. 
Such vertices are called \textit{doubly highest weight vertices} in \cite{Gerber2016} as they are singular for the action of $\sle$
and $\slell$ (see \cite{Uglov1999} and \cite{Gerber2016}) simultaneously.

Because $\cA^\circ$ is the abacus of the empty charged multipartition, it is totally periodic, and thus its fore and aft periods coincide by Remark \ref{remforeaft}.
Moreover, because $\bs=e\bz$, each period of $\cA^\circ$ is a sequence of $e$ beads $(b_i)_{i=1,\dots,e}$ in the same row and satisfying $b_1=(\be_1,j_1)$ with $\be_1=0\mod e$.
Therefore, we can construct a (reverse) tabloid $T$ by replacing the $k$'th period of $\cA^\circ$ by the number $k$.
For each $j=1,\dots,\ell$, let $T_j$ be the sequence of numbers in row $j$ of $T$, in increasing order.

For a partition $\si=(\si_1,\si_2,\dots)$ (with infinitely many zero parts), and for 
any increasing integer sequence $X=(x_1,x_2,\dots)$,
write $\si[X,e]=(\si^e_{x_1},\si^e_{x_2},\dots)$.

\begin{theorem}\label{thmwidthe}
Let $\si$ be a partition and $\bs$ be as above.
Then $\tb_\si|\bemp,\bs\rangle = |\bla,\bs\rangle$ where
$$
\la^j = \si[T_j,e]
$$
\end{theorem}

\begin{proof}

Corollary \ref{coredge} implies that 
if $\cA^\circ$ is a highest weight vertex in both the $\sle$- and $\slinf$-crystals, then every other vertex $\tilde{b}_\theta\cA^\circ$ in its $\slinf$-crystal connected component may be obtained from $\cA^\circ$ by first, shifting $P_1$ to the right $\theta_1$ times, then, shifting $P_2$ to the right $\theta_2$ times, etc.\footnote{Note that this recovers \cite[Proposition 7.4]{Gerber2016}. In the even more particular case where $\cA$ is a \textit{doubly highest weight vertex}, i.e. $\cA^\circ$ is the empty $\ell$-partition,
one recovers the original result of \cite[Remark 6.16]{Gerber2016}.} Complete information about the periods of $\cA^\circ$ is given by the tabloid $T$:  
the row of $T$ containing entry $k$ is the row containing the $k$'th period of $\cA$. Making the $k$'th period travel $\si_k$ times downstream is exactly shifting
its $e$ beads, that all belong to the same row, say $j$, by $\si_k$ steps to the right, which corresponds
to adding $\si_k^e$ in the $j$'th component of the $\ell$-partition.
\end{proof}

\begin{example}
Take $e=3$, $\ell=8$ and $\mbz=(2,5,3,0,2,1,1,2)$.
Then the abacus $\cA^\circ$ looks as follows
$$\dots\quad
\TikZ{[scale=.5]
\draw
(-5,0)node[fill,circle,inner sep=3pt]{}
(-4,0)node[fill,circle,inner sep=3pt]{}
(-3,0)node[fill,circle,inner sep=3pt]{}
(-2,0)node[fill,circle,inner sep=3pt]{}
(-1,0)node[fill,circle,inner sep=3pt]{}
(0,0)node[fill,circle,inner sep=3pt]{}
(1,0)node[fill,circle,inner sep=3pt]{}
(2,0)node[fill,circle,inner sep=3pt]{}
(3,0)node[fill,circle,inner sep=3pt]{}
(4,0)node[fill,circle,inner sep=3pt]{}
(5,0)node[fill,circle,inner sep=3pt]{}
(6,0)node[fill,circle,inner sep=3pt]{}
(7,0)node[fill,circle,inner sep=.5pt]{}
(8,0)node[fill,circle,inner sep=.5pt]{}
(9,0)node[fill,circle,inner sep=.5pt]{}
(10,0)node[fill,circle,inner sep=.5pt]{}
(11,0)node[fill,circle,inner sep=.5pt]{}
(12,0)node[fill,circle,inner sep=.5pt]{}
(13,0)node[fill,circle,inner sep=.5pt]{}
(14,0)node[fill,circle,inner sep=.5pt]{}
(15,0)node[fill,circle,inner sep=.5pt]{}
(16,0)node[fill,circle,inner sep=.5pt]{}
(17,0)node[fill,circle,inner sep=.5pt]{}
(18,0)node[fill,circle,inner sep=.5pt]{}

(-5,1)node[fill,circle,inner sep=3pt]{}
(-4,1)node[fill,circle,inner sep=3pt]{}
(-3,1)node[fill,circle,inner sep=3pt]{}
(-2,1)node[fill,circle,inner sep=3pt]{}
(-1,1)node[fill,circle,inner sep=3pt]{}
(0,1)node[fill,circle,inner sep=3pt]{}
(1,1)node[fill,circle,inner sep=3pt]{}
(2,1)node[fill,circle,inner sep=3pt]{}
(3,1)node[fill,circle,inner sep=3pt]{}
(4,1)node[fill,circle,inner sep=3pt]{}
(5,1)node[fill,circle,inner sep=3pt]{}
(6,1)node[fill,circle,inner sep=3pt]{}
(7,1)node[fill,circle,inner sep=3pt]{}
(8,1)node[fill,circle,inner sep=3pt]{}
(9,1)node[fill,circle,inner sep=3pt]{}
(10,1)node[fill,circle,inner sep=3pt]{}
(11,1)node[fill,circle,inner sep=3pt]{}
(12,1)node[fill,circle,inner sep=3pt]{}
(13,1)node[fill,circle,inner sep=3pt]{}
(14,1)node[fill,circle,inner sep=3pt]{}
(15,1)node[fill,circle,inner sep=3pt]{}
(16,1)node[fill,circle,inner sep=.5pt]{}
(17,1)node[fill,circle,inner sep=.5pt]{}
(18,1)node[fill,circle,inner sep=.5pt]{}

(-5,2)node[fill,circle,inner sep=3pt]{}
(-4,2)node[fill,circle,inner sep=3pt]{}
(-3,2)node[fill,circle,inner sep=3pt]{}
(-2,2)node[fill,circle,inner sep=3pt]{}
(-1,2)node[fill,circle,inner sep=3pt]{}
(0,2)node[fill,circle,inner sep=3pt]{}
(1,2)node[fill,circle,inner sep=3pt]{}
(2,2)node[fill,circle,inner sep=3pt]{}
(3,2)node[fill,circle,inner sep=3pt]{}
(4,2)node[fill,circle,inner sep=3pt]{}
(5,2)node[fill,circle,inner sep=3pt]{}
(6,2)node[fill,circle,inner sep=3pt]{}
(7,2)node[fill,circle,inner sep=3pt]{}
(8,2)node[fill,circle,inner sep=3pt]{}
(9,2)node[fill,circle,inner sep=3pt]{}
(10,2)node[fill,circle,inner sep=.5pt]{}
(11,2)node[fill,circle,inner sep=.5pt]{}
(12,2)node[fill,circle,inner sep=.5pt]{}
(13,2)node[fill,circle,inner sep=.5pt]{}
(14,2)node[fill,circle,inner sep=.5pt]{}
(15,2)node[fill,circle,inner sep=.5pt]{}
(16,2)node[fill,circle,inner sep=.5pt]{}
(17,2)node[fill,circle,inner sep=.5pt]{}
(18,2)node[fill,circle,inner sep=.5pt]{}

(-5,3)node[fill,circle,inner sep=3pt]{}
(-4,3)node[fill,circle,inner sep=3pt]{}
(-3,3)node[fill,circle,inner sep=3pt]{}
(-2,3)node[fill,circle,inner sep=3pt]{}
(-1,3)node[fill,circle,inner sep=3pt]{}
(0,3)node[fill,circle,inner sep=3pt]{}
(1,3)node[fill,circle,inner sep=.5pt]{}
(2,3)node[fill,circle,inner sep=.5pt]{}
(3,3)node[fill,circle,inner sep=.5pt]{}
(4,3)node[fill,circle,inner sep=.5pt]{}
(5,3)node[fill,circle,inner sep=.5pt]{}
(6,3)node[fill,circle,inner sep=.5pt]{}
(7,3)node[fill,circle,inner sep=.5pt]{}
(8,3)node[fill,circle,inner sep=.5pt]{}
(9,3)node[fill,circle,inner sep=.5pt]{}
(10,3)node[fill,circle,inner sep=.5pt]{}
(11,3)node[fill,circle,inner sep=.5pt]{}
(12,3)node[fill,circle,inner sep=.5pt]{}
(13,3)node[fill,circle,inner sep=.5pt]{}
(14,3)node[fill,circle,inner sep=.5pt]{}
(15,3)node[fill,circle,inner sep=.5pt]{}
(16,3)node[fill,circle,inner sep=.5pt]{}
(17,3)node[fill,circle,inner sep=.5pt]{}
(18,3)node[fill,circle,inner sep=.5pt]{}

(-5,4)node[fill,circle,inner sep=3pt]{}
(-4,4)node[fill,circle,inner sep=3pt]{}
(-3,4)node[fill,circle,inner sep=3pt]{}
(-2,4)node[fill,circle,inner sep=3pt]{}
(-1,4)node[fill,circle,inner sep=3pt]{}
(0,4)node[fill,circle,inner sep=3pt]{}
(1,4)node[fill,circle,inner sep=3pt]{}
(2,4)node[fill,circle,inner sep=3pt]{}
(3,4)node[fill,circle,inner sep=3pt]{}
(4,4)node[fill,circle,inner sep=3pt]{}
(5,4)node[fill,circle,inner sep=3pt]{}
(6,4)node[fill,circle,inner sep=3pt]{}
(7,4)node[fill,circle,inner sep=.5pt]{}
(8,4)node[fill,circle,inner sep=.5pt]{}
(9,4)node[fill,circle,inner sep=.5pt]{}
(10,4)node[fill,circle,inner sep=.5pt]{}
(11,4)node[fill,circle,inner sep=.5pt]{}
(12,4)node[fill,circle,inner sep=.5pt]{}
(13,4)node[fill,circle,inner sep=.5pt]{}
(14,4)node[fill,circle,inner sep=.5pt]{}
(15,4)node[fill,circle,inner sep=.5pt]{}
(16,4)node[fill,circle,inner sep=.5pt]{}
(17,4)node[fill,circle,inner sep=.5pt]{}
(18,4)node[fill,circle,inner sep=.5pt]{}

(-5,5)node[fill,circle,inner sep=3pt]{}
(-4,5)node[fill,circle,inner sep=3pt]{}
(-3,5)node[fill,circle,inner sep=3pt]{}
(-2,5)node[fill,circle,inner sep=3pt]{}
(-1,5)node[fill,circle,inner sep=3pt]{}
(0,5)node[fill,circle,inner sep=3pt]{}
(1,5)node[fill,circle,inner sep=3pt]{}
(2,5)node[fill,circle,inner sep=3pt]{}
(3,5)node[fill,circle,inner sep=3pt]{}
(4,5)node[fill,circle,inner sep=.5pt]{}
(5,5)node[fill,circle,inner sep=.5pt]{}
(6,5)node[fill,circle,inner sep=.5pt]{}
(7,5)node[fill,circle,inner sep=.5pt]{}
(8,5)node[fill,circle,inner sep=.5pt]{}
(9,5)node[fill,circle,inner sep=.5pt]{}
(10,5)node[fill,circle,inner sep=.5pt]{}
(11,5)node[fill,circle,inner sep=.5pt]{}
(12,5)node[fill,circle,inner sep=.5pt]{}
(13,5)node[fill,circle,inner sep=.5pt]{}
(14,5)node[fill,circle,inner sep=.5pt]{}
(15,5)node[fill,circle,inner sep=.5pt]{}
(16,5)node[fill,circle,inner sep=.5pt]{}
(17,5)node[fill,circle,inner sep=.5pt]{}
(18,5)node[fill,circle,inner sep=.5pt]{}

(-5,6)node[fill,circle,inner sep=3pt]{}
(-4,6)node[fill,circle,inner sep=3pt]{}
(-3,6)node[fill,circle,inner sep=3pt]{}
(-2,6)node[fill,circle,inner sep=3pt]{}
(-1,6)node[fill,circle,inner sep=3pt]{}
(0,6)node[fill,circle,inner sep=3pt]{}
(1,6)node[fill,circle,inner sep=3pt]{}
(2,6)node[fill,circle,inner sep=3pt]{}
(3,6)node[fill,circle,inner sep=3pt]{}
(4,6)node[fill,circle,inner sep=.5pt]{}
(5,6)node[fill,circle,inner sep=.5pt]{}
(6,6)node[fill,circle,inner sep=.5pt]{}
(7,6)node[fill,circle,inner sep=.5pt]{}
(8,6)node[fill,circle,inner sep=.5pt]{}
(9,6)node[fill,circle,inner sep=.5pt]{}
(10,6)node[fill,circle,inner sep=.5pt]{}
(11,6)node[fill,circle,inner sep=.5pt]{}
(12,6)node[fill,circle,inner sep=.5pt]{}
(13,6)node[fill,circle,inner sep=.5pt]{}
(14,6)node[fill,circle,inner sep=.5pt]{}
(15,6)node[fill,circle,inner sep=.5pt]{}
(16,6)node[fill,circle,inner sep=.5pt]{}
(17,6)node[fill,circle,inner sep=.5pt]{}
(18,6)node[fill,circle,inner sep=.5pt]{}

(-5,7)node[fill,circle,inner sep=3pt]{}
(-4,7)node[fill,circle,inner sep=3pt]{}
(-3,7)node[fill,circle,inner sep=3pt]{}
(-2,7)node[fill,circle,inner sep=3pt]{}
(-1,7)node[fill,circle,inner sep=3pt]{}
(0,7)node[fill,circle,inner sep=3pt]{}
(1,7)node[fill,circle,inner sep=3pt]{}
(2,7)node[fill,circle,inner sep=3pt]{}
(3,7)node[fill,circle,inner sep=3pt]{}
(4,7)node[fill,circle,inner sep=3pt]{}
(5,7)node[fill,circle,inner sep=3pt]{}
(6,7)node[fill,circle,inner sep=3pt]{}
(7,7)node[fill,circle,inner sep=.5pt]{}
(8,7)node[fill,circle,inner sep=.5pt]{}
(9,7)node[fill,circle,inner sep=.5pt]{}
(10,7)node[fill,circle,inner sep=.5pt]{}
(11,7)node[fill,circle,inner sep=.5pt]{}
(12,7)node[fill,circle,inner sep=.5pt]{}
(13,7)node[fill,circle,inner sep=.5pt]{}
(14,7)node[fill,circle,inner sep=.5pt]{}
(15,7)node[fill,circle,inner sep=.5pt]{}
(16,7)node[fill,circle,inner sep=.5pt]{}
(17,7)node[fill,circle,inner sep=.5pt]{}
(18,7)node[fill,circle,inner sep=.5pt]{}
;

}\quad\dots
$$
The tabloid $T$ for $\cA^\circ$ is: 
$$
T= \quad
\begin{array}{cccccccc}
\dots &32 & 24 & 16 & 9&  & & 
\\
\dots & 31& 23 & 15 & & & & 
\\
\dots & 30& 22 & 14 & & &  & 
\\
\dots & 29& 21 & 13 & 8 & &  & 
\\
\dots & 28& 20 & &  & &  & 
\\
\dots & 27& 19 & 12 & 7 & 4 &  & 
\\
\dots & 26& 18 & 11 &  6 & 3 & 2 & 1
\\
\dots & 25& 17 & 10 & 5 &  &  & 
\end{array}
$$
from which we may compute $\tilde{b}_\sigma(\cA^\circ)$ for any partition $\sigma$.
E.g. if $\si=(12,9^2,7,6,4,3^2,2^2,1^4)$,
$$\tb_\si|\bemp,\bs\rangle= ( (6^3,2^3) ,  (12^3,9^6,4^3,1^3) , (7^3,3^3,1^3) ,  \emp , (3^3,1^3) , (1^3), \emp, (2^3) ).
$$
\end{example}

In the case where $z_1\geq z_2\geq\dots\geq z_\ell=0$,
we can give a more direct formula for $\tb_\si\cA$.
In this case, we may 
write 
$$\bz=(z_1,\dots,z_\ell)=(y_1^{a_1}, y_2^{a_2},\dots,y_m^{a_m}, 0^{b})$$
where $0<y_i<y_{i-1}$ and where $a_1+a_2+\dots+a_m+b=\ell$. Set $N=\sum_{i=1}^\ell z_i=\sum_{i=1}^m a_iy_i$ and set $d_i=y_i-y_{i+1}$, $i=1,\dots,m$.
If $\lambda=(\la_1,\la_2,\dots)$ is a partition, write $\lambda[e]$ as shorthand for $\la[\N,e]=(\lambda_1^e,\lambda_2^e,\lambda_3^e,\dots)$. 

\begin{corollary} \label{zpartition}
Let $\si$ be a partition and $\bs=e\mbz$ with $z_1\geq z_2\geq\dots\geq z_\ell=0$.
Then $\tb_\si|\bemp,\bs\rangle = |\bla,\bs\rangle$ where

$$
\begin{array}{rcll}
\la^j & 
= & 
\begin{array}[t]{l}
(\sigma_j, \sigma_{a_1+j},\sigma_{2a_1+j},\dots \sigma_{(d_1-1)a_1+j}, \\
\quad \sigma_{d_1a_1+j},\sigma_{d_1a_1+a_2+j},\sigma_{d_1a_1+2a_2+j},\dots,\sigma_{d_1a_1+(d_2-1)a_2+j},\\
\quad \dots,\sigma_{d_1a_1+d_2a_2+\dots+(d_m-1)a_m+j},\sigma_{N+j},\sigma_{N+\ell+j},\sigma_{N+2\ell+j},\sigma_{N+3\ell+j},\dots)[e]
\end{array}\vspace{2mm}
\\ 
&& \text{ for } 1\leq j\leq a_1,
\\ \\
\la^j & 
= & 
\begin{array}[t]{l}
(\sigma_{d_1a_1+j},\sigma_{d_1a_1+a_2+j},\sigma_{d_1a_1+2a_2+j},\dots,\sigma_{d_1a_1+(d_2-1)a_2+j}, \\
\quad\quad \sigma_{d_1a_1+d_2a_2+j},\sigma_{d_1a_1+d_2a_2+a_3+j},\sigma_{d_1a_1+d_2a_2+2a_3+j},\dots,\sigma_{d_1a_1+d_2a_2+(d_3-1)a_3+j}, \\
\quad\quad\dots,\sigma_{d_1a_1+d_2a_2+\dots+(d_m-1)a_m+j},\sigma_{N+j},\sigma_{N+\ell+j},\sigma_{N+2\ell+j},\sigma_{N+3\ell+j},\dots)[e]
\end{array}\vspace{2mm}
\\
&&\text{ for } a_1+1\leq j\leq a_1+ a_2,
\\ \\
\vdots
\\ \\
\la^j & 
= & 
\begin{array}[t]{l}
(\sigma_{\sum_{i=1}^p d_ia_i+j},\sigma_{\sum_{i=1}^p d_ia_i+a_{p+1}+j},\sigma_{\sum_{i=1}^p d_ia_i+2a_{p+1}+j},\dots,\sigma_{\sum_{i=1}^p d_ia_i+(d_{p+1}-1)a_{p+1}+j}, \\
\quad\quad \sigma_{\sum_{i=1}^{p+1} d_ia_i+j},\sigma_{\sum_{i=1}^{p+1} d_ia_i+a_{p+2}+j},\sigma_{\sum_{i=1}^{p+1} d_ia_i+2a_{p+2}+j},\dots,\sigma_{\sum_{i=1}^{p+1} d_ia_i+(d_{p+2}-1)a_{p+2}+j},\\
\quad\quad\dots,\sigma_{d_1a_1+d_2a_2+\dots+(d_m-1)a_m+j},\sigma_{N+j},\sigma_{N+\ell+j},\sigma_{N+2\ell+j},\sigma_{N+3\ell+j},\dots)[e]
\end{array}\vspace{2mm}
\\
&&\text{ for } a_1+a_2+\dots+a_p+1\leq j \leq a_1+a_2+\dots+a_p+a_{p+1},
\\ \\
\vdots
\\ \\
\la^j & 
= & 
(\sigma_{N+j},\sigma_{N+\ell+j},\sigma_{N+2\ell+j},\sigma_{N+3\ell+j},\dots)[e]
\vspace{2mm}
\\
&&\text{ for } \sum_{i=1}^m a_i<j\leq \ell.
\\ \\
\end{array}
$$

\end{corollary}

\begin{proof} 

By Theorem \ref{thmwidthe},
the formula for $|\bla,\bs\rangle$ is obtained by identifying the periods in $\cA$ by looking at the tabloid $T$
and taking the subsequences of $\N$ corresponding to the rows of $T$.
Therefore, it suffices to determine a closed formula for the entries of row $j$ of $T$ for all $j=1,\dots,\ell$.
It is an elementary combinatorial problem to see that the formulas above compute the desired numbers.
\end{proof}

We can now give a combinatorial procedure for computing $\tb_\sigma\cA^\circ$ when $\mbz$ need not satisfy $z_j\geq z_{j+1}$, starting from the situation of Corollary \ref{zpartition}.
Take any $\mbz\in\Z^\ell$ with $\min\{z_j\}=0$, let $T$ be its tabloid, and let $T'$ be the tabloid resulting from switching $z_j$ and $z_{j'}$ for some $j> j'$. Since the order in $T$ goes from bottom to top across the rows, all entries of $T$ located in the rectangle consisting of columns of $T$ which have an entry in row $j$ or $j'$ but not both, and rows from $j'$ to $j$ inclusive, will slide up (if $z_j<z_j'$) or down (if $z_j>z_j'$) to the next available spot in the same column of the diagram of shape $T'$.

\begin{example}
Take $\mbz=(7,7,5,1,0)$ and say we want to change the $2$nd and $4$th entries to get $\mbz'=(7,1,5,7,0)$. 
We replace the diagram of shape $T$ with the diagram of shape $T'$ by switching the second and fourth rows. Then to get the tableau $T'$, all entries located in the rectangle
slide up to the next available row in the new diagram, while all entries outside the rectangle remain unchanged:
\begin{center} 
\begin{tikzpicture}
\matrix [matrix of math nodes] (m)
        {
& \dots & 35 & 30 & 25 & \ & \ & \ & \ & \ & \ & \ \\
& \dots & 34 & 29 & 24 & 20 & {\color{white}16} &\  & \ & \ &  & {\color{white}2} \\
T= \quad  &\dots & 33 & 28 & 23 & 19 & 16 & 13 & 10 & 7 & \ & \ \\
&\dots & 32 & 27 & 22 & 18 & 15 & 12 & 9 & 6 & 4 & 2 \\
& \dots & 31 & 26 & 21 & 17 & 14 & 11 & 8 & 5 & 3 & 1 \\
        };  
        \draw[color=black] (m-2-7.north west) -- (m-2-12.north east) -- (m-4-12.south east) -- (m-4-7.south west) -- (m-2-7.north west);
    \end{tikzpicture}
\end{center}
\begin{center}
\begin{tikzpicture}
       \matrix [matrix of math nodes] (m)
        {
&\dots & 35 & 30 & 25 & \ & \ & \ & \ & \ & \ & \\
&\dots & 34 & 29 & 24 & 20 & 16 & 13 & 10 & 7 & 4 & 2 \\
T'= \quad &\dots & 33 & 28 & 23 & 19 & 15 & 12 & 9 & 6 & \ & \\
&\dots & 32 & 27 & 22 & 18 & {\color{white}16}  & \ & \ & \ &  & {\color{white}2}  \\
&\dots & 31 & 26 & 21 & 17 & 14 & 11 & 8 & 5 & 3 & 1 \\
        };  
        \draw[color=black] (m-2-7.north west) -- (m-2-12.north east) -- (m-4-12.south east) -- (m-4-7.south west) -- (m-2-7.north west);
    \end{tikzpicture}
\end{center}
\end{example}

\section{Applications}\label{Applications}

\subsection{Depth of the trivial representation in the $\sle$- and $\slinf$-crystals }\label{Depth of triv}

Given an abacus $\cA$, we call the \textit{bidepth} of $\cA$ the pair $(q,p)\in\Z_{\geq0}^2$ where
$q$ is the depth of $\cA$ in the $\slinf$-crystal and $p$ is the depth of $\cA$ in the $\sle$-crystal.

\medskip

For $n\in\Z_{\geq0}$, let $\triv$ denote the $\ell$-partition $((1^n),\emp,\emp,\dots,\emp)$. 
This labels the trivial representation of $G(\ell,1,n)$.
The aim of this section is to answer the following 
question: given $(q,p)\in\Z_{\geq0}^2$ and $n\in\Z_{\geq0}$, for which values of the parameters $e, \bs$ 
does the abacus $\cA:=\cA( \triv, \bs )$ have bidepth $(q,p)$?
Take $\mbs=(s_1,s_2,\dots,s_\ell)\in\Z^\ell$ such that $s_1=n-e-1$, without loss of generality by Remarks \ref{remtranslationcharge1} and \ref{remtranslationcharge2} (we identify an abacus with its horizontal shifts). 
Write $n=qe+r$ with $q,r\in\N\cup\{0\}$ and $r<e$.
Set $$m=\min\{r ,\; s_j \mod e\;|\;s_j\geq 0,\; 2\leq j\leq \ell \}.$$ 
Let $(q(\triv), p(\triv))$ denote the bidepth of $\cA( \triv, \bs )$.

\begin{lemma}\label{lem6.1}
The number of free beads in $\cA$ is equal to $p(\triv)$. 
\end{lemma}
\begin{proof} We will induct on the number of free beads. By Theorem \ref{hwvchapeau}, $p(\triv)=0$ if and only if $\cA$ is totally periodic, which is the case if and only if the number of free beads is equal to $0$. The bead $b_n^1$ is the only left-shiftable bead in $\cA$. Then by Theorem \ref{thmchapeaucrystal}, $b_n^1$ is good left-shiftable if and only if $p(\triv)>0$, which holds if and only if $\cA$ has free beads. It is easy to see that if $\cA$ has free beads then $b_n^1$ must be a free bead, and that $b_{n+1}^1$ is never a free bead in the abacus of $((1^n),\emp,\dots,\emp)$ for any charge $\mbs$. Shifting $b_n^1$ to the left then reduces both $p(\triv)$ and the number of free beads by $1$, and now the statement is true by induction.
\end{proof}

\begin{lemma}\label{lem6.2} Suppose $s_j\geq 0$ for some $j\geq 2$. Then the number of free beads in $\cA$ is equal to $m$. 
\end{lemma}
\begin{proof}
By induction on $m$. Suppose $m=0$. If $m=r=0$ then $e$ divides $n$. Then $\cA$ is totally quasiperiodic, so by Lemma \ref{lemqper=>per}, $\cA$ is totally periodic and hence the number of free beads in $\cA$ is equal to $0$. If $m=s_j\mod e=0$ for some $j\geq 2$ and $s_j\geq 0$, then take $j$ such that $s_j$ is minimal with this property. Then in the word of left- and right-shiftable $0$-beads (see Section \ref{chapeaucrystal}), the $-$ for the unique left-shiftable bead $b_n^1=(-e+1,1)$ of $\cA$ cancels with the $+$ for the right-shiftable bead $(s_j,j)$ and therefore $b_n^1$ is not good left-shiftable. It follows by Theorem \ref{thmchapeaucrystal} that $\cA$ is a highest weight vertex for the $\sle$-crystal, so by Theorem \ref{hwvchapeau} $\cA$ is totally periodic, i.e. the number of free beads is $0$. 
For the induction step: if $b_n^1$ is good left-shiftable then shifting it to the left reduces the number of free beads by $1$. Since our formulas use a charge normalized in a way depending on $n$, we must renormalize the charge $\mbs$ for $((1^{n-1}),\emp,\dots,\emp)$ to charge $\mbt$ by subtracting $1$ from every component of $\mbs$, so that $t_1=s_1-1=(n-1)-e-1$ and $t_j=s_j-1$. If $n=eq+r$ with $r>0$ then $n-1=eq+(r-1)=eq+r'$, $e>r':=r-1\geq 0$; and if $s_j>0 \mod e$ for all $s_j\geq 0$, $j\geq 2$, then $s_j-1 \mod e= s_j\mod e -1$. Therefore $$m-1=\min\{r-1,s_j-1\mod e\mid j\geq 2, s_j\geq 0\}=\min\{r', t_j\mod e\mid j\geq 2, t_j\geq 0\}$$ and by induction, the right-hand-side equals the number of free beads in $\cA(((1^{n-1}),\emp,\dots,\emp),\mbt)=\cA(((1^{n-1}),\emp,\dots,\emp),\mbs)$. Hence the number of free beads in $\cA$ is $(m-1)+1=m$.
\end{proof}

\begin{theorem}\label{thmbidepthtriv}
The bidepth of $\cA( \triv, \bs )$ is given by the following formulas: 
\begin{align*}
q(\triv)&=\begin{cases} q\mbox{ if } s_j<0\mbox{ for all }j\geq 2 \\
0\mbox{ if }s_j\geq 0\mbox{ for some }j \geq 2
\end{cases}\\
p(\triv)&=\begin{cases} r\mbox{ if }s_j<0\mbox{ for all }j\geq 2 \\
m \mbox{ if }s_j\geq 0 \mbox{ for some }j\geq 2.
\end{cases}
\end{align*}
\end{theorem}
\begin{proof}
\textit{Proof of the formula for $q(\triv)$.} If $e>n$ then obviously $\triv$ is a highest weight vertex for the $\slinf$-crystal by Theorem \ref{thmedge}, as there is no space for a bead in row $2$ or higher to shift to the left, and any quasiperiod containing $\{b_1^1,\dots,b_1^n\}$ must contain a bead from a higher row since $e>n$. Then $q(\triv)=0=q$. 

For the remainder of the proof of the $q(\triv)$ formula, assume that $e\leq n$. Let $$P=\{b_{n-e+1}^1,\dots,b_n^1\}=\{(0,1),(-1,1),(-2,1),\dots,(-e+1,1)\},$$ a quasiperiod of $\cA$ consisting of the $e$ beads directly to the right of the space in row $1$. Observe that $P$ is the unique left-shiftable quasiperiod of $\cA$. Furthermore, since $(\beta(b_n^1)-2,1)\in\cA$, $(\Upsilon_k^-)^2\cA=0$. It then  follows by Theorems \ref{thmedge} and \ref{thmdepth} that either (a) $q(\triv)=0$, or (b) $P=Q_k$ for some $k\geq 1$ and the left shift $\tilde{P}$ of $P$ is the $k$'th fore period of $(\cA\setminus P)\cup\tilde{P}$, in which case $\theta=(1^k)$ and $q(\triv)=k$.

First, suppose that $s_j<0$ for all $j$. Then the first $q$ fore periods are chains of $e$ successive beads in row $1$ and $P=Q_q$ is the $q$'th aft period. Let $\tilde{Q}_q$ be the left shift of $Q_q$. If $b=(\beta,j)\in\cA$ with $j\geq 2$ then $\beta\leq -1$. The first bead of $\tilde{Q}_q$ is $(\beta(b_{n-e+1}^1)-1,1)=(-1,1)$. Thus $\tilde{Q}_q>Q$ for any quasiperiod $Q$ of $\cA$ whose first bead belongs to a row $j>1$, and so $\tilde{Q}_q$ is the $q$'th fore period of $(\cA\setminus Q_q)\cup \tilde{Q}_q$. Therefore $q(\triv)=q$. 

Next, suppose that $s_j\geq 0$ for some $j\geq 2$. If $P$ is not an aft period then $q(\triv)=0$, so suppose that $P=Q_k$ for some $k$. Either $P$ is also a fore period or it is not. If $P=P_k$ then for all $1\leq \beta \leq e-1$ and all $j\geq 2$, $(\beta,j)$ cannot be the first bead of a fore period $P_a$, because otherwise the bead $(0,1)$ would belong to $P_a$ and not to $P_k$. Since $s_j\geq 0$ for some $j\geq 2$, $(0,j)\in\cA$. Take $j$ minimal with this property. It follows that $(0,j)$ is the first bead of $P_{k+1}$. Then by Lemma \ref{lemgesperrt}, $\Upsilon_k^-\cA=0$, and since $Q_k$ is the only left-shiftable quasiperiod, we conclude that $q(\triv)=0$. 

On the other hand, if $P\neq P_k$ then beads $\{(0,1),(-1,1),\dots,(-\alpha,1)\}$ for some $0\leq \alpha<e-1$ are beads of $P_k$ and beads $\{(-\alpha-1,1),\dots,(-e+1,1)\}$ are free beads. Pick $j\geq 2$ minimal with $s_j\geq 0$. Then $(0,j)$ belongs to some fore period $P_m$. If the first bead $b_m^{(1)}$ of $P_m$ is to the right of the first bead $b$ of $P_k$, i.e. if $\beta(b_m^{(1)})>\beta(b)$, then $k>m$, but then a larger quasiperiod than $P_m$ not intersecting any $P_a$, $a<m$, could be constructed using $(0,1),(-1,1)$, ... instead of $(0,j)$, etc., contradicting the definition of $m$'th fore period. If $\beta(b_m^{(1)})<\beta(b)$ then the bead $(-\alpha-1,1)$ would belong to $P_m$ and would not be free. So $\beta(b_m^{(1)})=\beta(b)$, and thus $\beta(b_m^{(e)})=-\alpha$. Then $\beta(b_{k+1}^{(e)})=-\alpha$ also since $P_m\leq P_{k+1}<P_k$. But then $P$ is not an aft period $Q_k$ as its free beads $b'$ do not satisfy the condition $\beta(b_{k+1}^{(e)})\leq \beta(b')$ in Definition \ref{defperiods}\ref{vessel}. So in this situation, $P\neq Q_k$ and thus $q(\triv)=0$.

\textit{Proof of the formula for $p(\triv)$.} The case $s_j\geq 0$ for some $j\geq2$ follows from Lemmas \ref{lem6.1} and \ref{lem6.2}. Consider the case $s_j<0$ for all $j$. Then the first $q$ fore periods of $\cA$ are $P_1=\{b_1^1,\dots, b_e^1\}$, $P_2=\{b_{1+e}^1,\dots b_{2e}^1\}$,... $P_q=\{b_{1+(q-1)e}^1,\dots b_{qe}^1\}$. If $r=0$, then peeling off $P_1,...,P_q$ off $\cA$ yields $\cA(\bemp,\mbs)$, which by Theorem \ref{hwvchapeau} implies $p(\triv)=0$. If $r>0$, beads $b_n^1=(-e+1,1),\dots, b_{n-r+1}^1=(-e+r,1)$ are the free beads because $(-e,1),(0,j)\notin\cA$ for all $j\geq 2$. By Lemma \ref{lem6.1}, $q(\triv)=r$. 
\end{proof}

\begin{remark}\label{remsupporttriv} 
By the results of Shan, Vasserot, and Losev explained in Sections \ref{shancat} and \ref{amu}, Theorem \ref{thmbidepthtriv}
has a representation theoretic meaning: it gives the support of the irreducible representation $\el(\triv)\in\oh_c(G(\ell,1,n))$,
where $c$ is the parameter for the Cherednik algebra determined by $(e,\mbs)$ as in Sections \ref{RCAparams1} and \ref{RCAparams2}.
Conversely, Theorem \ref{thmbidepthtriv} also permits the description of all $n$ and all parameters $c$ (or equivalently $(e,\bs)$) 
such that $\el(\triv)\in\oh_c(G(\ell,1,n))$ has a given support.
\end{remark}

Consequently, we deduce the set of parameters (without normalizing $\mbs$) such that the spherical representation $\el(\triv)$ is finite-dimensional. Let $n\in\N$.
\begin{corollary}\label{cortrivfd}
The set of parameters $(e,\mbs)\in\Z\times\Z^\ell$ such that $\el(\triv)\in\oh_c(G(\ell,1,n))$ is finite-dimensional consists of all $(e,\mbs)$ satisfying either
\begin{itemize}
 \item[(i)] $s_j-s_1=ke-n+1$ for some $k\in\N$ and some $2\leq j\leq \ell$, or 
 \item[(ii)] the following two conditions hold:
 \begin{itemize}
\item[(a)] $e$ divides $n$, and 
\item[(b)] $s_j-s_1\geq e-n+1$ for some $2\leq j\leq \ell$.
\end{itemize}
\end{itemize}
\end{corollary}
\begin{proof}
As explained in Section \ref{findim}, a simple module $\el(\bla)$ in $\oh_c(G(\ell,1,n)$ is finite-dimensional if and only if $q(\bla)=p(\bla)=0$.
From Theorem \ref{thmbidepthtriv}, we deduce that the module $\el(\triv)$
is finite-dimensional if and only if the following two conditions hold:
\begin{itemize}
 \item[(i')] $s_j-s_1\geq e-n+1$ for some $2\leq j\leq \ell$, and
 \item[(ii')] either 
 \begin{itemize}
\item[(a')] $e$ divides $n$, or
\item[(b')] $s_j-s_1 =-n+1 \mod e$ for some $2\leq j\leq \ell$ such that $s_j-s_1\geq e-n+1$.
\end{itemize}
\end{itemize}
Note that (a) is the same as (a'), (b) is the same as (i'), and that 
(b') is equivalent to (i). Moreover, (b') trivially implies (i'),
so ((i') and (ii')) is equivalent to ((i) or (ii)).
\end{proof}

\begin{remark}\label{remcompotherresults}
The problem of determining the parameters $c$ such that the $H_c(W)$-representation $\el(\triv)$ is finite-dimensional has been studied by a number of authors. We list prior results intersecting the case of $W=G(\ell,1,n)$. The first result in this direction was Berest, Etingof and Ginzburg's theorem for $W=S_n=G(1,1,n)$, which states that $\el(\triv)$ is finite-dimensional if and only if $c=r/n$ for some $r\in\N$ with $\mathrm{gcd}(r,n)=1$ \cite{BEG}. Using methods of geometric representation theory, Varagnolo and Vasserot showed several years later that for $W$ a Weyl group, $\el(\triv)$ is a finite-dimensional $H_c(W)$-module for equal parameters $c$ if and only if $c=r/e$ for $e$ an elliptic number of $W$, $r\in\N$, $\mathrm{gcd}(r,e)=1$ \cite{VV}. Etingof extended Varagnolo and Vasserot's result to unequal parameters and to the more general question of the support of $\el(\triv)$, thus giving a complete answer for $W=B_n=G(2,1,n)$ for any pair of parameters \cite{Etingof2012}. More recently, Griffeth, Gusenbauer, Juteau, and Lanini 
developed a ``degenerate" version of Bezrukavnikov-Etingof parabolic restriction aimed towards studying finite-dimensional representations \cite{GGJL}; this produced a sufficient condition for $\el(\triv)$ to be finite-dimensional for  $W=G(\ell,1,n)$ \cite[Corollary 5.4]{GGJL}. They followed this with the natural question: is the condition of \cite[Corollary 5.4]{GGJL} also necessary for $\el(\triv)$ to be finite-dimensional? \cite[Question 5.5]{GGJL}. 
We now check that our Corollary \ref{cortrivfd} agrees with the answers for $\ell=2$ in \cite{Etingof2012} and for $\ell\geq 2$ in \cite{GGJL}, giving a positive answer to \cite[Question 5.5]{GGJL}. Note that we have restricted ourselves to the case of integral charges for the Fock space and $\kappa=1/e$ instead of $r/e$, but the general case can be reduced to this case, see Section \ref{integralpar}.

In the case of $G(2,1,n)=B_n$, the change of parameters of Section \ref{RCAparams2} directly translates Corollary \ref{cortrivfd} into Etingof's criterion \cite[Section 4.2]{Etingof2012} for the set of parameters $c=(c_1,c_2)$ such that $\el(\triv)$ is finite-dimensional. If $(e,(s_1,s_2))$ is a Fock space charge, then $c_1=1/e$ and $c_2=(s_2-s_1)/e-1/2$. Corollary \ref{cortrivfd} (i) gives $c_2=(ke-n+1)/e-1/2$ for some $k\in\N$, and therefore $$c_1(n-1)+c_2=(n-1)/e+c_2=(2k-1)/2$$ for some $k\in\N$, which are the lines described by Etingof in \cite[Section 4.2]{Etingof2012}. In the case that $e$ divides $n$, Corollary \ref{cortrivfd} (ii) gives additional parameters at which $\el(\triv)$ is finite-dimensional: those not already included in (i) consist of $(s_1,s_2)$ with $s_2-s_1\geq e-n+1$ and $s_2-s_1+(n-1)\neq 0 \mod e$. Write $s_2-s_1=ke-n+1+a$ for $1\leq a\leq e-1$, $k\in\N$. Write $n=er$. Changing to Cherednik parameters, 
\begin{align*}(c_1,c_2)&=(1/e,(s_2-s_1)/e-1/2)=(1/e,(2k-1)/2-n/e+(1+a)/e)\\&=(r/n,(2k-1)/2-r+r(1+a)/n)=(r/n,p/2-r+(rs)/n)
\end{align*}
with $p$ an odd positive integer and $2\leq s\leq e=n/r=n/\mathrm{gcd}(r,n)$. These are exactly the isolated points of \cite[Section 4.2]{Etingof2012} where $\el(\triv)$ is finite-dimensional, when $c_1=1/e$.

In the general case of $G(\ell,1,n)$, we compare our result with the sufficient condition for $\el(\triv)$ to be finite-dimensional given in \cite[Corollary 5.4]{GGJL}. The surface appearance of their criterion is slightly different as they use a different parametrization of the Cherednik algebra. We check that after 
reparametrization, it is consistent with Corollary \ref{cortrivfd}. In our notation, the condition of \cite[Corollary 5.4]{GGJL} is that either
\begin{itemize}
 \item[(i')] $\ell(n-1)/e + d_0-d_j=-j+k\ell$ for some $1\leq j\leq \ell-1$ and some positive integer $k$, or 
 \item[(ii')] $e$ divides $n$ and $\ell(p-1)/e + d_0-d_j=-j+k\ell$ for some $1\leq j\leq \ell-1$, some positive integer $k$, and some $n-e+1\leq p \leq n$
\end{itemize}
where the parameters $d_j$, $0\leq j\leq \ell-1$, are related to the parameters $s_{j+1}$, $0\leq j\leq \ell-1$ via the formula
$$ (1+d_{j-1}-d_j)/\ell = h_j = (s_{j+1}-s_j)/e \text{ \quad for all } 1\leq j\leq \ell-1$$
(where $h_j$ is given in Section \ref{RCAparams2}).
Therefore, we have
\begin{align*}
(i') & \eq  \ell(n-1)/e + d_0-d_j=-j+k\ell \\
& \eq \frac{\ell}{e}(n-1) + \frac{\ell}{e}(s_{j+1}-s_1) -j = -j+k\ell \\
& \eq  s_{j+1}-s_1 = ke -n+1 \\
& \eq  (i).
\end{align*}
Similarly, 
\begin{align*}
(ii') & \eq  \ell(p-1)/e + d_0-d_j=-j+k\ell \\
& \eq  s_{j+1}-s_1 = ke -p+1 \\
& \eq  (ii)
\end{align*}
We see that the condition in \cite[Corollary 5.4]{GGJL} is not only sufficient, but also necessary.
Note finally that this is proved independently in Griffeth and Juteau's latest work \cite{GriffethJuteau2017}.
\end{remark}

\subsection{Depth of a rectangular partition concentrated in one component in the $\slinf$-crystal}
The part of Theorem \ref{Depth of triv} concerning depth in the $\slinf$-crystal may be generalized to multipartitions $\bla$ such that one component is a rectangle and all the other components are empty.

\begin{theorem}\label{thmdepthrectangle} Let $m$ and $n$ be nonnegative integers and let $\bla$ be an $\ell$-partition of $mn$ such that $\lambda^a=(m^n)$ and $\lambda^j=\emp$ for all $j\neq a$. Normalize any charge $\mbs$ so that $s_a=n-e-m$. Write $n=qe+r$ with $q,r\in\Z_{\geq0}$ and $r<e$. 
Set $$t'=\begin{cases} \max\{s_j,s_{j'}+e\;|\; j>a,\;j'<a\} \mbox{ if all } s_j,s_{j'}+e<0\\
0 \mbox{ otherwise }
\end{cases}$$ 
and set $t=\min\{-t',m\}$. Then $q(\bla)=tq$.
\end{theorem}
\begin{proof}
Let $\cA:=\cA(\bla,\mbs)$. As in the proof of Theorem \ref{thmbidepthtriv}, $q(\bla)$ can only be nonzero if 
 $$Q~=~\{(0,a),(-1,a),(-2,a),\dots,(-e+1,a)\}=\{b_{n-e+1}^a,b_{n-e+2}^a,\dots,b_n^a\}$$ is an aft period. When $n<e$ then $Q$ is not an aft period and $q=0$, so the statement $q(\bla)=tq$ is true as both sides are $0$. For the rest of the proof suppose $q>0$. 
 
 First, consider the case that $s_j<0$ for all $j>a$ and $s_{j'}<-e$ for all $j'<a$. Then $(\beta,j)\notin\cA$ for all $\beta\geq 0$, $(\beta',j')\notin\cA$ for all $\beta'\geq -e$. This implies that the first $q$ fore periods of $\cA$ are $P_k:=\{b_{(k-1)e+1}^a,b_{(k-1)e+2}^a,\dots,b_{ke}^a\}$, $k=1,\dots,q$ and that $Q=Q_q$ is the $q$'th aft period. $Q$ can be shifted to the left up to $m$ times, and each left shift gives an edge in the crystal so long as the situations of Lemma \ref{lemgesperrt} do not occur. The number $t$ is by construction the number of times $Q$ can be shifted to the left without the situations of Lemma \ref{lemgesperrt} occurring. Theorem \ref{thmdepth} then implies $\theta=(t^q)$ and hence $q(\bla)=tq$.

Next, suppose $s_j\geq 0$ for some $j>a$. Then by arguments identical to those used in the proof of Theorem \ref{thmbidepthtriv}, $q(\bla)=0$. Suppose $s_{j'}\geq -e$ for some $j'<a$. Definition \ref{defperiods} implies that if $Q$ is an aft period $Q_k$ then $b_n^a$ does not lie above a bead $b$ belonging to $P_m$ for $m>k$. On the other hand, by assumption $(-e,j')\in\cA$ and it belongs to some fore period. Therefore if $Q$ is an aft period, then $(-e,j')$ is the first bead of the fore period that it belongs to.  Then by Lemma \ref{lemgesperrt}, shifting $Q$ to the left does not give an edge in the $\slinf$-crystal.
\end{proof}

\begin{corollary} Suppose that $|\bla,\mbs\rangle$ is as in the theorem and $e$ divides $n$. Then $\el(\bla)$ is finite-dimensional if and only if $s_j\geq 0$ for some $j>a$ or $s_{j'}\geq -e$ for some $j'<a$.
\end{corollary}
\begin{proof} In this case, $\bla$ is totally $e$-quasiperiodic, hence totally $e$-periodic by Lemma \ref{lemqper=>per}, so $\bla$ is a highest weight vertex for the $\sle$-crystal by Theorem \ref{hwvchapeau}. Applying Theorem \ref{thmdepthrectangle} gives the additional condition for $\bla$ to be a highest weight vertex for the $\slinf$-crystal.
\end{proof}

\subsection{Depth of $(\lambda,\emp,\dots,\emp)$ in the $\slinf$-crystal for any partition $\lambda$}
Throughout this subsection, $\lambda=(\lambda_1^{a_1},\lambda_2^{a_2},\lambda_3^{a_3},\dots)$ denotes a partition with $\lambda_i$ being the distinct parts of $\lambda$ and $a_i$ their multiplicities, and we study the position in the $\slinf$-crystal of $\bla:=(\lambda,\emp,\dots,\emp),$ an $\ell$-partition all of whose nonzero parts are in the first component.

\begin{lemma}\label{lemmalambda1} $q(\bla)=0$ if and only if  
either (i) $e>\max\{a_i\}$, or (ii) taking $\lambda_i$ to be maximal so that $a_i\geq e$, position $(\beta_{N_i-e+1}^1,j)$ is occupied by a bead for some $j\geq 2$, where $N_i:=\sum_{m=1}^ia_m$. 
\end{lemma}

\begin{proof}
The depth $q(\bla)=0$ if and only if there is no aft period in row $1$ which can travel upstream in the crystal. The first way for there to be none such is if $a_i<e$ for all $i$, because then there is no uninterrupted string of $e$ beads in the first row with space to its left. The second way is if for any consecutive string of $e$ beads in row $1$ with a space to its left, it can't shift to the left without changing the other periods. As in the proof of Theorem \ref{thmbidepthtriv}, if there are beads above all the beads of such a quasiperiod, either it will fail to be an aft period, or there will be a fore period above it, which means that when it is shifted to the left, the smallest fore period above will change and take its first (shifted) bead. 
\end{proof}
 
The first result we may obtain from the lemma gives the conditions on $\mbs$ for when $\el(\lambda,\emp,\dots,\emp)\in\oh_c(G(\ell,1,n)$ is finite-dimensional, generalizing Theorem \ref{thmbidepthtriv} from the case $\lambda=(1^n)$ to an arbitrary partition of $n$.
Set $N=N_r=\sum_{i=1}^r a_i$, the number of parts of $\lambda$ including multiplicities.
Normalize $\mbs$ so that $s_1=N$: this means the rightmost bead of $\lambda$ representing a part of size $0$ has column position $0$. 
\begin{theorem}\label{firstcompfd} $\bla=(\lambda,\emp,\dots,\emp)$ is a highest weight vertex in both the $\slinf$- and $\sle$-crystals if and only if 
\begin{enumerate}
\item for each distinct nonzero part $\lambda_i$ of $\lambda$, $e\mid a_i$ or there exists $j\geq 2$ such that\\ $s_j=
ke+\lambda_i+\sum_{t=i+1}^r a_t$ for some $k\in\N$,
\item $s_{j'}\geq
\lambda_1+N-a_1+e$ for some $j'\geq 2$.
\end{enumerate}
\end{theorem}
\begin{proof}
The abacus of $|\bla,\mbs\rangle$ is totally periodic if and only if every bead with a space to its left is in a period. If $b$ is in row $1$ and to the right of a space, then $b$ is in a period if and only if it is the last bead of a period. It follows from the proof of Theorem \ref{thmbidepthtriv} that this is the case if and only if (1) holds. Lemma \ref{lemmalambda1} restricted to the situation that (1) holds is (2). 
\end{proof}

Next, we may iterate using the lemma to obtain a complete description of $\theta(\bla)$, the position of $\bla$ in the $\slinf$-crystal, and thus of $q(\bla)=|\theta|$. Let $\lambda_{i_1},\dots,\lambda_{i_s}\subset\{\lambda_1,\dots,\lambda_r\}$ be the distinct parts of $\lambda$ for which $a_i\geq e$, ordered so that $i_u<i_{u+1}$, and set $\delta_{i_u}=\lambda_{i_u}-\lambda_{i_{u+1}}$ for $u=1,\dots,s-1$ and set $\delta_{i_s}=\lambda_{i_s}$. For $i=1,\dots, r,$ write 
$a_i=q_ie+r_i$
with $q_i,r_i\in\Z_{\geq0}$, $r_i<e$.

\begin{theorem}\label{firstcompdepth}Consider the largest $s_j$, $j\geq 2$.\begin{itemize}
\item  If $s_j\geq(\sum_{k=i_1+1}^r a_k)+\lambda_{i_1}+e$ then $\theta(\bla)=\emp$.
\item Let $1< u \leq s$. If $(\sum_{k=i_u+1}^r a_k)+\lambda_{i_u}+e\leq s_j \leq (\sum_{k=i_u}^r a_k)+\lambda_{i_u}+e$ then\\ $\theta(\bla)=((q_{i_1}+q_{i_2}+\dots+q_{i_{u-1}})^{\delta_{i_{u-1}}},(q_{i_1}+q_{i_2}+\dots+q_{i_{u-2}})^{\delta_{i_{u-2}}},\dots, q_{i_1}^{\delta_{i_1}})^t$. 
\item Let $1\leq u\leq s$. If $\lambda_{i_{u+1}}+\sum_{t=i_{u+1}}^r a_t < s_j-e < \lambda_{i_u}+\sum_{t=i_u+1}^ra_t$, 
then\\
$\theta(\bla)=((q_{i_1}+q_{i_2}+\dots+q_{i_{u-1}}+q_{i_u})^b,(q_{i_1}+q_{i_2}+\dots+q_{i_{u-1}})^{\delta_{i_{u-1}}},(q_{i_1}+q_{i_2}+\dots+q_{i_{u-2}})^{\delta_{i_{u-2}}},
\dots, q_{i_1}^{\delta_{i_1}})^t$
where $b=\lambda_{i_u}+\sum_{t=i_u+1}^ra_t-s_j-\#\{\hbox{beads b in row }1\mid s_j-e<\beta(b)<\lambda_{i_u}+\sum_{t=i_u+1}^ra_t\}$
\item If $s_j\leq e$ then $\theta(\bla)=((q_{i_1}+q_{i_2}+\dots+q_{i_s})^{\delta_{i_s}},(q_{i_1}+q_{i_2}+\dots+q_{i_{s-1}})^{\delta_{i_{s-1}}},\dots, q_{i_1}^{\delta_{i_1}})^t$.
\end{itemize}
\end{theorem}

\begin{proof}
The quasiperiod $Q$ consisting of the last $e$ beads of a sequence of $a_{u_i}$ beads corresponding to a part $\lambda_{u_i}$ occurring $a_{u_i}\geq e$ times, can, first of all, move to the left as many times as there is space to do so, which is until it encounters the next part of $\lambda$.  If the next part occurs less than $e$ times then the shift of $Q$ forms a chain of $>e$ beads with those beads and the minimal quasiperiod, the last $e$ beads of that chain, can continue moving left... and this continues until such a quasiperiod encounters the beads corresponding to the next part $\lambda_{i_{u+1}}$ of $\lambda$ occurring $a_{i_{u+1}}\geq e$ times. Thus $Q$ can physically shift (recursively) to the left as many times as there are spaces between the clump of beads corresponding to parts of $\lambda$ of size $\lambda_{u_i}$ and the clump of beads corresponding to parts of size $\lambda_{u_{i+1}}$, so $\delta_{u_i}$ times. This will always be traveling upstream in the crystal unless at some point the first bead of $Q$ passes to the left of a bead in one of the rows above; if this happens then that is where $Q$ stops traveling upstream. Translating these remarks into formulas gives the theorem.
\end{proof}

\begin{example}
Let $\lambda=(12^7,7,6,4^{11})$, $\mbs=(20,s_2,\dots,s_\ell)$, and $e=3$. Let $s_j=\max\{s_2,\dots,s_\ell\}$. We illustrate the computation of $\theta(\bla)$ and $q(\bla)$, drawing rows $1$ and $j$ only since no other row plays a role. There are two distinct parts $\lambda_i$ with $a_i\geq e$, $\lambda_1$ and $\lambda_4$, and so $i_1=1$, $i_2=4$, $q_{i_1}=2,\;q_{i_2}=3$.
\begin{enumerate}
\item  $s_j=28$. The blue bead, the first bead of the $q_{i_1}$'st aft period, cannot move to the left of the red bead, which means it cannot move left at all. Then $\theta(\bla)=\emp$ and $q(\bla)=0$. 
$$\TikZ{[scale=.5]
\draw
(0,0)node[fill,circle,inner sep=3pt]{}
(1,0)node[fill,circle,inner sep=.5pt]{}
(2,0)node[fill,circle,inner sep=.5pt]{}
(3,0)node[fill,circle,inner sep=.5pt]{}
(4,0)node[fill,circle,inner sep=.5pt]{}
(5,0)node[fill,circle,inner sep=3pt]{}
(6,0)node[fill,circle,inner sep=3pt]{}
(7,0)node[fill,circle,inner sep=3pt]{}
(8,0)node[fill,circle,inner sep=3pt]{}
(9,0)node[fill,circle,inner sep=3pt]{}
(10,0)node[fill,circle,inner sep=3pt]{}
(11,0)node[fill,circle,inner sep=3pt]{}
(12,0)node[fill,circle,inner sep=3pt]{}
(13,0)node[fill,circle,inner sep=3pt]{}
(14,0)node[fill,circle,inner sep=3pt]{}
(15,0)node[fill,circle,inner sep=3pt]{}
(16,0)node[fill,circle,inner sep=.5pt]{}
(17,0)node[fill,circle,inner sep=.5pt]{}
(18,0)node[fill,circle,inner sep=3pt]{}
(19,0)node[fill,circle,inner sep=.5pt]{}
(20,0)node[fill,circle,inner sep=3pt]{}
(21,0)node[fill,circle,inner sep=.5pt]{}
(22,0)node[fill,circle,inner sep=.5pt]{}
(23,0)node[fill,circle,inner sep=.5pt]{}
(24,0)node[fill,circle,inner sep=.5pt]{}
(25,0)node[fill,circle,inner sep=.5pt]{}
(26,0)node[fill,circle,inner sep=3pt]{}
(27,0)node[fill,circle,inner sep=3pt]{}
(28,0)node[fill,blue,circle,inner sep=3pt]{}
(29,0)node[fill,circle,inner sep=3pt]{}
(30,0)node[fill,circle,inner sep=3pt]{}
(31,0)node[fill,circle,inner sep=3pt]{}
(32,0)node[fill,circle,inner sep=3pt]{}
(33,0)node[fill,circle,inner sep=.5pt]{}
(0,1)node[fill,circle,inner sep=3pt]{}
(1,1)node[fill,circle,inner sep=3pt]{}
(2,1)node[fill,circle,inner sep=3pt]{}
(3,1)node[fill,circle,inner sep=3pt]{}
(4,1)node[fill,circle,inner sep=3pt]{}
(5,1)node[fill,circle,inner sep=3pt]{}
(6,1)node[fill,circle,inner sep=3pt]{}
(7,1)node[fill,circle,inner sep=3pt]{}
(8,1)node[fill,circle,inner sep=3pt]{}
(9,1)node[fill,circle,inner sep=3pt]{}
(10,1)node[fill,circle,inner sep=3pt]{}
(11,1)node[fill,circle,inner sep=3pt]{}
(12,1)node[fill,circle,inner sep=3pt]{}
(13,1)node[fill,circle,inner sep=3pt]{}
(14,1)node[fill,circle,inner sep=3pt]{}
(15,1)node[fill,circle,inner sep=3pt]{}
(16,1)node[fill,circle,inner sep=3pt]{}
(17,1)node[fill,circle,inner sep=3pt]{}
(18,1)node[fill,circle,inner sep=3pt]{}
(19,1)node[fill,circle,inner sep=3pt]{}
(20,1)node[fill,circle,inner sep=3pt]{}
(21,1)node[fill,circle,inner sep=3pt]{}
(22,1)node[fill,circle,inner sep=3pt]{}
(23,1)node[fill,circle,inner sep=3pt]{}
(24,1)node[fill,circle,inner sep=3pt]{}
(25,1)node[fill,circle,inner sep=3pt]{}
(26,1)node[fill,circle,inner sep=3pt]{}
(27,1)node[fill,circle,inner sep=3pt]{}
(28,1)node[fill,red,circle,inner sep=3pt]{}
(29,1)node[fill,circle,inner sep=.5pt]{}
(30,1)node[fill,circle,inner sep=.5pt]{}
(31,1)node[fill,circle,inner sep=.5pt]{}
(32,1)node[fill,circle,inner sep=.5pt]{}
(33,1)node[fill,circle,inner sep=.5pt]{}
;}
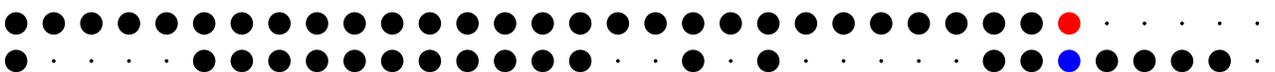
$$
\item $s_j=24$. The blue bead, the first bead of the $q_{i_1}$'st aft period, cannot move to the left of the red bead, which is $4$ units to its left. Since $q_{i_1}=2$ we have $\theta=(2^4)^t=(4^2)$ and $q(\bla)=8$.
$$\TikZ{[scale=.5]
\draw
(0,0)node[fill,circle,inner sep=3pt]{}
(1,0)node[fill,circle,inner sep=.5pt]{}
(2,0)node[fill,circle,inner sep=.5pt]{}
(3,0)node[fill,circle,inner sep=.5pt]{}
(4,0)node[fill,circle,inner sep=.5pt]{}
(5,0)node[fill,circle,inner sep=3pt]{}
(6,0)node[fill,circle,inner sep=3pt]{}
(7,0)node[fill,circle,inner sep=3pt]{}
(8,0)node[fill,circle,inner sep=3pt]{}
(9,0)node[fill,circle,inner sep=3pt]{}
(10,0)node[fill,circle,inner sep=3pt]{}
(11,0)node[fill,circle,inner sep=3pt]{}
(12,0)node[fill,circle,inner sep=3pt]{}
(13,0)node[fill,circle,inner sep=3pt]{}
(14,0)node[fill,circle,inner sep=3pt]{}
(15,0)node[fill,circle,inner sep=3pt]{}
(16,0)node[fill,circle,inner sep=.5pt]{}
(17,0)node[fill,circle,inner sep=.5pt]{}
(18,0)node[fill,circle,inner sep=3pt]{}
(19,0)node[fill,circle,inner sep=.5pt]{}
(20,0)node[fill,circle,inner sep=3pt]{}
(21,0)node[fill,circle,inner sep=.5pt]{}
(22,0)node[fill,circle,inner sep=.5pt]{}
(23,0)node[fill,circle,inner sep=.5pt]{}
(24,0)node[fill,circle,inner sep=.5pt]{}
(25,0)node[fill,circle,inner sep=.5pt]{}
(26,0)node[fill,circle,inner sep=3pt]{}
(27,0)node[fill,circle,inner sep=3pt]{}
(28,0)node[fill,circle,blue,inner sep=3pt]{}
(29,0)node[fill,circle,inner sep=3pt]{}
(30,0)node[fill,circle,inner sep=3pt]{}
(31,0)node[fill,circle,inner sep=3pt]{}
(32,0)node[fill,circle,inner sep=3pt]{}
(33,0)node[fill,circle,inner sep=.5pt]{}
(0,1)node[fill,circle,inner sep=3pt]{}
(1,1)node[fill,circle,inner sep=3pt]{}
(2,1)node[fill,circle,inner sep=3pt]{}
(3,1)node[fill,circle,inner sep=3pt]{}
(4,1)node[fill,circle,inner sep=3pt]{}
(5,1)node[fill,circle,inner sep=3pt]{}
(6,1)node[fill,circle,inner sep=3pt]{}
(7,1)node[fill,circle,inner sep=3pt]{}
(8,1)node[fill,circle,inner sep=3pt]{}
(9,1)node[fill,circle,inner sep=3pt]{}
(10,1)node[fill,circle,inner sep=3pt]{}
(11,1)node[fill,circle,inner sep=3pt]{}
(12,1)node[fill,circle,inner sep=3pt]{}
(13,1)node[fill,circle,inner sep=3pt]{}
(14,1)node[fill,circle,inner sep=3pt]{}
(15,1)node[fill,circle,inner sep=3pt]{}
(16,1)node[fill,circle,inner sep=3pt]{}
(17,1)node[fill,circle,inner sep=3pt]{}
(18,1)node[fill,circle,inner sep=3pt]{}
(19,1)node[fill,circle,inner sep=3pt]{}
(20,1)node[fill,circle,inner sep=3pt]{}
(21,1)node[fill,circle,inner sep=3pt]{}
(22,1)node[fill,circle,inner sep=3pt]{}
(23,1)node[fill,circle,inner sep=3pt]{}
(24,1)node[fill,circle,red,inner sep=3pt]{}
(25,1)node[fill,circle,inner sep=.5pt]{}
(26,1)node[fill,circle,inner sep=.5pt]{}
(27,1)node[fill,circle,inner sep=.5pt]{}
(28,1)node[fill,circle,inner sep=.5pt]{}
(29,1)node[fill,circle,inner sep=.5pt]{}
(30,1)node[fill,circle,inner sep=.5pt]{}
(31,1)node[fill,circle,inner sep=.5pt]{}
(32,1)node[fill,circle,inner sep=.5pt]{}
(33,1)node[fill,circle,inner sep=.5pt]{}
;}
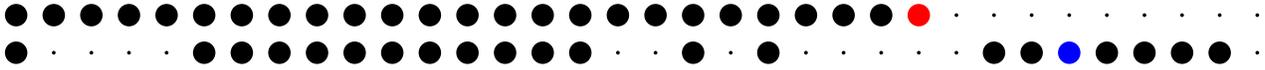
$$
\item $s_j=18$. The blue bead cannot move to the left of the red bead, but it's a moot point because there's no space for it to do so: $Q_2$ can move $8=12-4=\lambda_{i_1}-\lambda_{i_2}$ units left and then runs into the next fore period. So $Q_2$ can travel upstream in the crystal $8$ times and we have $\theta(\bla)=(2^8)^t=(8^2)$ and $q(\bla)=16$.
$$\TikZ{[scale=.5]
\draw
(0,0)node[fill,circle,inner sep=3pt]{}
(1,0)node[fill,circle,inner sep=.5pt]{}
(2,0)node[fill,circle,inner sep=.5pt]{}
(3,0)node[fill,circle,inner sep=.5pt]{}
(4,0)node[fill,circle,inner sep=.5pt]{}
(5,0)node[fill,circle,inner sep=3pt]{}
(6,0)node[fill,circle,inner sep=3pt]{}
(7,0)node[fill,circle,inner sep=3pt]{}
(8,0)node[fill,circle,inner sep=3pt]{}
(9,0)node[fill,circle,inner sep=3pt]{}
(10,0)node[fill,circle,inner sep=3pt]{}
(11,0)node[fill,circle,inner sep=3pt]{}
(12,0)node[fill,circle,inner sep=3pt]{}
(13,0)node[fill,circle,inner sep=3pt]{}
(14,0)node[fill,circle,inner sep=3pt]{}
(15,0)node[fill,circle,inner sep=3pt]{}
(16,0)node[fill,circle,inner sep=.5pt]{}
(17,0)node[fill,circle,inner sep=.5pt]{}
(18,0)node[fill,circle,inner sep=3pt]{}
(19,0)node[fill,circle,inner sep=.5pt]{}
(20,0)node[fill,circle,inner sep=3pt]{}
(21,0)node[fill,circle,inner sep=.5pt]{}
(22,0)node[fill,circle,inner sep=.5pt]{}
(23,0)node[fill,circle,inner sep=.5pt]{}
(24,0)node[fill,circle,inner sep=.5pt]{}
(25,0)node[fill,circle,inner sep=.5pt]{}
(26,0)node[fill,circle,inner sep=3pt]{}
(27,0)node[fill,circle,inner sep=3pt]{}
(28,0)node[fill,circle,blue,inner sep=3pt]{}
(29,0)node[fill,circle,inner sep=3pt]{}
(30,0)node[fill,circle,inner sep=3pt]{}
(31,0)node[fill,circle,inner sep=3pt]{}
(32,0)node[fill,circle,inner sep=3pt]{}
(33,0)node[fill,circle,inner sep=.5pt]{}
(0,1)node[fill,circle,inner sep=3pt]{}
(1,1)node[fill,circle,inner sep=3pt]{}
(2,1)node[fill,circle,inner sep=3pt]{}
(3,1)node[fill,circle,inner sep=3pt]{}
(4,1)node[fill,circle,inner sep=3pt]{}
(5,1)node[fill,circle,inner sep=3pt]{}
(6,1)node[fill,circle,inner sep=3pt]{}
(7,1)node[fill,circle,inner sep=3pt]{}
(8,1)node[fill,circle,inner sep=3pt]{}
(9,1)node[fill,circle,inner sep=3pt]{}
(10,1)node[fill,circle,inner sep=3pt]{}
(11,1)node[fill,circle,inner sep=3pt]{}
(12,1)node[fill,circle,inner sep=3pt]{}
(13,1)node[fill,circle,inner sep=3pt]{}
(14,1)node[fill,circle,inner sep=3pt]{}
(15,1)node[fill,circle,inner sep=3pt]{}
(16,1)node[fill,circle,inner sep=3pt]{}
(17,1)node[fill,circle,inner sep=3pt]{}
(18,1)node[fill,red,circle,inner sep=3pt]{}
(19,1)node[fill,circle,inner sep=.5pt]{}
(20,1)node[fill,circle,inner sep=.5pt]{}
(21,1)node[fill,circle,inner sep=.5pt]{}
(22,1)node[fill,circle,inner sep=.5pt]{}
(23,1)node[fill,circle,inner sep=.5pt]{}
(24,1)node[fill,circle,inner sep=.5pt]{}
(25,1)node[fill,circle,inner sep=.5pt]{}
(26,1)node[fill,circle,inner sep=.5pt]{}
(27,1)node[fill,circle,inner sep=.5pt]{}
(28,1)node[fill,circle,inner sep=.5pt]{}
(29,1)node[fill,circle,inner sep=.5pt]{}
(30,1)node[fill,circle,inner sep=.5pt]{}
(31,1)node[fill,circle,inner sep=.5pt]{}
(32,1)node[fill,circle,inner sep=.5pt]{}
(33,1)node[fill,circle,inner sep=.5pt]{}
;}
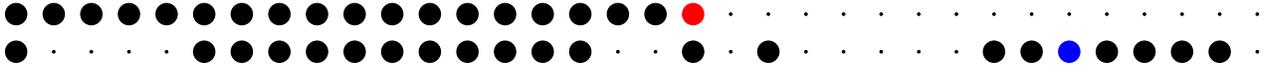
$$
\item $s_j=5$. As in the previous example, $Q_2$ which is marked by the rightmost blue bead can travel $8$ times upstream. The leftmost blue bead marks the first bead of the aft period $Q_{q_{i_1}+q_{i_2}}=Q_5$. It cannot move to the left of the red bead, so $Q_5$ can travel upstream twice in the crystal. Then $\theta(\bla)=(5^2,2^8)^t=(10^2,2^3)$ and $q(\bla)=26$.
$$\TikZ{[scale=.5]
\draw
(0,0)node[fill,circle,inner sep=3pt]{}
(1,0)node[fill,circle,inner sep=.5pt]{}
(2,0)node[fill,circle,inner sep=.5pt]{}
(3,0)node[fill,circle,inner sep=.5pt]{}
(4,0)node[fill,circle,inner sep=.5pt]{}
(5,0)node[fill,circle,inner sep=3pt]{}
(6,0)node[fill,circle,inner sep=3pt]{}
(7,0)node[fill,circle,blue,inner sep=3pt]{}
(8,0)node[fill,circle,inner sep=3pt]{}
(9,0)node[fill,circle,inner sep=3pt]{}
(10,0)node[fill,circle,inner sep=3pt]{}
(11,0)node[fill,circle,inner sep=3pt]{}
(12,0)node[fill,circle,inner sep=3pt]{}
(13,0)node[fill,circle,inner sep=3pt]{}
(14,0)node[fill,circle,inner sep=3pt]{}
(15,0)node[fill,circle,inner sep=3pt]{}
(16,0)node[fill,circle,inner sep=.5pt]{}
(17,0)node[fill,circle,inner sep=.5pt]{}
(18,0)node[fill,circle,inner sep=3pt]{}
(19,0)node[fill,circle,inner sep=.5pt]{}
(20,0)node[fill,circle,inner sep=3pt]{}
(21,0)node[fill,circle,inner sep=.5pt]{}
(22,0)node[fill,circle,inner sep=.5pt]{}
(23,0)node[fill,circle,inner sep=.5pt]{}
(24,0)node[fill,circle,inner sep=.5pt]{}
(25,0)node[fill,circle,inner sep=.5pt]{}
(26,0)node[fill,circle,inner sep=3pt]{}
(27,0)node[fill,circle,inner sep=3pt]{}
(28,0)node[fill,circle,blue,inner sep=3pt]{}
(29,0)node[fill,circle,inner sep=3pt]{}
(30,0)node[fill,circle,inner sep=3pt]{}
(31,0)node[fill,circle,inner sep=3pt]{}
(32,0)node[fill,circle,inner sep=3pt]{}
(33,0)node[fill,circle,inner sep=.5pt]{}
(0,1)node[fill,circle,inner sep=3pt]{}
(1,1)node[fill,circle,inner sep=3pt]{}
(2,1)node[fill,circle,inner sep=3pt]{}
(3,1)node[fill,circle,inner sep=3pt]{}
(4,1)node[fill,circle,inner sep=3pt]{}
(5,1)node[fill,circle,red,inner sep=3pt]{}
(6,1)node[fill,circle,inner sep=.5pt]{}
(7,1)node[fill,circle,inner sep=.5pt]{}
(8,1)node[fill,circle,inner sep=.5pt]{}
(9,1)node[fill,circle,inner sep=.5pt]{}
(10,1)node[fill,circle,inner sep=.5pt]{}
(11,1)node[fill,circle,inner sep=.5pt]{}
(12,1)node[fill,circle,inner sep=.5pt]{}
(13,1)node[fill,circle,inner sep=.5pt]{}
(14,1)node[fill,circle,inner sep=.5pt]{}
(15,1)node[fill,circle,inner sep=.5pt]{}
(16,1)node[fill,circle,inner sep=.5pt]{}
(17,1)node[fill,circle,inner sep=.5pt]{}
(18,1)node[fill,circle,inner sep=.5pt]{}
(19,1)node[fill,circle,inner sep=.5pt]{}
(20,1)node[fill,circle,inner sep=.5pt]{}
(21,1)node[fill,circle,inner sep=.5pt]{}
(22,1)node[fill,circle,inner sep=.5pt]{}
(23,1)node[fill,circle,inner sep=.5pt]{}
(24,1)node[fill,circle,inner sep=.5pt]{}
(25,1)node[fill,circle,inner sep=.5pt]{}
(26,1)node[fill,circle,inner sep=.5pt]{}
(27,1)node[fill,circle,inner sep=.5pt]{}
(28,1)node[fill,circle,inner sep=.5pt]{}
(29,1)node[fill,circle,inner sep=.5pt]{}
(30,1)node[fill,circle,inner sep=.5pt]{}
(31,1)node[fill,circle,inner sep=.5pt]{}
(32,1)node[fill,circle,inner sep=.5pt]{}
(33,1)node[fill,circle,inner sep=.5pt]{}
;}
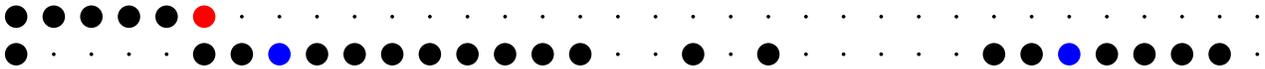
$$

\item $s_j=3$. As in the previous example $Q_2$ can travel $8$ times upstream. The blue bead of $Q_5$ cannot move to the left of the red bead, but this is a moot point as there's no space for it to do so, it runs into the infinite chain of beads starting $4=\delta_{i_2}$ spaces to its left and there it stops. So $Q_5$ can travel upstream $4$ times in the crystal. Then $\theta(\bla)=(5^4,2^8)^t=(12^2,4^3)$ and $q(\bla)=36$.
$$\TikZ{[scale=.5]
\draw
(0,0)node[fill,circle,inner sep=3pt]{}
(1,0)node[fill,circle,inner sep=.5pt]{}
(2,0)node[fill,circle,inner sep=.5pt]{}
(3,0)node[fill,circle,inner sep=.5pt]{}
(4,0)node[fill,circle,inner sep=.5pt]{}
(5,0)node[fill,circle,inner sep=3pt]{}
(6,0)node[fill,circle,inner sep=3pt]{}
(7,0)node[fill,circle,blue,inner sep=3pt]{}
(8,0)node[fill,circle,inner sep=3pt]{}
(9,0)node[fill,circle,inner sep=3pt]{}
(10,0)node[fill,circle,inner sep=3pt]{}
(11,0)node[fill,circle,inner sep=3pt]{}
(12,0)node[fill,circle,inner sep=3pt]{}
(13,0)node[fill,circle,inner sep=3pt]{}
(14,0)node[fill,circle,inner sep=3pt]{}
(15,0)node[fill,circle,inner sep=3pt]{}
(16,0)node[fill,circle,inner sep=.5pt]{}
(17,0)node[fill,circle,inner sep=.5pt]{}
(18,0)node[fill,circle,inner sep=3pt]{}
(19,0)node[fill,circle,inner sep=.5pt]{}
(20,0)node[fill,circle,inner sep=3pt]{}
(21,0)node[fill,circle,inner sep=.5pt]{}
(22,0)node[fill,circle,inner sep=.5pt]{}
(23,0)node[fill,circle,inner sep=.5pt]{}
(24,0)node[fill,circle,inner sep=.5pt]{}
(25,0)node[fill,circle,inner sep=.5pt]{}
(26,0)node[fill,circle,inner sep=3pt]{}
(27,0)node[fill,circle,inner sep=3pt]{}
(28,0)node[fill,circle,blue,inner sep=3pt]{}
(29,0)node[fill,circle,inner sep=3pt]{}
(30,0)node[fill,circle,inner sep=3pt]{}
(31,0)node[fill,circle,inner sep=3pt]{}
(32,0)node[fill,circle,inner sep=3pt]{}
(33,0)node[fill,circle,inner sep=.5pt]{}
(0,1)node[fill,circle,inner sep=3pt]{}
(1,1)node[fill,circle,inner sep=3pt]{}
(2,1)node[fill,circle,inner sep=3pt]{}
(3,1)node[fill,circle,red,inner sep=3pt]{}
(4,1)node[fill,circle,inner sep=.5pt]{}
(5,1)node[fill,circle,inner sep=.5pt]{}
(6,1)node[fill,circle,inner sep=.5pt]{}
(7,1)node[fill,circle,inner sep=.5pt]{}
(8,1)node[fill,circle,inner sep=.5pt]{}
(9,1)node[fill,circle,inner sep=.5pt]{}
(10,1)node[fill,circle,inner sep=.5pt]{}
(11,1)node[fill,circle,inner sep=.5pt]{}
(12,1)node[fill,circle,inner sep=.5pt]{}
(13,1)node[fill,circle,inner sep=.5pt]{}
(14,1)node[fill,circle,inner sep=.5pt]{}
(15,1)node[fill,circle,inner sep=.5pt]{}
(16,1)node[fill,circle,inner sep=.5pt]{}
(17,1)node[fill,circle,inner sep=.5pt]{}
(18,1)node[fill,circle,inner sep=.5pt]{}
(19,1)node[fill,circle,inner sep=.5pt]{}
(20,1)node[fill,circle,inner sep=.5pt]{}
(21,1)node[fill,circle,inner sep=.5pt]{}
(22,1)node[fill,circle,inner sep=.5pt]{}
(23,1)node[fill,circle,inner sep=.5pt]{}
(24,1)node[fill,circle,inner sep=.5pt]{}
(25,1)node[fill,circle,inner sep=.5pt]{}
(26,1)node[fill,circle,inner sep=.5pt]{}
(27,1)node[fill,circle,inner sep=.5pt]{}
(28,1)node[fill,circle,inner sep=.5pt]{}
(29,1)node[fill,circle,inner sep=.5pt]{}
(30,1)node[fill,circle,inner sep=.5pt]{}
(31,1)node[fill,circle,inner sep=.5pt]{}
(32,1)node[fill,circle,inner sep=.5pt]{}
(33,1)node[fill,circle,inner sep=.5pt]{}
;}
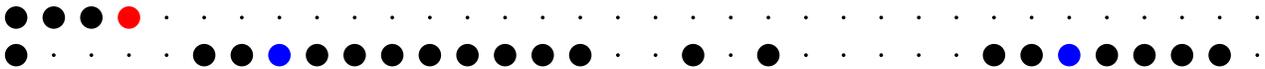
$$
\end{enumerate}
\end{example}

\subsection{A criterion for finite-dimensionality of Cherednik algebra modules in type $B$}
Let $\ell=2$ and take a charged bipartition $|\bla,\mbs\rangle$: $\bla=(\lambda^1,\lambda^2)$, $\mbs=(s_1,s_2)\in\Z^2$. Fix $e\geq 2$. Let $\cA$ be the abacus of  $\bla=(\lambda^1,\lambda^2)$ and let $N\in\Z$ be the $\beta$-number (i.e. column position) of the first bead of the first fore period of $\cA$.

\begin{theorem}\label{depth0bipartition} The level $2$ abacus $\cA$ is a highest weight vertex for the $\slinf$-crystal if and only if $\cA$ avoids the following $e+1$ patterns from column $N$ and to the left:
$$(1)\;\TikZ{[scale=.5]
\draw
(0,1)node[fill,circle,inner sep=.5pt]{}
(0,0)node[fill,circle,inner sep=.5pt]{}
;}\quad
(2)\;\TikZ{[scale=.5]\draw
(1,1)node[fill,circle,inner sep=.5pt]{}
(0,0)node[fill,circle,inner sep=.5pt]{}
(1,0)node[fill,circle,inner sep=2pt]{}
(0,1)node[fill,circle,inner sep=2pt]{}
;}\quad
(3)\;\TikZ{[scale=.5]\draw
(-1,0)node[fill,circle,inner sep=.5pt]{}
(1,1)node[fill,circle,inner sep=.5pt]{}
(0,0)node[fill,circle,inner sep=2pt]{}
(-1,1)node[fill,circle,inner sep=2pt]{}
(1,0)node[fill,circle,inner sep=2pt]{}
(0,1)node[fill,circle,inner sep=2pt]{}
;}\quad
(4)\;\TikZ{[scale=.5]\draw
(-1,0)node[fill,circle,inner sep=.5pt]{}
(2,1)node[fill,circle,inner sep=.5pt]{}
(0,0)node[fill,circle,inner sep=2pt]{}
(-1,1)node[fill,circle,inner sep=2pt]{}
(1,0)node[fill,circle,inner sep=2pt]{}
(0,1)node[fill,circle,inner sep=2pt]{}
(2,0)node[fill,circle,inner sep=2pt]{}
(1,1)node[fill,circle,inner sep=2pt]{}
;}\quad\dots\quad
$$

More formally, the $k+1$'st pattern (k+1) for a given $\beta\leq N$, $0\leq k\leq e$, is that $(\beta,2),(\beta-k,1)\notin\cA$ and $(\beta',j)\in\cA$ for all $\beta-k\leq\beta'\leq \beta$, $j=1,2$, $(\beta',j)\neq(\beta,2),(\beta-k,1)$. The statement is that $\cA$ is a highest weight vertex for the $\slinf$-crystal if and only if patterns (1) through (e+1) do not occur in $\cA$ for every $\beta\leq N$. 
\end{theorem}

\begin{proof}
First we show that if patterns (1) through (e+1) do not occur then $\cA$ is the source of its $\slinf$-crystal component. There are $e+1$ distinct quasiperiods up to shift when $\ell=2$: for $0\leq k \leq e$, the $k$'th type of quasiperiod is $\{(\gamma,2),(\gamma-1,2),\dots,(\gamma-e+k+1,2) ,(\gamma-e+k,1),\dots,(\gamma-e+2,1),(\gamma-e+1,1)\}$ for some $\gamma\in\Z$. If patterns (1) through (k+1) do not occur then the $k$'th type of quasiperiod has no space to move to the left, for all $k=1,...,e-1$. Now consider the two extremal cases of a quasiperiods of types $0$ and $e$ concentrated in rows $2$ and $1$ respectively. For the former: if case (1) does not occur and the quasiperiod $P$  has space to the left to move, then the position one step below and left of the last bead of $P$ is occupied by a bead $b$. The left shift of $P$ would then not form a fore period as such a period would use $b$ instead of the last bead of the shift of $P$. For the latter: if $\cA$ avoids patterns (1) through (e) and a quasiperiod $P$ of type $e$, concentrated in row $1$, has space to move to its left, then either pattern (e+1) occurs or the space directly left of $P$ is free, the space above it is occupied by a bead, and all spaces above the beads of $P$ are occupied by beads. A slightly longer but straightforward argument shows that in the latter case, $P$ cannot travel upstream in the crystal. 

For the reverse implication: assume that $\cA$ is a highest weight vertex for the $\slinf$-crystal. In pattern (e+1), the $e$ beads in row $2$ always form a fore period, and the $e$ beads in row $1$ an aft period which can move upstream in the crystal by shifting left. Thus pattern (e+1) does not occur. Suppose one of the patterns (1) through (e) appears in $\cA$, say pattern (k+1) for some $0\leq k\leq e-1$, with the column position of the upper right space equal to $\beta$. 
As $k<e$, there is no quasiperiod that is a subset of the pattern, and there is no quasiperiod starting to the right of the pattern and ending to the left of the pattern. But we have assumed $\beta\leq N$ where $N$ is the position of the first bead of the first fore period of $\cA$. This implies there must be some quasiperiod $Q\subset\cA$ whose last bead $b^{(e)}$ satisfies $\beta(b^{(e)})\geq \beta-k+1$.
Consider the smallest such $Q$, and say $Q$ is of type $k'$. Set $\beta'=\beta(b^{(e)})$. We have $Q=\{(\beta'+e-1,2),(\beta'+e-2,2),(\beta'+e-3,2)\dots,(\beta'+k',2) ,(\beta'+k'-1,1),\dots,(\beta'+1,1),(\beta',1)\}$. If $(\beta'-1,1)\in\cA$ then $\beta'>\beta-k+1$ and $\{(\beta'+e-2,2),(\beta'+e-3,2)\dots,(\beta'+k',2) ,(\beta'+k'-1,1),\dots,(\beta'+1,1),(\beta',1),(\beta'-1,1)\}\subset\cA$ is a smaller quasiperiod than $Q$ whose last bead $b$ satisfies $\beta(b)\geq \beta-k+1$, contradicting minimality of $Q$. So $(\beta'-1,1)\notin\cA$. If $(\beta'+k'-1,2)\in\cA$, then $\{(\beta'+e-1,2),(\beta'+e-2,2),(\beta'+e-3,2)\dots,(\beta'+k',2) ,(\beta'+k'-1,2),\dots,(\beta'+1,1),(\beta',1)\}\subset\cA$ is a smaller quasiperiod than $Q$, again contradicting minimality of $Q$. So $(\beta'+k'-1,2)\notin\cA$. Since $(\beta'-1,1),(\beta'+k'-1,2)\notin\cA$, $Q$ is a left-shiftable quasiperiod. 

In the pattern (k) it is impossible for the beads $(\gamma,2)$ to belong to the same quasiperiod as the beads $(\gamma',1)$ for all $\beta-k\leq \gamma\leq \beta-1$ and all $\beta-k+1\leq \gamma'\leq \beta$. This implies that $Q$ is the minimal quasiperiod in its vessel, hence is an aft period. Moreover, we claim that the left shift of $Q$ is a fore period: this is clear if $Q$ is not concentrated in a single row. If $Q$ is concentrated in row $1$ then $(\beta'+e-1,2)\notin\cA$ because otherwise $Q$ is not minimal. Then the left shift of $Q$ is a fore period. 
 If $Q$ is concentrated in row $2$, so if $k'=0$, then $(\beta'-1,2)=(\beta'+k'-1,2)$, $(\beta'-1,1)\notin\cA$ and this implies that the left shift of $Q$ is a fore period. So in every possible case, the left shift of $Q$ is a fore period. This contradicts the assumption that $\cA$ is a highest weight vertex in the $\slinf$-crystal. 
\end{proof}

\begin{corollary}\label{2fd}
Suppose $|\bla,\mbs\rangle$ is a charged bipartition, $|\bla|=n$. Then $\el(\bla)$ is a finite-dimensional representation of the rational Cherednik algebra $\mathsf{H}_c(B_n)=H_c(G(2,1,n))$ at the corresponding parameters $c=(1/e,(s_2-s_1)/e-1/2)$ if and only if 
\begin{enumerate}
\item $\cA(\bla,\mbs)$ satisfies the pattern avoidance condition of Theorem \ref{depth0bipartition}, 
\item and additionally, 
\begin{enumerate}
\item if $b=(\beta,1)\in\cA$ and $(\beta-1,1)\notin\cA$, then $b$ is the last bead of a period.
\item if $b=(\beta,2)\in\cA$ and $(\beta-1,2)\notin\cA$, then either $(\beta,1)\notin\cA$ or $(\beta,1)\in\cA$ is the last bead of a period.
\end{enumerate}
\end{enumerate}
\end{corollary}
\begin{proof} $\el(\bla)$ is finite-dimensional if and only if $\cA:=\cA(\bla,\mbs)$ is a source vertex in both the $\slinf$- and the $\sle$-crystals. By Theorem \ref{depth0bipartition}, $\cA$ is a source vertex for the $\slinf$-crystal if and only if $\cA$ avoids the $e+1$ patterns described in the theorem. In order for $\cA$ to also be a source vertex for the $\sle$-crystal, $\cA$ must in addition be totally $e$-periodic. This is the case if and only if any bead directly to the right of a space belongs to an $e$-period. If $b:=(\beta,1)\in\cA$ and $(\beta-1,1)\notin\cA$ 
 then $b$ is in a period if and only if it is the last bead of a period. If $b:=(\beta,2)$ and $(\beta-1,2)\notin\cA$, 
then since we are already assuming $\cA$ is a highest weight vertex for the $\slinf$-crystal, it follows from Theorem \ref{depth0bipartition} that $(\beta-1,1)$, $(\beta-2,1)$,..., $(\beta-e,1)$,$(\beta-e-1,1)\in\cA$. 
Thus $b$ will belong to a period if and only if condition (2)(b) is satisfied.
\end{proof}

\begin{remark} Note that this theorem concerns \textit{any} bipartition $\bla$, and thus detects finite-dimensionality of $\el(\bla)$ for any $\bla\in\Irr\;  B_n$,  not just $\el(\triv)$. This is a new result which can be applied not only to Cherednik algebras of type $B$ but to type $D$ as well, since finite-dimensional representations of $H_{1/e}(D_n)$ are obtained via restriction from finite-dimensional representations of $H_{(1/e,0)}(B_n)$. For example, the main theorem of \cite{Shelley-AbrahamsonSun2016}, which was proved by computing wall-crossings, follows immediately from Corollary \ref{2fd}. Indeed: let $e\in2\N$ and  $\mbs=(0,e/2)$, so that $c=(1/e,0)$ is the corresponding type $B$ Cherednik algebra parameter. Let $\lambda=(\lambda_1^{a_1},\lambda_2^{a_2},\dots)\neq \emp$ be an arbitrary partition with $\lambda_i$ the distinct nonzero parts of $\lambda$ and $a_i$ their multiplicities. The abacus $\cA$ of the charged bipartition $|(\lambda,\lambda),(0,e/2)\rangle$ has the same arrangement of beads in row $2$ as in row $1$, except that the beads in row $2$ are shifted $e/2$ units to the right relative to those in row $1$. This means that $\cA$ has a pair of spaces in positions $(\lambda_1-a_1,1)$ and $(\lambda_1-a_1+e/2,2)$, violating the pattern avoidance condition of Theorem \ref{depth0bipartition}. Therefore $\el(\lambda,\lambda)$ is infinite-dimensional, implying \cite[Theorem 2.1]{Shelley-AbrahamsonSun2016}.
\end{remark}

\begin{remark}
Combining Corollary \ref{2fd} and \cite[Theorem 7.7]{Gerber2016}, we get an easy combinatorial characterization
of charged bipartitions whose level-rank dual is an FLOTW $e$-partition (see e.g. \cite[Definition 5.7.8]{GeckJacon2011} for the definition).
\end{remark}

\textbf{Acknowledgments.} We thank Galyna Dobrovolska, Olivier Dudas, Pavel Etingof, and Peng Shan for helpful communications; Daniel Juteau and Seth Shelley-Abrahamson for asking if we could compute the set of parameters such that $\el(\triv)$ is finite-dimensional; Stephen Griffeth and Daniel Juteau for interesting exchanges regarding \cite{GGJL} and \cite{GriffethJuteau2017}; and the anonymous referee for many helpful remarks. E.N. gratefully acknowledges the support of the Max Planck Institute for Mathematics, Bonn. We would also like to thank the organizers of the conference ``Representation Theory in Samos" held in Greece in July 2016, where we first discussed some of the problems solved in this paper.

\bibliographystyle{plain}

\begin{thebibliography}{10}


{
  

\bibitem{Ariki2007}
Susumu Ariki.
\newblock {Proof of the modular branching rule for cyclotomic Hecke algebras}.
\newblock {\em J. Alg.}, 306:290--300, 2007.


\bibitem{BEG}
Yuri Berest, Pavel Etingof, and Victor Ginzburg.
\newblock {Finite-dimensional representations of rational Cherednik algebras}. 
\newblock {\em Int. Math. Res. Not.} 2003, no. 19, 1053--1088. 

\bibitem{BezrukavnikovEtingof}
Roman Bezrukavnikov and Pavel Etingof. 
\newblock {Parabolic induction and restriction functors for rational Cherednik algebras}. 
\newblock {\em Selecta Math. (N.S.)} 14 (2009), no. 3-4, 397--425. 

\bibitem{ChuangRouquier2008}
Joseph Chuang and Rapha\"{e}l Rouquier.
\newblock{Derived equivalences for symmetric groups and $\mathfrak{sl}_2$-categorification}.
\newblock {\em Ann. of Math. (2)} 167 (2008), no. 1, 245--298. 

\bibitem{DipperMathas}
Richard Dipper and Andrew Mathas.
\newblock{Morita equivalences of Ariki-Koike algebras.}
\newblock{\em Math. Z.} 240 (2002), no. 3, 579--610. 

\bibitem{DVV}
Olivier Dudas, Michela Varagnolo, and \'Eric Vasserot.
\newblock {Categorical actions on unipotent representations I. Finite unitary
  groups}.
\newblock 2015.
\newblock arXiv:1509.03269.

\bibitem{DVV2}
Olivier Dudas, Michela Varagnolo, and \'Eric Vasserot.
\newblock{Categorical actions on unipotent representations of finite classical groups}.
\newblock To appear in {\em Contemporary Mathematics}.

\bibitem{Etingof2012}
Pavel Etingof.
\newblock{Supports of irreducible spherical representations of rational Cherednik algebras of finite Coxeter groups}.
\newblock{\em Adv. Math.} 229 (2012), no. 3, 2042--2054. 

\bibitem{EtingofGinzburg2002}
Pavel Etingof and Victor Ginzburg.
\newblock{Symplectic reflection algebras, Calogero-Moser space, and deformed Harish-Chandra homomorphism.}
\newblock {\em Invent. Math.} 147 (2002), no. 2, 243--348. 


\bibitem{FLOTW1999}
Omar Foda, Bernard Leclerc, Masato Okado, Jean-Yves Thibon, and Trevor Welsh.
\newblock {Branching functions of $A_{n-1}^{(1)}$ and Jantzen-Seitz problem for
  Ariki-Koike algebras}.
\newblock {\em Adv. Math.}, 141:322--365, 1999.

\bibitem{GeckJacon2011}
Meinolf Geck and Nicolas Jacon.
\newblock {\em {Representations of Hecke Algebras at Roots of Unity}}.
\newblock Springer, 2011.

\bibitem{Gerber2015}
Thomas Gerber.
\newblock {Crystal isomorphisms in Fock spaces and Schensted correspondence in
  affine type A}.
\newblock {\em Alg. and Rep. Theory}, 18:1009--1046, 2015.

\bibitem{Gerber2016}
Thomas Gerber.
\newblock {Triple crystal action in Fock spaces}.
\newblock 2016.
\newblock arXiv:1601.00581.

\bibitem{Gerber2016a}
Thomas Gerber.
\newblock {Heisenberg algebra, wedges and crystals}.
\newblock 2016.
\newblock arXiv:1612.08760.


\bibitem{GGOR}
\newblock Victor Ginzburg, Nicolas Guay, Eric Opdam, and Rapha\"{e}l Rouquier.
\newblock {On the category $\mathcal{O}$ for rational Cherednik algebras.}
\newblock {\em Invent. Math.} 154 (2003), no. 3, 617--651. 

\bibitem{GGJL}
\newblock Stephen Griffeth, Armin Gusenbauer, Daniel Juteau and Martina Lanini.
\newblock {Parabolic degeneration of rational Cherednik algebras}.
\newblock Sel. Math. New Ser. (2017). https://doi.org/10.1007/s00029-017-0337-3


\bibitem{GriffethJuteau2017}
\newblock Stephen Griffeth and Daniel Juteau.
\newblock {$W$-exponentials, Schur elements, and the support of the spherical representation of the rational Cherednik algebra}.
\newblock arXiv:1707.08196.

\bibitem{HongKang2002}
Jin Hong and Seok-Jin Kang.
\newblock {\em {Introduction to Quantum Groups and Crystal Bases}}.
\newblock American Mathematical Society, 2002.

\bibitem{JaconLecouvey2012}
Nicolas Jacon and C\'edric Lecouvey.
\newblock {A combinatorial decomposition of higher level Fock spaces}.
\newblock {\em Osaka J. Math.}, 50(4):897--920, 2013.


\bibitem{JaconLecouvey2016}
Nicolas Jacon and C\'edric Lecouvey.
\newblock {Crystal isomorphisms and wall-crossing maps for rational Cherednik
  algebras}.
\newblock {\em Transf. Groups}, 2016.
\newblock doi:10.1007/s00031-016-9402-9.


\bibitem{James1978}
Gordon James.
\newblock {Some combinatorial results involving Young diagrams}.
\newblock {\em Math. Proc. Cambridge Phil. Soc.}, 83(1):1--10, 1978.

\bibitem{JamesKerber1984}
Gordon James and Adalbert Kerber.
\newblock {\em {The Representation theory of the Symmetric Group}}.
\newblock Cambridge University Press, 1984.

\bibitem{JMMO1991}
Michio Jimbo, Kailash~C. Misra, Tetsuji Miwa, and Masato Okado.
\newblock Combinatorics of representations of {$U_q(\widehat{sl(n)})$} at
  {$q=0$}.
\newblock {\em Comm. Math. Phys.}, 136(3):543--566, 1991.

\bibitem{Kashiwara1993}
Masaki Kashiwara.
\newblock {Global crystal bases of quantum groups}.
\newblock {\em Duke Math. J.}, 69:455--485, 1993.


\bibitem{Losev2013}
Ivan Losev.
\newblock {Highest weight $\mathfrak{sl}_2$-categorifications I: crystals}.
\newblock {\em Math. Z.}, 274:1231--1247, 2013.

\bibitem{Losev2015a}
Ivan Losev.
\newblock{Rational Cherednik algebras and categorification}.
\newblock 2015.
\newblock arXiv:1509.08550.

\bibitem{Losev2015}
Ivan Losev.
\newblock {Supports of simple modules in cyclotomic Cherednik categories O}.
\newblock 2015.
\newblock arXiv:1509.00526.

\bibitem{Losev}
Ivan Losev.
\newblock{Proof of Varagnolo and Vasserot conjecture on cyclotomic categories $\mathcal{O}$.}
\newblock{\em Selecta Math. (N.S.)} 22 (2016), no. 2, 631--668. 

\bibitem{LosevSA}
Ivan Losev and Seth Shelley-Abrahamson.
\newblock {On refined filtration by supports for rational Cherednik categories O}.
\newblock 2016
\newblock arXiv:1612.08211.

\bibitem{Rouquier}
Rapha\"{e}l Rouquier.
\newblock{$q$-Schur algebras and complex reflection groups.}
\newblock {\em Mosc. Math. J.} 8 (2008), no. 1, 119--158, 184. 

\bibitem{RSVV}
Rapha\"{e}l Rouquier, Peng Shan, Michela Varagnolo, and \'{E}ric Vasserot.
\newblock{Categorifications and cyclotomic rational double affine Hecke algebras}.
\newblock {\em Invent. Math.} 204 (2016), no. 3, 671--786. 

\bibitem{Sagan2001}
Bruce Sagan.
\newblock {\em {The Symmetric Group. Representations, Combinatorial Algorithms, and Symmetric Functions}}.
\newblock Springer, 2001.

\bibitem{Shan2011}
Peng Shan.
\newblock {Crystals of Fock spaces and cyclotomic rational double affine Hecke
  algebras}.
\newblock {\em Ann. Sci. \'Ec. Norm. Sup\'er.}, 44:147--182, 2011.

\bibitem{ShanVasserot2012}
Peng Shan and \'Eric Vasserot.
\newblock {Heisenberg algebras and rational double affine Hecke algebras}.
\newblock {\em J. Amer. Math. Soc.}, 25:959--1031, 2012.

\bibitem{Shelley-AbrahamsonSun2016}
Seth Shelley-Abrahamson and Alec Sun.
\newblock{Towards a classification of finite-dimensional representations of rational Cherednik algebras of type $D$}.
\newblock arXiv:1612.05090.

\bibitem{Uglov1999}
Denis Uglov.
\newblock {Canonical bases of higher-level $q$-deformed Fock spaces and
  Kazhdan-Lusztig polynomials}.
\newblock {\em Progr. Math.}, 191:249--299, 1999.


\bibitem{VV}
Michela Varagnolo and \'{E}ric Vasserot.
\newblock {Finite-dimensional representations of DAHA and affine Springer fibers: the spherical case.}
\newblock {\em Duke Math. J.} 147 (2009), no. 3, 439--540. 

\bibitem{Wilcox}
Stewart Wilcox.
\newblock{Representations of the rational Cherednik algebra}.
\newblock Thesis (Ph.D.)–Harvard University. 2011. 123 pp. ISBN: 978-1124-72988-6.

}

\end{thebibliography}

\end{document}